\DeclareSymbolFont{AMSb}{U}{msb}{m}{n}
\documentclass[11pt,a4paper,leqno,noamsfonts]{amsart}
   \makeatletter
   \renewcommand\@biblabel[1]{#1.}
   \makeatother
\usepackage[english]{babel}
\usepackage[dvipsnames]{xcolor}
\usepackage{graphicx,pifont,soul,physics} % physics to get bold + italics (\vb*{v})
\usepackage[utopia]{mathdesign}
\usepackage[a4paper,top=3cm,bottom=3cm,
           left=3cm,right=3cm,marginparwidth=60pt]{geometry}
\usepackage[utf8]{inputenc}
\usepackage{braket,caption,comment,mathtools,stmaryrd,ytableau}
\usepackage[usestackEOL]{stackengine}
\usepackage{multirow,booktabs,microtype,relsize}
\usepackage[colorlinks,bookmarks]{hyperref} %
      \hypersetup{colorlinks,%
            citecolor=britishracinggreen,%
            filecolor=black,%
            linkcolor=cobalt,%
            urlcolor=cornellred}
      \numberwithin{equation}{section}
\usepackage[capitalise]{cleveref}
\usepackage[colorinlistoftodos]{todonotes}%textwidth=35mm
\definecolor{antiquewhite}{rgb}{0.98, 0.92, 0.84}
\definecolor{buff}{rgb}{0.94, 0.86, 0.51}
\definecolor{palecopper}{rgb}{0.85, 0.54, 0.4}
\definecolor{fluorescentyellow}{rgb}{0.8, 1.0, 0.0}
\definecolor{bole}{rgb}{0.47, 0.27, 0.23}
\usepackage{amsmath,float}
\usetikzlibrary{decorations.markings}

\definecolor{cornellred}{rgb}{0.7, 0.11, 0.11}
\definecolor{britishracinggreen}{rgb}{0.0, 0.26, 0.15}
\definecolor{cobalt}{rgb}{0.0, 0.28, 0.67}
\DeclareSymbolFont{usualmathcal}{OMS}{cmsy}{m}{n}
\DeclareSymbolFontAlphabet{\mathcal}{usualmathcal}

\newcommand{\bfk}{\mathbf{k}}

\newcommand{\BA}{{\mathbb{A}}}

\newcommand{\BC}{{\mathbb{C}}}

\newcommand{\BN}{{\mathbb{N}}}

\newcommand{\BP}{{\mathbb{P}}}
\newcommand{\BQ}{{\mathbb{Q}}}
\newcommand{\BR}{{\mathbb{R}}}
\newcommand{\BS}{{\mathbb{S}}}

\newcommand{\BZ}{{\mathbb{Z}}}

\newcommand{\FC}{{\mathfrak{C}}}

\newcommand{\FN}{{\mathfrak{N}}}

\newcommand{\ord}{\mathrm{ord}}

\newcommand{\simto}{\,\widetilde{\to}\,}

\newcommand{\MF}{{\mathsf{MF}}}

\DeclareMathOperator{\Hilb}{Hilb}
\DeclareMathOperator{\Eu}{Eu}
\DeclareMathOperator{\im}{image}
\DeclareMathOperator{\Con}{Con}

\DeclareMathOperator{\Quot}{Quot}

\DeclareMathOperator{\red}{red}

\DeclareMathOperator{\vir}{\mathrm{vir}}

\DeclareMathOperator{\GL}{GL}

\DeclareMathOperator{\length}{length}

\newcommand{\into}{\hookrightarrow}
\newcommand{\onto}{\twoheadrightarrow}

\DeclareFontFamily{OT1}{rsfs}{}
\DeclareFontShape{OT1}{rsfs}{n}{it}{<-> rsfs10}{}
\DeclareMathAlphabet{\curly}{OT1}{rsfs}{n}{it}
\renewcommand\hom{\mathscr{H}\kern-0.3em\mathit{om}}

\DeclareMathOperator{\lHom}{\mathscr{H}\kern-0.3em\mathit{om}}
\DeclareMathOperator{\RRlHom}{\mathbf{R}\kern-0.025em\mathscr{H}\kern-0.3em\mathit{om}}

\DeclareMathOperator{\lExt}{{\mathscr{E}\kern-0.2em\mathit{xt}}}

\newcommand\Spec{\operatorname{Spec}}
\newcommand\Supp{\operatorname{Supp}}
\newcommand\Proj{\operatorname{Proj}}
\newcommand\Bl{\operatorname{Bl}}

\newcommand\id{\operatorname{id}}

\newcommand{\HH}{\mathrm{H}}
\newcommand{\OO}{\mathscr O}
\newcommand{\OI}{\mathscr I}
\newcommand{\OK}{\mathscr K}



\usepackage[all]{xy}
\usepackage{tikz}
\usepackage{tikz-cd}
\usepackage{adjustbox}
\usepackage{rotating}
\usepackage{comment}

\usetikzlibrary{matrix,shapes,intersections,arrows,decorations.pathmorphing}
\tikzset{commutative diagrams/arrow style=math font}
\tikzset{commutative diagrams/.cd,
mysymbol/.style={start anchor=center,end anchor=center,draw=none}}

\tikzset{
shift up/.style={
to path={([yshift=#1]\tikztostart.east) -- ([yshift=#1]\tikztotarget.west) \tikztonodes}
}
}

\theoremstyle{definition}

\newtheorem*{lemma*}{Lemma}
\newtheorem*{theorem*}{Theorem}
\newtheorem*{example*}{Example}
\newtheorem*{fact*}{Fact}
\newtheorem*{notation*}{Notation}
\newtheorem*{definition*}{Definition}
\newtheorem*{prop*}{Proposition}
\newtheorem*{remark*}{Remark}
\newtheorem*{corollary*}{Corollary}

\newtheorem*{conventions*}{Conventions}

\newtheorem{definition}{Definition}[section]

\newtheorem{example}[definition]{Example}

\newtheorem{conventions}[definition]{Conventions}

\newtheorem{openproblem}[definition]{Open Problem}
\newtheorem{notation}[definition]{Notation}
\newtheorem{remark}[definition]{Remark}

\newtheoremstyle{thm} % <name> % (ambienti con dimostrazione)
        {3mm}% <Space above>
        {3mm}% <Space below>
        {\slshape}% <Body font> % 
        {0mm}% <Indent amount>
        {\bfseries}% <Theorem head font>
        {.}% <Punctuation after theorem head>
        {1mm}% <Space after theorem head>
        {}% <Theorem head spec (can be left empty, meaning 'normal')> 
\theoremstyle{thm}
\newtheorem{theorem}[definition]{Theorem}
\newtheorem{corollary}[definition]{Corollary}
\newtheorem{lemma}[definition]{Lemma}
\newtheorem{prop}[definition]{Proposition}
\newtheorem{thm}{Theorem}

\newtheoremstyle{ex} % <name> % (ambienti con dimostrazione)
        {3mm}% <Space above>
        {3mm}% <Space below>
        {}% <Body font> % \slshape
        {0mm}% <Indent amount>
        {\scshape}% <Theorem head font>
        {.}% <Punctuation after theorem head>
        {1mm}% <Space after theorem head>
        {}% <Theorem head spec (can be left empty, meaning 'normal')> 
\theoremstyle{ex}

\newtheoremstyle{sol} % <name> % (ambienti con dimostrazione)
        {3mm}% <Space above>
        {3mm}% <Space below>
        {}% <Body font> % 
        {0mm}% <Indent amount>
        {\scshape}% <Theorem head font>
        {.}% <Punctuation after theorem head>
        {1mm}% <Space after theorem head>
        {}% <Theorem head spec (can be left empty, meaning 'normal')> 
\theoremstyle{sol}

\newtheorem*{Acknowledgments*}{Acknowledgments}

\newcommand{\nn}{{\mathfrak{n}}}
\newcommand{\mm}{{\mathfrak{m}}}
\newcommand{\Calt}{{\mathcal{T}}}

\newcommand{\Cald}{{\mathcal{D}}}

\newcommand{\sfrac}[2]{
	{\raise0.8ex\hbox{$#1$} \!\mathord{\left/
			{\vphantom {#1 #2}}\right.\kern-\nulldelimiterspace}
		\!\lower0.8ex\hbox{$#2$}}
}

\DeclareMathOperator{\conv}{Conv}

\DeclareMathOperator{\exc}{Exc}

\DeclareMathOperator{\mult}{mult}

\DeclareMathOperator{\SL}{SL}

\usetikzlibrary{patterns}
\DeclareMathAlphabet\BCal{OMS}{cmsy}{b}{n}

\title[On the Behrend function and the blowup of some fat points]{On the Behrend function \\ and the blowup of some fat points}
\author{Michele Graffeo \and Andrea T. Ricolfi}

\begin{document}
\maketitle
\begin{abstract}
The \emph{Behrend function} of a $\BC$-scheme $X$ is a constructible function $\nu_X\colon X(\BC) \to \BZ$ introduced by Behrend, intrinsic to the scheme structure of $X$. It is a (subtle) invariant of singularities of $X$, playing a prominent role in enumerative geometry. To date, only a handful of general properties of the Behrend function are known. In this paper, we compute it for a large class of \emph{fat points} (schemes supported at a single point). We first observe that, if $X \into \BA^N$ is a fat point, $\nu_X$ is the sum of the multiplicities of the irreducible components of the exceptional divisor $E_{X}\BA^N$ in the blowup $\Bl_{X}\BA^N$. Moreover, we prove that $\nu_X$ can be computed explicitly through the normalisation of $\Bl_{X}\BA^N$. 

The proofs of our explicit formulas for the Behrend function of a fat point in $\BA^2$ rely heavily on toric geometry techniques.
Along the way, we find a formula for the number of irreducible components of $E_{X}\BA^2$, where $X \into \BA^2$ is a fat point such that $\Bl_{X}\BA^2$ is normal.
\end{abstract}

{\hypersetup{linkcolor=black}
\tableofcontents}

%%%%%%%%%%%%%%%%%%%%%%%%%%%%%%%%%%%%%%%%%%%%%
%%%%%%%%%%%%%%%%%%%%%%%%%%%%%%%%%%%%%%%%%%%%%
\section{Introduction}

%%%%%%%%%%%%%%%%%%%%%%%%%%%%%%%%%%%%%%%%%%%%%
\subsection{Overview and goal of the paper}\label{intro-1}
Let $X$ be a scheme of finite type over the field $\BC$. One of the (arguably few) intrinsic geometric objects attached to $X$ is a certain cone stack $\FC_X \to X$, called the \emph{intrinsic normal cone}, constructed by  Behrend--Fantechi in their seminal work \cite{BFinc}. This construction has been a breakthrough in enumerative geometry, for it opened the way to a rigorous definition of a cascade of invariants that have been of central importance in algebraic geometry ever since: Gromov--Witten invariants, Donaldson--Thomas invariants, stable pair invariants to name a few.

A few years after the intrinsic normal cone was born, Behrend proved that any scheme $X$ of finite type over $\BC$ carries a canonical constructible function
\[
\nu_X \colon X(\BC) \to \BZ,
\]
defined (cf.~\Cref{sec:behrend}) as the local Euler obstruction of a canonical cycle $\mathfrak c_X \in Z_\ast(X)$ called the \emph{signed support of the intrinsic normal cone} \cite{Beh}. This function is universally referred to as the \emph{Behrend function}, and it has the following remarkable property: whenever $X$ is proper and carries a perfect symmetric obstruction theory (in the sense of Behrend--Fantechi \cite{BFHilb}), the degree of the virtual fundamental class $[X]^{\vir} \in A_0(X)$ attached to the obstruction theory agrees with the weighted Euler characteristic of $X$,
\[
\chi(X,\nu_X) = \sum_{n \in \BZ} n \chi(\nu_X^{-1}(n)),
\]
the `weight' being precisely the Behrend function \cite{Beh}. This result allowed algebraic geometers to compute a huge number of enumerative invariants, previously inaccessible, attached to moduli spaces $X$ satisfying the assumptions of Behrend's theorem. See e.g.~\cite{CohDT,PT13half} for some background and references related to this subject.

Of course, knowing the precise values of the Behrend function is more refined information than knowing just the weighted Euler characteristic. Unfortunately, the Behrend function is quite elusive. We refer the reader to the original papers \cite{Beh,BFHilb} for its main properties, some of which are recalled in \Cref{sec:behrend}. First of all, it is an \emph{invariant of singularities}, in the sense that it pulls back along \'etale maps; in particular, it is sensitive to the scheme structure. It satisfies $\nu_X(p) = (-1)^{\dim_p X}$ if $p$ is a smooth point of $X$. When $X$ is a critical locus, i.e.~a scheme of the form $V(\dd f) \subset U$, for some regular function $f$ on a smooth scheme $U$, the function $\nu_X$ agrees with the \emph{Milnor function} of $(U,f)$. See \Cref{ex:milnor} for more details. Not much more is known about the Behrend function in general. See \Cref{open-problem-DT} for a hard open problem in Donaldson--Thomas theory, related to the Behrend function and also to the geometry of Quot schemes.

\smallbreak
The purpose of this paper is to develop an effective technique to compute the Behrend function of a large class of \emph{fat points}, namely $0$-dimensional schemes $X$ that topologically consist of just one point. We mostly focus on the case where $X$ sits inside the affine plane $\BA^2$, i.e.~it has embedding dimension $1$ or $2$. Such planar situation, arguably the easiest one, already presents (interesting) technical difficulties, confirming that the Behrend function is a subtle invariant of a scheme, no matter how many points the scheme has! For instance, it is easy to observe that
\[
\nu_X = \length(X)
\]
in many \emph{special} cases (see \Cref{sec:first-examples-nu} for several examples), but in general the length of the fat point, namely the number $\length(X)=\chi(\OO_X)$, is not equal to $\nu_X$, and is neither a lower bound nor an upper bound for $\nu_X$. As an example (cf.~\Cref{powermax} for more details), consider the ideal $\mm = (x,y) \subset \BC[x,y]$. Then, if $X_d = \Spec \BC[x,y]/\mm^d$ for $d>1$, one has
\[
\nu_{X_d} = d < \frac{d(d+1)}{2} = \length(X_d).
\]
We anticipate some of the main results proved in this paper in \Cref{sec:main-results}. 
Our main technique is a fine analysis of the multiplicities of the components of the exceptional divisor of the blowup $\Bl_X\BA^N$, where $X \into \BA^N$ is a given fat point. These multiplicities add up to $\nu_X$ by \Cref{lemma:multiplicities_blowup}.

%%%%%%%%%%%%%%%%%%%%%%%%%%%%%%%%%%%%%%%%%%%%%
\subsection{Main results}\label{sec:main-results}
Throughout the paper, $\mm = (x,y) \subset \BC[x,y]$ denotes the ideal of the origin $0 \in \BA^2$. All ideals in $\BC[x,y]$ are assumed to be of finite colength, where the \emph{colength} of an ideal $I \subset \BC[x,y]$ is defined to be the $\BC$-vector space dimension of the quotient, namely $\dim_{\BC}\BC[x,y]/I$.

The following result is obtained combining toric geometry techniques with a deep analysis of the blowups of $\BA^2$ along the given ideals.

\begin{thm}[Theorems \ref{TEOREMA1} and \ref{thm:general-tower-nu}]\label{thm:intro1}
Let $1\leq i_1<\cdots < i_s$ be a strictly increasing sequence of positive integers. Consider the ideal $K = \prod_{1\leq k\leq s} (x+f(y)) + \mm^{i_k} \subset \BC[x,y]$, where $f(y) \in \BC[y]$ has degree smaller that $i_s$. Then $K$ has colength $\sum_{1\leq k\leq s}\sum_{1\leq j\leq k}i_j$, and its Behrend number is 
\[
\nu_{\BC[x,y]/K}=\sum_{k=1}^s\sum_{j=1}^ki_j+\underset{j=1}{\overset{s-1}{\sum}}i_j(s-j).
\]
In particular, when $i_k = k$ for all $k = 1,\ldots,s$, the ideal $K_s = \prod_{1\leq k\leq s}(x+f(y)) + \mm^k$ has colength $\binom{s+2}{3}$, and its Behrend number is
\[
\nu_{\BC[x,y]/K_s} = \frac{s(s+1)(2s+1)}{6}.
\]
\end{thm}

Ideals as in the statement of \Cref{thm:intro1} are called \emph{towers} in our paper (cf.~\Cref{def:tower}), and they are called \emph{complete towers} when $i_k=k$ for all $k$. In \Cref{TORIC PROD TOWERS} we give a formula for the Behrend number of a product of two complete towers; in \Cref{sec:Algorithm} we examine the case of arbitrary finite products of towers, and we present an algorithm to compute the Behrend number also in this case.

\Cref{thm:intro1} covers a large class of ideals, including \emph{some} monomial ideals. We now present a few more explicit formulas. %To deal with a general monomial ideal $I \subset \BC[x,y]$ of finite colength, we proceed as follows. 

If $I \subset \BC[x,y]$ is \emph{normal} (which means that $\Bl_I\BA^2$ is normal, cf.~\Cref{sec:normal-blowups}), then, by \Cref{normal-factorisation}, it factors uniquely as a product 
\begin{equation}\label{eqn:factorisation}
I = \prod_{k=1}^t \mathfrak n_{\alpha_k,\beta_k}^{\delta_k},
\end{equation}
where $\mathfrak n_{\alpha,\beta}\subset \BC[x,y]$ denotes the normalisation of the ideal $(x^\alpha,y^\beta)$ and $\gcd(\alpha_k,\beta_k)=1$ for all $k=1,\ldots,t$. Thanks to the following explicit result, we know the Behrend number of $\mathfrak n_{\alpha,\beta}$.

\begin{thm}[\Cref{Behrenormalisation}]\label{thm:intro2}
Let $\alpha,\beta>0$ be two positive integers, and let $\mathfrak n_{\alpha,\beta}\subset \BC[x,y]$ be the normalisation of the ideal $(x^\alpha,y^\beta)$. Then,
\[
\nu_{\BC[x,y]/\nn_{\alpha,\beta}}=\frac{\alpha\cdot\beta}{\gcd(\alpha,\beta)}.
\]
\end{thm}

One can describe thoroughly the blowup $\Bl_{I}\BA^2$ along a normal monomial ideal as in \Cref{eqn:factorisation} via toric geometry (cf.~\Cref{cor:madame_bovary}), and this allows one to generalise the identity in \Cref{thm:intro2} to cover all normal monomial ideals in $\BC[x,y]$.

In fact, the existence of the factorisation \eqref{eqn:factorisation} readily implies the following statement.

\begin{thm}[\Cref{bijection-irr-cpt-factorisation}]\label{thm:intro1241}
Let $I \subset \BC[x,y]$ be a normal monomial ideal of finite colength. There is a bijective correspondence 
\[
\Set{
\begin{array}{c}
 \mbox{ideals }\nn_{\alpha,\beta}^\delta \mbox{ appearing in the}\\
  \mbox{factorisation \eqref{eqn:factorisation} of $I$}
\end{array}
}\xleftrightarrow{1:1}\left\{
\begin{array}{c}
\mbox{irreducible}\\
  \mbox{components of }E_I\BA^2
\end{array}\right\},
\]
where $E_I\BA^2$ is the exceptional divisor in the blowup of $\BA^2$ with center the ideal $I$. In particular, if $J \subset \BC[x,y]$ is an arbitrary monomial ideal and $I = \overline J$ is its normalisation, then $E_J\BA^2$ has at most $t$ irreducible components, where $t$ is as in \Cref{eqn:factorisation}.
\end{thm}

The Behrend number of a non-normal monomial ideal in $\BC[x,y]$ can be computed from some explicit data defined on the normalisation of $\Bl_I \BA^2$. In fact, in \Cref{sec:monomial} we prove a general statement which is true in all dimensions, not just in dimension $2$. We consider an \emph{arbitrary} fat point $I \subset \BC[x_1,\ldots,x_N]$, and the normalisation morphism
\[
\mu_I \colon Z_I \to \Bl_I \BA^N.
\]
We let $\set{D_i|1\leq i\leq s}$ be the irreducible components of the exceptional divisor $E_I\BA^N \subset \Bl_I\BA^N$, we set $Y_I = \mu_I^{-1}(E_I\BA^N)$ and for each $i=1,\ldots,s$ we let $\set{V_j^{(i)}|1\leq j\leq k_i}$ be the irreducible components of $Y_I$ dominating $D_i$. We then consider the two numbers 
\[
d_{ij} = \deg\left(\mu_I\big|_{V_j^{(i)}}\colon V_j^{(i)} \to D_i\right), \qquad e_{ij} = \mult_{V_j^{(i)}}(Y_I).
\]
We obtain the following result.

\begin{thm}[\Cref{thm:formula:monomial}]\label{thm:intro3}
    Let $X \into \BA^N$ be a fat point defined by an ideal $I \subset \BC[x_1,\ldots,x_N]$. Then, there is an identity
    \[
    \nu_{X} = \sum_{i=1}^s \sum_{j=1}^{k_i} d_{ij} e_{ij}.
    \]
\end{thm}

In \Cref{sec:3-fold-difficulties} we argue, via an explicit example, that the toric techniques used in this paper are not directly applicable to handle fat points in higher dimensional affine spaces $\BA^N$. For instance, \Cref{thm:intro1241} fails. However, \Cref{thm:intro3} is true in all dimensions.

%%%%%%%%%%%%%%%%%%%%%%%%%%%%%%%%%%%%%%%%%%%%%
\subsection{The Behrend function in DT theory}
As mentioned in \Cref{intro-1}, the Behrend function has a crucial role in those enumerative theories where the moduli spaces involved carry a symmetric obstruction theory. This is for instance the case of Donaldson--Thomas theory (DT theory, for short), an enumerative theory designed to `count' sheaves on smooth $3$-folds \cite{ThomasThesis}. If $X$ is a moduli space of stable sheaves on a projective Calabi--Yau $3$-fold, then the expected dimension of $X$ is $0$. When $X$ really has dimension $0$, it is equal to a disjoint union $X_1\amalg\cdots \amalg X_e$ where each $X_i$ is a fat point, the main object of study in this paper. That is, we have $X = X_1\amalg\cdots \amalg X_e$ where $X_i$ is a fat point. If $X_i$ is reduced for all $i$, then the DT invariant is just the number of points, namely $e$. But in the general case, the DT invariant is
\[
\chi(X,\nu_X) = \sum_{i=1}^e \nu_{X_i},
\]
which is one motivation for the interest in the computation of the Behrend number of a fat point. Even though a moduli space as above is rarely $0$-dimensional, there are examples where this actually happens, see e.g.~\cite[Thm.~3.55 and \S\,4]{ThomasThesis}, but also \cite[Ex.~8.1]{CohDT} and \cite[Thm.~1.1]{Thomas_obstructed}.

\smallbreak
We conclude the introduction with a challenging open problem in DT theory. 

\begin{openproblem}\label{open-problem-DT}
Fix integers $r \geq 1$ and $n\geq 0$, and let $\Quot_{\BA^3}(\OO^{\oplus r},n)$ be the Quot scheme  parametrising length $n$ quotients of the trivial sheaf $\OO^{\oplus r}$ on $\BA^3$, a key character in DT theory \cite{Virtual_Quot,FMR_K-DT}. As proved in \cite{BR18} (see \cite{BFHilb} for the $r=1$ case), there is an identity
\[
\chi(\Quot_{\BA^3}(\OO^{\oplus r},n),\nu) = (-1)^{rn}\chi(\Quot_{\BA^3}(\OO^{\oplus r},n)),
\]
and the value of the Behrend function at a torus-fixed quotient $p=[\OO^{\oplus r} \onto T]$, with respect to the natural $(\BC^{\times})^{3+r}$-action on the Quot scheme, is
\[
\nu(p)=(-1)^{rn}.
\]
However, it is not known whether $\nu$ is \emph{constantly} equal to $(-1)^{rn}$. Its constancy would show that the Quot scheme mentioned above is generically reduced, which is currently unknown even for $r=1$, i.e.~in the case of the Hilbert scheme of points $\Hilb^n(\BA^3)$.

In dimension $N>3$, the question of reducedness of $\Quot_{\BA^N}(\OO^{\oplus r},n)$ has already been answered in the negative: see \cite[\S\,6.5]{Joachim-Sivic} for an example of a generically nonreduced component of $\Quot_{\BA^N}(\OO^{\oplus r},8)$, where $r>3$. See also the recent work of Szachniewicz \cite{Szachniewicz} for a proof of the fact that $\Hilb^{13}(\BA^6)$ is nonreduced.
\end{openproblem}

%%%%%%%%%%%%%%%%%%%%%%%%%%%%%%%%%%%%%%%%%%%%%
\subsection{Organisation of contents}
The paper is structured as follows. In \Cref{sec:conventions} we set up the notation, we review the notions of fat points, cones, blowups and their normalisations; we recall the definition of the Behrend function. In \Cref{sec:behrend-first-examples} we prove the key result (\Cref{lemma:multiplicities_blowup}) that we will exploit to perform our computations, and we compute a number of examples of Behrend functions using its elementary properties. In \Cref{sec:towers} we introduce \emph{towers}, we completely describe their blowups (subsection \ref{sec:blowup-tower}), and we prove \Cref{thm:intro1}. An algorithm to generalise such results is explained in \Cref{sec:Algorithm}. In \Cref{sec:monomial-normal} we prove \Cref{thm:intro2} and \Cref{thm:intro1241}. In \Cref{sec:monomial} we express the Behrend number of an arbitrary fat point $I \subset \BC[x_1,\ldots,x_N]$ in terms of data defined on the normalisation $Z_I \to \Bl_I\BA^N$, thus proving \Cref{thm:intro3}. In \Cref{sec:3-fold-difficulties} we give an example involving fat points $X \subset \BA^3$ showing that the analysis we carried out in $\BA^2$ needs nontrivial modifications in order to work in higher dimension.

\begin{conventions}
We work over the field $\BC$ of complex numbers. All \emph{schemes} will be separated and of finite type over $\BC$. A \emph{variety} will be an integral (reduced and irreducible) scheme over $\BC$. If $X$ is a scheme and $V \subset X$ is a subvariety, we denote by $\OO_{X,V}$ the local ring of $X$ at the generic point of $V$. When $V$ is an irreducible component of $X$, we denote by $\mult_VX$ the length of the local artinian ring $\OO_{X,V}$, viewed as a module over itself. The function field of a variety $X$, namely the residue field of $\OO_{X,X}$, will be denoted $\BC(X)$. We shall denote by $\mm = (x,y) \subset \BC[x,y]$ the maximal ideal of the origin $0 \in \BA^2$. More terminology and background will be set in \Cref{sec:conventions}.
\end{conventions}

\subsection*{Acknowledgments}
We wish to thank Ugo Bruzzo, Joachim Jelisiejew, Yunfeng Jiang, Andrea Petracci and Fatemeh Rezaee for very helpful and enlightening discussions around several of the topics discussed in this paper.

%%%%%%%%%%%%%%%%%%%%%%%%%%%%%%%%%%%%%%%%%%%%%
%%%%%%%%%%%%%%%%%%%%%%%%%%%%%%%%%%%%%%%%%%%%%
\section{Background material}\label{sec:conventions}

In this subsection we fix the notation used throughout the paper, and we collect some frequently used results.

%%%%%%%%%%%%%%%%%%%%%%%%%%%%%%%%%%%%%%%%%%%%%
\subsection{Fat points, monomial ideals and the Hilbert scheme}

\begin{definition}\label{def:fat_points}
A \emph{fat point} is a $\BC$-scheme $X$ isomorphic to $\Spec R$, where $(R,\mm_R)$ is a local artinian $\BC$-algebra. The \emph{embedding dimension} of a fat point $X=\Spec R$ is the integer $\dim_{\BC}(\mm_R/\mm_R^2)$. When this number is $1$, we say that $X$ is \emph{curvilinear}.
\end{definition}

Thus a fat point is a $\BC$-scheme $X$ such that $X_{\red} \into X \to \Spec \BC$ is the identity. In other words, it is a $0$-dimensional $\BC$-scheme whose underlying topological space is just one point. The embedding dimension of $X$ is nothing but the smallest dimension of a smooth $\BC$-scheme containing $X$ as a closed subscheme.

\begin{definition}\label{def:length}
The \emph{length} of a fat point $X=\Spec R$ is defined as
\[
\length (X) = \dim_{\BC} \HH^0(X,\OO_X) = \dim_{\BC}(R).
\]
\end{definition}

\begin{notation}
Occasionally, for the sake of readability, if $R=\BC[x_1,\ldots,x_N]/I$ defines a fat point $X=\Spec R \subset \BA^N$, we shall write $\ell_{R}$ instead of $\length(X)$.
\end{notation}

Up to isomorphism of $\BC$-schemes, there is only $\Spec \BC$ of length $1$, only $\Spec \BC[t]/t^2$ of length $2$, and only $\Spec \BC[t]/t^3$ and $\Spec \BC[x,y]/(x^2,xy,y^2)$ of length $3$, the latter being of embedding dimension $2$. If $\bfk$ is an arbitrary algebraically closed field, it is known that there is a finite number of isomorphism classes of local artinian $\bfk$-algebras of length $n\leq 6$, and that this number is infinite when $n >6$. See \cite{Poo1} for a complete classification of finite dimensional algebras, and \cite{MR18} for a classification of $\bfk[x,y]$-\emph{modules} of length up to $4$.

\smallbreak
Let $A = \BC[x_1,\ldots,x_N]$ be a polynomial ring. We say that an ideal $I \subset A$ is of \emph{finite colength} equal to $n$ if $A/I$ is a finite dimensional $\BC$-vector space of dimension $n$, i.e.~if $X=\Spec A/I$ is a disjoint union of fat points. Amongst all fat points $X \subset \BA^N$ of length $n$, there are finitely many special ones that are cut out by \emph{monomial} equations. There is a bijective correspondence between monomial ideals of colength $n$ in $\BC[x_1,\ldots,x_N]$ and $(N-1)$-dimensional partitions of $n$. If $N=2$, a $1$-dimensional partition corresponds to a \emph{Ferrers diagram} (also known in the literature as a \emph{Young diagram}) made of $n$ boxes, the correspondence being depicted in \Cref{fig:monomial-Ferrers}.
\begin{figure}[h!]
    \centering
\begin{tikzpicture}[scale=0.75]
\draw (0,1)--(7,1)--(7,0)--(0,0)--(0,6)--(1,6)--(1,0);
\draw (0,4)--(2,4)--(2,0);
\draw (0,3)--(3,3)--(3,0);
\draw (4,1)--(4,0);
\draw (5,1)--(5,0);
\draw (6,1)--(6,0);

\draw (0,2)--(3,2);
\draw (0,5)--(1,5);

\node at (0.5,0.5) {\small $1$};
\node at (1.5,0.5) {\small $x$};
\node at (2.5,0.5) {\small $x^2$};
\node at (3.5,0.5) {\small $x^3$};
\node at (4.5,0.5) {\small $x^4$};
\node at (5.5,0.5) {\small $x^5$};
\node at (6.5,0.5) {\small $x^6$};
\node at (7.5,0.5) {\small $\textcolor{blue}{x^7}$};

\node at (2.5,2.5) {\small $x^2y^2$};
\node at (1.5,2.5) {\small $xy^2$};
\node at (0.5,2.5) {\small $y^2$};

\node at (2.5,1.5) {\small $x^2y$};
\node at (1.5,1.5) {\small $xy$};
\node at (0.5,1.5) {\small $y$};
\node at (3.5,1.5) {\small $\textcolor{blue}{x^3y}$};

\node at (1.5,3.5) {\small $xy^3$};
\node at (0.5,3.5) {\small $y^3$};
\node at (2.5,3.5) {\small $\textcolor{blue}{x^2y^3}$};

\node at (0.5,4.5) {\small $y^4$};
\node at (1.5,4.5) {\small $\textcolor{blue}{xy^4}$};

\node at (0.5,5.5) {\small $y^5$};
\node at (0.5,6.5) {\small $\textcolor{blue}{y^6}$};
\end{tikzpicture}
\caption{The Ferrers diagram corresponding to the monomial ideal $I=(x^7,x^3y,x^2y^3,xy^4,y^6)$, whose generators define the staircase of the diagram. The length (number of boxes) is $17$.}
    \label{fig:monomial-Ferrers}
\end{figure}
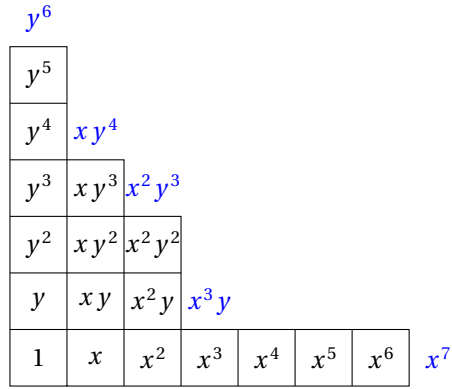

Any finite subscheme $X \subset \BA^N$ is a disjoint union of fat points. The moduli space parametrising finite subschemes $X \subset \BA^N$ of length $n$ is the \emph{Hilbert scheme of points} $\Hilb^n(\BA^N)$. It contains a projective subscheme $\Hilb^n(\BA^N)_0 \subset \Hilb^n(\BA^N)$, called the \emph{punctual Hilbert scheme}, parametrising fat points supported at the origin $0 \in \BA^N$. This scheme is known to be irreducible of dimension $n-1$ in the case $N=2$, by work of Briançon \cite{BRIANCON}. It is also irreducible if $N=3$ and $n\leq 11$, by work of Jelisiejew--Keneshlou \cite{Jelisiejew-Keneshlou}. The locus of \emph{all} fat points $X \subset \BA^N$ of length $n$ is of course given by $\BA^N \times \Hilb^n(\BA^N)_0$.

%%%%%%%%%%%%%%%%%%%%%%%%%%%%%%%%%%%%%%%%%%%%%
\subsection{Cones}
A \emph{cone} over a scheme $X$ is an $X$-scheme of the form
\[
\pi\colon \Spec_{\OO_X} \mathscr A \to X,
\]
where $\mathscr A = \bigoplus_{i\geq 0}\mathscr A_i$ is a quasicoherent sheaf of graded $\OO_X$-algebras such that the canonical map $\mathscr A_0 \to \OO_X$ is an isomorphism, $\mathscr A_1$ is coherent and generates $\mathscr A$ over $\mathscr A_0$.
Given a cone $C = \Spec_{\OO_X}\mathscr A \to X$, one can construct another cone
\[
C\oplus \mathbb 1 = \Spec_{\OO_X} \mathscr A[z] \to X
\]
where the $i$-th graded piece of $\mathscr A[z]$ is 
\[
\left(\mathscr A[z]\right)_i 
= \mathscr A_i  \oplus \mathscr A_{i-1}z \oplus  \cdots \oplus \mathscr A_1 z^{i-1} \oplus \mathscr A_0z^i.
\]
On the other hand, the \emph{projective cone} of $C$ is defined to be the $X$-scheme
\[
P(C) = \Proj_{\OO_X} \mathscr A \to X.
\]
The \emph{projective completion} of $C$, namely the projective cone
\[
P(C\oplus \mathbb 1) \to X,
\]
contains $C$ as a dense open subset with closed complement $P(C) \into P(C\oplus \mathbb 1)$, locally cut out by the equation $z = 0$.

The main example to which these constructions apply, of crucial importance in this paper, is the \emph{normal cone} of a closed immersion $X \into M$ of $\BC$-schemes, namely the cone
\[
C_{X/M} = \Spec_{\OO_X}\left(\bigoplus_{i\geq 0}\mathscr I^i/\mathscr I^{i+1}\right) \to X,
\]
where $\mathscr I \subset \OO_M$ is the ideal sheaf of $X \into M$.

%%%%%%%%%%%%%%%%%%%%%%%%%%%%%%%%%%%%%%%%%%%%%%%%%%%%%%%%%%%%%%%
\subsection{Blowups and exceptional loci}
Given a variety $M$ along with an ideal sheaf $\mathscr I \subset \OO_M$ cutting out a subscheme $X=V(\mathscr I) \into M$, we shall denote by
\[
\begin{tikzcd}
\Bl_{\mathscr I}M = \Proj_{\OO_M}\left(\displaystyle\bigoplus_{i\geq 0}\mathscr I^i\right) \arrow{r}{\varepsilon_{\mathscr I}} & M
\end{tikzcd}
\]
the blowup of $M$ along $X$. Sometimes we shall adopt the notation $\Bl_XM$, often used in the literature. The map $\varepsilon_{\mathscr I}$ is a projective birational (surjective) morphism of varieties which restricts to an isomorphism over $M\smallsetminus X$. The \emph{exceptional divisor} attached to such a blowup is, by definition, the effective Cartier divisor
\begin{equation}\label{inclusion_EXC_DIV}
E_{\mathscr I}M \into \Bl_{\mathscr I}M
\end{equation}
defined by the (invertible) sheaf of ideals
\[
\varepsilon_{\mathscr I}^{-1}(\mathscr I)\cdot \OO_{\Bl_{\mathscr I}M} = \im \left(\varepsilon_{\mathscr I}^\ast \mathscr I \to \OO_{\Bl_{\mathscr I}M}\right).
\]
In other words, $E_{\mathscr I}M = \Bl_{\mathscr I}M \times_MX$.
If $C_{X/M}=\Spec_{\OO_X}\left(\bigoplus_{i\geq 0}\mathscr I^i/\mathscr I^{i+1}\right)$ is the normal cone of the inclusion $X \into M$, then \eqref{inclusion_EXC_DIV} agrees with the natural inclusion of the projective cone
	    \[
	    P(C_{X/M}) = \Proj_{\OO_X}\left( \bigoplus_{i\geq 0}\mathscr I^i/\mathscr I^{i+1}\right)=E_{\mathscr I}M
	    \]
	    inside $\Bl_{\mathscr I}M$. In other words, the diagram
	    \[
	    \begin{tikzcd}
	    P(C_{X/M}) \arrow{d}\arrow[hook]{r} & \Bl_{\mathscr I}M \arrow{d}{\varepsilon_{\mathscr I}} \\
	    X\arrow[hook]{r} & M
	    \end{tikzcd}
	    \]
is cartesian.
Finally, we set
\[
\exc (\varepsilon_{\mathscr I}) = (E_{\mathscr I}M)_{\red}=P(C_{X/M})_{\red},
\]
and we call this reduced closed subscheme of $\Bl_{\mathscr I}M$ the \emph{exceptional locus} of the blowup. Sometimes, when no confusion is likely to arise, we shall denote it by $\exc(\Bl_{\mathscr I}M)$.

\begin{notation}
More generally, given a projective birational morphism $f\colon Y \to Z$ between quasiprojective varieties, we shall denote by $\exc(f) \subset Y$ the reduction of the preimage of the indeterminacy locus of the birational map $f^{-1}$.
\end{notation}

We shall make extensive use of the following results.

\begin{lemma}[{\cite[\href{https://stacks.math.columbia.edu/tag/01OF}{Tag 01OF}]{stacks-project}}]\label{LEMMATECH}
Let $M$ be a scheme, and let $\mathscr I_1,\mathscr I_2 \subset \OO_M$ be quasicoherent sheaves of ideals. Let $\varepsilon_{\mathscr I_1}\colon \Bl_{\mathscr I_1}M \to M$ be the blowup of $M$ along $\mathscr I_1$.  Then there is a canonical isomorphism of $M$-schemes
\[
\begin{tikzcd}
\Bl_{\varepsilon_{\mathscr I_1}^{-1}(\mathscr I_2) \cdot \OO_{\Bl_{\mathscr I_1}M}}(\Bl_{\mathscr I_1}M) \arrow{r}{\sim} & \Bl_{\mathscr I_1\cdot \mathscr I_2}(M).
\end{tikzcd}
\]
\end{lemma}

\begin{prop}[{\cite[Prop.~IV-22]{GEOFSCHEME}}]\label{blowaffine}
Let $M = \Spec A$ be an affine scheme, and consider a closed subscheme $X = V (f_0,f_1,\ldots,f_r)\into M$. The blowup of $M$ along $X$ agrees with the closure in $M\underset{A}{\times} \BP^{r}_A= \BP^{r}_A$ of the graph of the morphism
\[
\alpha_{(f_0,f_1,\ldots,f_r)} \colon M \smallsetminus X \rightarrow \BP^{r}_A
\]
induced by the map $\OO_M^{\oplus (r+1)} \to \OO_M$ sending $(a_0,a_1,\ldots,a_r) \mapsto \sum_{0\leq i\leq r} a_if_i$.
%ring homomorphism
%$$\begin{matrix}
%A^n &\rightarrow & A\\
%(a_1,\ldots,a_n) & \mapsto& \underset{i=1}{\overset{n}{\sum}}a_if_i.
%\end{matrix}$$
\end{prop}

\subsection{Normalisation and order functions}\label{section:normalization}

Recall that an irreducible quasicompact scheme $X$ is normal if for every closed point $p \in X$ the local ring $\OO_{X,p}$ is normal (integrally closed in its field of fractions). If $X$ is an integral scheme, then a \emph{normalisation} of $X$ is a pair $(Y,\mu)$, where $Y$ is a normal scheme and $\mu\colon Y \to X$ is a morphism such that if $\mu'\colon Y' \to X$ is a dominant morphism from a normal scheme $Y'$, then there exists a unique morphism $\theta \colon Y' \to Y$ such that $\mu\circ \theta = \mu'$.

\begin{prop}[{\cite[\S\,4.1.2, Prop.~1.22 and 1.25]{LIU}}]\label{prop:normalisation}
Let $X$ be an integral scheme. Then there exists a normalisation morphism $\mu \colon Y \rightarrow X$, unique up to unique isomorphism (of $X$-schemes). Moreover, a morphism $f\colon Y\rightarrow X$ is the normalisation morphism if and only if $Y$ is normal, and $f$ is birational and integral. If $X$ is a variety, the normalisation $\mu \colon Y \to X$ is a finite morphism.
\end{prop}

The following two results describe the behaviour of birational morphisms with a normal variety as a target. 

\begin{theorem}[Zariski's Main Theorem {\cite[Cor.~11.4]{HARTSHORNE}}]\label{ZMT} 
Let $f\colon X\rightarrow Y$ be a birational projective morphism of noetherian integral schemes, and assume that $Y$ is normal. Then, for every point $y\in Y$, the fibre $f^{-1}(y)$ is connected.
\end{theorem}

\begin{lemma}[{\cite[\href{https://stacks.math.columbia.edu/tag/0AB1}{Tag 0AB1}]{stacks-project}}]\label{lemmabir+norm=iso} A finite (or even integral) birational morphism $f\colon X\rightarrow Y$ of integral schemes with $Y$ normal is an isomorphism. 
\end{lemma}

Recall that if $X$ is a variety and $V \subset X$ is a prime cycle of codimension $1$, the order function $\ord_V \colon \BC(X)^\times \to \BZ$ is defined as follows: for $a \in \OO_{X,V}$, one sets $\ord_V(a) = \length_{\OO_{X,V}}(\OO_{X,V}/a\cdot \OO_{X,V})$, and for $h = a/b \in \BC(X)^\times$, one proves easily that the definition $\ord_V(h) = \ord_V(a) - \ord_V(b)$ is well given. This definition generalises the more familiar notion that applies when $X$ is normal: in this case, the local ring $(\OO_{X,V},\mm_{X,V})$ is a discrete valuation ring (and not just a local integral domain), and one defines $\ord_V(a)$ to be the largest integer $k$ such that $a \in \mm_{X,V}^k$. The two notions are related as shown in the next result.

\begin{prop}[{\cite[Ex.~1.2.3]{FULTON}}] \label{fulton_example}
Let $X$ be a variety, $\mu\colon Y\rightarrow X$ the normalisation of $X$, and let $V \into X$ be a subvariety. If $h \in \BC(X)^\times = \BC(Y)^\times$, then
\[
\ord_V(h) = \underset{\mu\colon W \to V}{\sum} \ord_W(h)\cdot  [\BC(W):\BC(V)],
\]
where the sum is over all subvarieties $W\into Y$ which map onto $V$, and $[\BC(W):\BC(V)]$ denotes the degree of the corresponding field extension. 
\end{prop}

\subsection{Normalisation of blowups}\label{sec:normal-blowups}
Recall that if $I$ is an ideal in a polynomial ring $A=\BC[x_1,\ldots,x_N]$, then the \emph{Rees algebra} of $I$ is 
\[
A[It] = A \oplus It \oplus I^2t^2 \oplus I^3t^3 \oplus \cdots \subset A[t].
\]
Since $A$ is a domain, so is $A[It]$, see \cite[\href{https://stacks.math.columbia.edu/tag/01OF}{Tag 01OF}]{stacks-project}, and therefore the blowup $\Bl_I\BA^N$ is again a variety (an integral scheme of finite type over $\BC$). 
If $I$ is monomial, by the general theory of normalisation in this setup (see e.g.~\cite[\S\,II.5]{Lipman_monomial} for a thorough treatment), the integral closure of $A[It]$ is 
\[
\overline{A[It]} = A \oplus \overline It \oplus \overline{I^2}t^2 \oplus \overline{I^3}t^3 \oplus \cdots
\]
where, after setting $x^m = x_1^{m_1}\cdots x_N^{m_N}$ for $m = (m_1,\ldots,m_N) \in \BN^N$, one defines
\begin{equation}\label{eqn:I_bar}
\overline{I^i} = \left(x^m \in A \,\,\big|\,\, (x^m)^p \in I^{ip}\textrm{ for some }p\geq 1\right) \subset A.
\end{equation}
By \cite[Ex.~6.C.9]{Patil}, the inclusion $A[It] \into \overline{A[It]}$ induces an everywhere defined morphism
\[
\mu_I \colon \Proj \overline{A[It]} \to \Bl_I\BA^N,
\]
which agrees with the normalisation morphism.
In general, $A[It]$ is normal if and only if $I^i = \overline{I^i}$ for every $i \geq 1$. 

By \cite[Prop.~3.1]{REID}, in the case $N=2$, the algebra $\overline{\BC[x,y][It]}$ agrees with the Rees algebra $\BC[x,y][\overline{I}t]$ of the monomial ideal $\overline I \subset \BC[x,y]$. In particular the normalisation of $\Bl_I\BA^2$ is the blowup of $\BA^2$ along $\overline I$.
This motivates the following common terminology.

\begin{definition}
We will say that an ideal $I \subset A$ is \emph{normal} if its Rees algebra $A[It]$ is normal. When $A=\BC[x,y]$, we will call $\overline{I}$ the normalisation of $I$.
\end{definition}

We next state a special case of \cite[Prop.~1.1]{MR2029820} suited for our purposes (the general statement involves a polynomial ring in an arbitrary number of variables). %This result is very useful to determine whether the blowup of $\BA^2$ along a monomial ideal (of finite colength) is normal.

\begin{prop}[{\cite[Prop.~1.1]{MR2029820}}]\label{villareal}
Let $I=(x^{a_1},x^{a_2}y^{b_2},\ldots, x^{a_{s-1}}y^{b_{s-1}},y^{b_s})\subset \BC[x,y]$ be a monomial ideal of finite colength and let $Q_I\subset \BR^2$ be the subset defined by
\[
Q_I=\conv_{\BQ} ((a_1,0),(a_2,b_2),\ldots,(a_{s-1},b_{s-1}),(0,b_s))+\BQ^2_{\ge0}
\]
where $\conv_{\BQ}(p_1,\ldots,p_s)\subset \BQ^2$ denotes the convex hull of a set of points $p_1,\ldots,p_s \in \BN^2$.
Then, for $i\neq 0$, one has
\[
\overline{I^i} = \left(x^ay^b \,\,\big|\,\, (a,b) \in i\cdot Q_I \cap \BZ^2 \right).
\]
\end{prop}

\begin{remark}\label{rmk:normality-convexity}
\Cref{villareal} provides a criterion to establish whether, given a monomial ideal of finite colength $I\subset\BC[x,y]$, the blowup variety $\Bl_I\BA^2$ is normal or not. Explicitly, if $Q_I$ is defined as in \Cref{villareal} and $A_I=\set{(a,b)\in\BN^2|x^ay^b\in I}$, then $\Bl_I\BA^2$ is normal if and only if 
\[
A_I=Q_I\cap \BN^2.
\]
Moreover, we have the equality $Q_I=\conv_{\BQ}(A_I)$ and, as a consequence, if $I,J\subset\BC[x,y]$ are normal ideals, then $IJ$ is normal. Indeed, by general properties of convexes (see {\cite[\S\,2.2.]{COX}}), we have
\begin{align*}
Q_{IJ}\cap \BN^2
&=\conv_{\BQ}(A_{IJ})\cap\BN^2\\
&=\conv_{\BQ}(A_{I}+A_J)\cap \BN^2\\
&=(\conv_{\BQ}(A_{I})+\conv_{\BQ}(A_J))\cap \BN^2\\
&=A_I+A_J\\
&=A_{IJ}.
\end{align*}
Notice that the converse in not true. For instance, setting $\mm = (x,y)$, we have
\[
\mm^3=\mm\cdot (x^2,y^2),
\]
and we shall see in \Cref{esnonnorm} that $(x^2,y^2)$ is not normal.
\end{remark}

\begin{example}
Consider the two ideals $I=(x^2,y^2)$ and $J=(x^2,y^3)$ in $\BC[x,y]$. Then, 
\begin{align*}
A_I&=\Set{(a,b)\in\BN^2 | a,b\geq 2 }, \\
A_J&=\Set{(a,b)\in\BN^2 | a\geq 2 ,b\geq 3 }.
\end{align*}
Since $(1,1)\in (Q_I\cap \BN^2)\smallsetminus A_I$ and $(1,2)\in (Q_J\cap \BN^2)\smallsetminus A_J$, the blowups $\Bl_I\BA^2$ and $\Bl_J\BA^2$ are not normal. The integral closures of the Rees algebras are respectively given by 
\begin{align*}
    \overline{\BC[x,y][It]}&=\BC[x,y][\overline{I}t] \\
    \overline{\BC[x,y][Jt]}&=\BC[x,y][\overline{J}t]
\end{align*}
where $\overline{I}=\mm^2$ and $\overline{J}=(x^2,xy^2,y^3)$.
\end{example}

We will see many examples of normalisations of blouwps of the affine plane $\BA^2$ with center a monomial ideal of finite colength $I\subset\BC[x,y]$ in \Cref{sec:monomial-normal,sec:monomial}.

Ferrers diagrams of ideals which have normal Rees algebras admit a useful description that was given in \cite{VILLAREAL2} and we present below.
\begin{theorem}[{\cite[Thm.~2.13]{VILLAREAL2}}]\label{teovilla}
    Let $I\subset\BC[x,y]$ be minimally generated by the $n+1$ monomials
\begin{equation}\label{eqn:monomial_cazzi}
    x^{a_0} , x^{a_1}y^{b_{n-1}} ,\ldots , x^{a_i}y^{b_{n-i}} ,\ldots, x^{a_{n-1}}y^{b_1}, y^{b_0}
\end{equation}
    where $a_i> a_{i+1}$ and $b_i>b_{i+1}$ for $i=0,\ldots,n-2$. Set $a_n=b_n=0$. If the blowup $\Bl_I\BA^2$ is normal, then there exists an integer $k$ such that $0 \le k \le n$ and
\begin{enumerate}
    \item  $a_n=0, a_{n-1} = 1, a_{n-2} = 2,\ldots , a_k = n - k$,
    \item $b_n=0, b_{n-1} = 1, b_{n-2} = 2, \ldots, b_{n-k} = k $,
    \item $b_i \le \lceil \frac{b_{i-1}+b_{i+1}}{2} \rceil$ for $i= 1,\ldots,n-k-1$,
    \item $a_i \le \lceil \frac{a_{i-1}+a_{i+1}}{2} \rceil$ for $i= 1,\ldots,k-1$.
\end{enumerate}
\end{theorem}
\begin{remark}\label{conseguenzeVillareal}
If an ideal $I \subset \BC[x,y]$, as in \eqref{eqn:monomial_cazzi}, is normal, then the boundary $\partial Q_I$ of $Q_I$ is piece-wise linear, i.e.
\[
\partial Q_I=\overline{\BQ^2\smallsetminus Q_I}\cap Q_I=s_0\cup\cdots\cup s_{t+1},
\]
where $\overline{\BQ^2\smallsetminus Q_I}$ denotes the Euclidean closure and $s_0,\ldots,s_{t+1}$ are, possibly unbounded, segments with different slopes. Let us also denote by $v_0,\ldots,v_t$ the vertices of $\partial Q_I$ i.e. 
\[
\set{v_i|0\leq i\leq t}=\set{s_i\cap s_j|0\leq i<j\leq t+1}.
\]
For instance, $I=(1)$ if and only if $t=0$ which also implies $v_0=(0,0)$. 

Then, as a consequence of \Cref{teovilla}, up to relabeling the linear pieces and the vertices, the following properties hold:
\begin{itemize}
    \item [$\circ$] the pieces $s_0$ and $s_{t+1}$ are unbounded and respectively supported on the positive half horizontal axis and on the positive half vertical axis,
    \item [$\circ$] for $i=0,\ldots,t$, we have $v_i=s_i\cap s_{i+1}$,
    \item [$\circ$] for $i=1,\ldots,t$, the segments $s_i$ are bounded and supported on certain lines $r_1,\ldots,r_t$, such that each line $r_i$ has negative slope $m_i$ and $0>m_i>m_{i+1}$ for all $i=1,\ldots,t-1$,
    \item [$\circ$] there is a strictly increasing sequence $0=k_0< \cdots< k_t=n$ of positive integers such that
    $v_i=(a_{k_i},b_{n-k_i})$, where $a_n=b_n=0$.
\end{itemize}
Now, the integer $k$ of \Cref{teovilla}, can be chosen as
\[
k=\max \Set{i|m_i\geq -1}.
\]
\end{remark}
Below we show an example of Ferrers diagram of a monomial $0$-dimensional scheme whose associated Rees algebra is normal.
\begin{example}\label{ex:villa} Consider the ideal $I = (x^6,x^4y,x^2y^2,xy^3,y^5) \subset \BC[x,y]$ and $Q_I,A_I$ as in \Cref{villareal,rmk:normality-convexity}. The Ferrers diagram of $I$ is
\[
\begin{tikzpicture}
\coordinate (1) at (-0.54,0.36);
\coordinate (2) at (-0.9,1.07);
\coordinate (3) at (1.25,-0.72);
\coordinate (4) at (0.54,-0.36);
\coordinate (5) at (-0.18,0);
\node ()  at (0,0) {\tiny{\ydiagram{1,1,2,4,6}}};
\foreach \x in {(1), (2), (3), (4), (5)}{
        \fill \x circle[radius=2pt];
    }
\draw[pattern=north east lines] (-0.9,2)-- (2)-- (1)-- (5)-- (4)-- (3) --(2,-0.72);
\end{tikzpicture}
\]
where the highlighted area in the above picture corresponds to $Q_I$. Then, the ideal $I$ is normal because $A_I=Q_I\cap\BN^2$.
\end{example}

\begin{example}\label{esnonnorm}
Set $I = (x^k,y^k) \subset \BC[x,y]$, where $k>1$. Then $I$ is not normal, and the normalisation of $\Bl_I\BA^2$ is given by $\Bl_{\mm^{k}}\BA^2$. Moreover, as we shall see in \Cref{powermax} (but see also \cite[Ex.~II.7.11]{HARTSHORNE}), there is a canonical isomorphism
\[
\begin{tikzcd}
\Bl_{\mm}\BA^2 \arrow{r}{\sim} & \Bl_{\mm^{k}}\BA^2 =
\Bl_{\overline I}\BA^2.
\end{tikzcd}
\]
Composing with the normalisation morphism, one obtains a morphism
\[
\Bl_{\mm}\BA^2 \to \Bl_I\BA^2,
\]
induced of course by $I \subset \overline I \subset \mm$.
An example with $k=5$ is depicted below.
\begin{figure}[h!]
    \centering
\begin{tikzpicture}
\coordinate (2) at (-0.72,1.07);
\coordinate (3) at (1.06,-0.72);
\node ()  at (0,0) {\tiny{\ydiagram{5,5,5,5,5}}};
\foreach \x in {(2), (3)}{
        \fill \x circle[radius=2pt];
    }
\draw[pattern=north east lines] (-0.72,2.4)--(2)--(3) --(2.8,-0.72);
\end{tikzpicture}
\caption{The normalisation of the ideal $I=(x^5,y^5)$ is $\mm^5$.}% The Ferrers diagram of $\overline I$ consists of those squares that survived, i.e.~that stand below the grey area.}
    \label{fig:non-normal-square}
\end{figure}
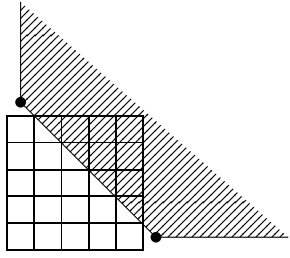
Again, as in \Cref{ex:villa}, if $Q_I\subset \BQ^2$ is defined as in \Cref{villareal}, then the highlighted area in the above figure corresponds to $Q_I$, but this time $Q_I \cap \BN^2 \neq A_I$.
\end{example}

\begin{example}\label{blowuplci}
In general, if $I=(x^k,y^h)$, the blowup $\Bl_I\BA^2$ is canonically isomorphic (see \cite[Prop.~IV-25]{GEOFSCHEME}), over $\BA^2$, to the quasiprojective surface
\[
B=\Set{((x,y),[u:v])\in\BA^2\times \BP^1 | vx^k=uy^h}\into\BA^2\times\BP^1.
\]
In particular, for $h,k>1$, the blowup $\Bl_I\BA^2$ is singular along the exceptional divisor and hence, it is not normal. 
\end{example}

%%%%%%%%%%%%%%%%%%%%%%%%%%%%%%%%%%%%%%%%%%%%%%%%%%%%%%%%%%%%%%%
\subsection{Self-intersection inside quasiprojective varieties}
Let $B=\Bl_I\BA^2$ be the blowup of the affine plane with center a fat point supported at the origin. Then $B$ contains a finite number of irreducible projective rational curves $C_1,\ldots,C_r$.

Suppose that $B$ is smooth. Whenever we will talk about self-intersection we will refer to the quadratic form on $\bigoplus_{1\leq i\leq r}\BZ C_i$ induced by the map
\[
\begin{tikzcd}[row sep=tiny]
\set{C_1,\ldots,C_r} \arrow{r}{(-)^2} & \BZ \\
C_i \arrow[mapsto]{r} & C_i^2 = \deg_{C_i}\left(\OO_B(C_i)\big|_{{C_i}}\right).
\end{tikzcd}
\]

%%%%%%%%%%%%%%%%%%%%%%%%%%%%%%%%%%%%%%%%%%%%%%%%%%%%%%%%%%%%%%%%%%%%%%%%%%%%%%%%%%%%%%%%%%
\subsection{Definition and main properties of the Behrend function}\label{sec:behrend}
Let $X$ be a scheme of finite type over $\BC$, and let $\Con(X)$ be the abelian group of ($\BZ$-valued) constructible functions on $X$. In \cite{Beh}, Behrend constructs a canonical constructible function
\[
\nu_X \colon X(\BC) \to \BZ,
\]
nowadays referred to as the `Behrend function' of $X$. It has been proven a powerful tool in enumerative geometry, mainly because of the following remarkable property: whenever $X$ is proper and carries a symmetric perfect obstruction theory in the sense of Behrend--Fantechi \cite{BFHilb}, one has an identity
\[
\int_{[X]^{\vir}} 1 = \chi(X,\nu_X),
\]
where the left hand side is the degree of $[X]^{\vir}\in A_0(X)$, the virtual fundamental class attached to the obstruction theory, and the right hand side is the \emph{weighted Euler characteristic} of $X$, the `weight' being $\nu_X$. Explicitly, for a constructible function $\gamma \in \Con(X)$, one defines
\[
\chi(X,\gamma) = \sum_{m \in \BZ} m \chi(\gamma^{-1}(m)).
\]
The Behrend function of a scheme $X$ is defined as 
\[
\nu_X = \Eu(\mathfrak c_X),
\]
where $\Eu\colon Z_\ast(X) \simto \Con(X)$ is the \emph{local Euler obstruction}, an isomorphism from cycles on $X$ to constructible functions on $X$, and $\mathfrak c_X$ is a canonical cycle attached to $X$. The definition of $\Eu$ is recalled in \cite{Beh} and is classical; we refer the reader to Jiang's work \cite{Jiang-Euler-Obs} for more details on local Euler obstruction (both in algebraic and analytic setting) and the Behrend function. Here we recall how the cycle $\mathfrak c_X$ is defined. Suppose $(f\colon U \to X,\iota\colon U \into M)$ is a local embedding for $X$, i.e.~$f$ is an \'etale morphism of $\BC$-schemes, and $\iota$ is a closed immersion into a smooth $\BC$-scheme $M$. Let 
\[
\pi\colon C_{U/M} \to U
\]
be the normal cone of this immersion. Note that $C_{U/M}$ is of pure dimension $\dim M$.
One can form the cycle
\[
\mathfrak c_{U/M} = \sum_{D \subset C_{U/M}}(-1)^{\dim \pi(D)} \mult_D(C_{U/M})\left[ \pi(D) \right]\,\in\,Z_\ast (U),
\]
where the sum ranges over all irreducible components $D$ of $C_{U/M}$, and $\mult_D(C_{U/M})$ denotes the \emph{geometric multiplicity} of the irreducible component $D$, namely the length
\[
\mult_D(C_{U/M}) = \length_{\OO_{C_{U/M},D}}(\OO_{C_{U/M},D})
\]
of the artinian ring $\OO_{C_{U/M},D}$ viewed as a module over itself, see e.g.~\cite[\href{https://stacks.math.columbia.edu/tag/0DR4}{Tag 0DR4}]{stacks-project}. The cycles $\mathfrak c_{U/M}$ naturally glue together along local embeddings to give a cycle $\mathfrak c_X \in Z_\ast(X)$, i.e.~there exists a unique global cycle $\mathfrak c_X$ such that if $(f\colon U \to X,\iota\colon U \into M)$ is a local embedding as above, one has $\mathfrak c_X|_U = \mathfrak c_{U/M}$.

When $X$ has a global embedding $\iota\colon X \into M$ inside a smooth scheme $M$ (e.g.~when $X$ is quasiprojective), we can use the local embedding $(\id_X,\iota)$ and compute directly
\begin{equation}\label{eqn:nu-embedding}
\nu_X = \Eu(\mathfrak c_X) = \sum_{D \subset C_{X/M}}(-1)^{\dim \pi(D)} \mult_D(C_{X/M})\Eu\left(\left[ \pi(D) \right]\right).
\end{equation}
Thus, when $X$ is a fat point, say with embedding dimension $N$, we have a closed immersion of $X$ inside $M=\BA^N$ and Equation \eqref{eqn:nu-embedding} becomes
\begin{equation}\label{formula:fat-behrend}
\nu_X = \sum_{D \subset C_{X/M}}\mult_D(C_{X/M}),
\end{equation}
because the local Euler obstruction of $[X]$ is equal to $1$ and $\dim \pi(D) = 0$. 

\begin{notation}
Occasionally, for the sake of readability, if $R=\BC[x_1,\ldots,x_N]/I$ defines a fat point $X=\Spec R \subset \BA^N$, we shall write $\nu_R$ instead of $\nu_X$, referring to the latter as the \emph{Behrend number} of $I$, since $\Spec R$ has only one point.
\end{notation}

The Behrend function also has the following remarkable property, which has been exploited several times for computations in Donaldson--Thomas theory, see e.g.~\cite{BFHilb,BR18}.

\begin{example}[{\cite[Cor.~2.4\,(iii)]{APPP1}}] \label{ex:milnor}
When $X$ is a \emph{critical locus}, i.e.~$X=V(\dd f)$ is the zero scheme of an exact $1$-form on a smooth scheme $M$, one has the relation
\begin{equation}\label{formula:milnor-fibre}
\nu_X(p) = (-1)^{\dim M}(1-\chi(\MF_{f,p})),
\end{equation}
where $\MF_{f,p}$ is the Milnor fibre of $f$ at $p \in X$. The right hand side is, by definition, the value of the \emph{Milnor function} attached to $(U,f)$. The above situation includes the important case $f=0 \in \OO_M(M)$, which yields $X=M$ and the formula $\nu_M(p) = (-1)^{\dim_p M}$. So the Behrend function of a smooth point of a scheme is always $\pm 1$.
\end{example}

%%%%%%%%%%%%%%%%%%%%%%%%%%%%%%%%%%%%%%%%%%%%%
%%%%%%%%%%%%%%%%%%%%%%%%%%%%%%%%%%%%%%%%%%%%%
\section{Behrend functions and blowups}
\label{sec:behrend-first-examples}

Ideally, it would be nice to compute the Behrend function of an arbitrary $0$-dimensional $\BC$-scheme.
It is of course enough to perform the computation for fat points, since a finite scheme is a disjoint union of fat points.
The Behrend function of a fat point $X$ is the constant given by Equation \eqref{formula:fat-behrend}, so our goal is to compute this constant exploiting such relation in a large number of cases. 

If $X$ is a (proper) moduli space of sheaves on a Calabi--Yau $3$-fold of dimension equal to the expected dimension, namely $0$, then $X = X_1\amalg \cdots \amalg X_e$ is a disjoint union of fat points $X_i$, and the non-reducedness of $X_i$ is a shadow of the existence of obstructed deformations for the object parametrised by $X_{i,\red} \into X$. Even though there is in general no control on these obstructions, one can compute the Donaldson--Thomas invariant of $X$ as the integer $\nu_{X_1} + \cdots + \nu_{X_e}$.

%%%%%%%%%%%%%%%%%%%%%%%%%%%%%%%%%%%%%%%%%%%%%%%%%%%%%%%%%%%%%%%%%%%%%%%%%%%%%%%%%%%%%%%%%
\subsection{A formula for the Behrend function in terms of blowup}\label{sec:key_relation}
Let $X=\Spec R$ be a fat point over $\BC$, and let $X \into U$ be a closed immersion into a smooth affine scheme $U$. Let $I \subset \OO_U$ be the ideal defining this inclusion, so that $R\cong \OO_U/I$, and let $C = C_{X/U} = \Spec \left(\bigoplus_{d\geq 0} I^d/I^{d+1} \right)$ be the normal cone to $X$ in $U$. As in \Cref{formula:fat-behrend}, we have
\begin{equation}\label{nu_length}
\nu_X=\sum_{D\subset C} \length_{\OO_{C,D}}\left(\OO_{C,D}\right),
\end{equation}
where the sum runs over all irreducible components $D$ of $C$. This sum does not depend on the particular embedding $X \into U$ we picked. If $N=\dim U$, then we know that $C$ is purely $N$-dimensional, but we do not know how many irreducible components it has in general; however, in the case where $I \subset \BC[x,y]$ is a normal monomial ideal, the number of components of the normal cone to the fat point $X=V(I) \into \BA^2$ can be computed via \Cref{bijection-irr-cpt-factorisation}.
Note that a natural choice for $U$ is the affine space $\BA^N$, where $N = \dim_{\BC}(\mathfrak m_R/\mathfrak m_R^2)$ is the embedding dimension of $X=\Spec R$.

\smallbreak
One first observation, towards the computation of $\nu_X$ via  \Cref{nu_length}, is that the projective cone $P(C)$ sits in the projective completion $P(C\oplus \mathbb{1})$ as the divisor `at infinity', with open \emph{dense} complement equal to $C$. Hence we may rewrite \Cref{nu_length} as
\begin{equation}\label{eqn:P(C plus 1)}
\nu_X = \sum_{D\subset C} \length_{\OO_{P(C\oplus \mathbb{1}),P(D\oplus \mathbb{1})}}\left(\OO_{P(C\oplus \mathbb{1}),P(D\oplus \mathbb{1})}\right).
\end{equation}
We notice that $P(C\oplus \mathbb 1)$ is the exceptional divisor of a blowup, just as $P(C)$ is. We consider the embedding $X \into \BA^N \into \BA^N \times \BA^1 = M$, where the second map is induced by the inclusion of the origin $0$ in $\BA^1$. Then, we have an identity
\[
E_XM = P(C\oplus \mathbb{1}) \subset \Bl_XM,
\]
so by \Cref{eqn:P(C plus 1)} we have to determine the geometric multiplicities of the irreducible components of the exceptional divisor $E_XM$.

\Cref{eqn:P(C plus 1)} will be used to explicitly compute the Behrend function of curvilinear schemes (see \Cref{ex:curvilinear}). However, for most of this paper the key relation that will be exploited is the one contained in the following lemma. 

\begin{lemma}\label{lemma:multiplicities_blowup}
Let $X \subset \BA^N$ be a fat point, with normal cone $C=C_{X/\BA^N}=\Spec S$ and associated projective cone $P = P(C) = \Proj S$. Then the association $D \mapsto P(D)$ is a bijective correspondence between irreducible components of $C$ and irreducible components of $P$, and there is an identity
\[
\nu_X = \sum_{D \subset C} \length_{\OO_{P(C),P(D)}} \left(\OO_{P(C),P(D)}\right).
\]
In `blowup language', this can be rephrased as
\[
\nu_X = \sum_{D \subset C} \mult_{P(D)}E_X\BA^N.
\]
\end{lemma}

\begin{proof}
For a general cone $C=\Spec S$ over an affine scheme $X=\Spec R$, the irreducible components $D \subset C$ are themselves cones (over subvarieties of $X$), each of which is given as $D = V(\mathfrak p) = \Spec(S/\mathfrak p)$, where $\mathfrak p \subset S$ is a homogeneous minimal prime ideal. Thus $D \mapsto P(D)$ is a bijection. 

If $X \subset \BA^N$ is a fat point with normal cone $C$ as in the statement, and $D=V(\mathfrak p) \subset C$ is an irreducible component, the local ring $S_{\mathfrak p}$ is artinian, as well as the homogeneous localisation $S_{(\mathfrak p)}$, and we have
\begin{align*}
   \length_{\OO_{C,D}}\left(\OO_{C,D}\right)
&= \length_{S_{\mathfrak p}} \left(S_{\mathfrak p}\right) \\
&= \length_{S_{(\mathfrak p)}} \left(S_{(\mathfrak p)}\right) \\
&= \length_{\OO_{P(C),P(D)}}\left(\OO_{P(C),P(D)}\right),
\end{align*}
which by \Cref{nu_length} implies the formula for $\nu_X$.
\end{proof}

\subsection{The Behrend function of the easiest fat points}\label{sec:first-examples-nu}
We conclude this section with some examples of computation of $\nu_X$ for $X$ a fat point.

\begin{example}[Critical loci]\label{ex:critical-locus}
Let $X=\Spec R$ be a fat point that is also a critical locus, i.e.~the zero locus of an exact $1$-form $\dd f$, for some function $f \in \OO_U(U)$ on a smooth scheme $U$. In particular, $R$ is equal to the Jacobian ring attached to $(U,f)$, whose dimension as a $\BC$-vector space is by definition the Milnor number $\mu_f$. In this case, one has 
\[
\nu_X = \length(X).
\]
Indeed, since $X_{\red}$ is just one point, $X \into U$ is isolated and so \Cref{formula:milnor-fibre} holds, giving
\[
\nu_X = (-1)^{m+1} (1-\chi(\MF_f)),
\]
where $m+1$ is the complex dimension of $U$ and where $\MF_f$ has the same homotopy type of a bouquet of $\mu_f$ spheres $\BS^m \subset \BR^{m+1}$. This implies
\[
\chi(\MF_f) = \mu_f \cdot (1+(-1)^m) - (\mu_f-1) = 1+(-1)^m \mu_f,
\]
which indeed gives
\[
\nu_X = (-1)^{m+1}(1-(1+(-1)^m \mu_f)) = \mu_f = \dim_{\BC}(R) = \length(X).
\]
\end{example}

\begin{example}[Local complete intersections]\label{ex:lci}
Let $X \subset \BA^N$ be fat point that is also a local complete intersection subscheme. Then $C_{X/\BA^N} = N_{X/\BA^N}$ is the total space of a vector bundle over $X$ of rank $N$. Thus $P(C_{X/\BA^N})$ is a $\BP^{N-1}$-bundle over $X$, and as such it is irreducible, with multiplicity equal to $\length(X)$. So 
\[
\nu_X=\length(X).
\]
\end{example}

\begin{example}[Curvilinear scheme]\label{ex:curvilinear}
Fix an integer $n>0$ and consider the curvilinear scheme $X_n = \Spec \BC[t]/t^n$. Then $\nu_{X_n}=n$ follows by both \Cref{ex:critical-locus} and \Cref{ex:lci}. We first confirm this formula by means of \Cref{nu_length}, as follows: for every $d \geq 0$, the $d$-th graded piece of the coordinate ring of $C_{X_n/\BA^1}$ is isomorphic to $R_n = \BC[t]/t^n$ as an $R_n$-module: if $I=(t^n)$, then
\[
I^d/I^{d+1} = \Braket{t^{nd},t^{nd+1},\ldots,t^{nd+n-1}}_{\BC} \cong R_n.
\]
Thus
\[
\bigoplus_{d\geq 0} I^d/I^{d+1} = \bigoplus_{d\geq 0} R_n\cdot z^d = \BC[t,z]/t^n,
\]
proving that
\[
C_{X_n/\BA^1} = \BA^1 \times X_n,
\]
which is irreducible with generic point $(t) \subset \BC[z,t]/(t^n)$ of length $n$.
Alternatively, we could have checked the formula $\nu_{X_n}=n$ through \Cref{eqn:P(C plus 1)} as follows. We can blow up $X_n$ inside $M=\BA^1 \times \BA^1$, obtaining the exceptional divisor
\begin{align*}
P(C_{X_n/\BA^1}\oplus \mathbb 1)&=\Proj \left[\bigoplus_{d\geq 0} (I,z)^d/(I,z)^{d+1}\right] \\
&=\Proj \left[\bigoplus_{d\geq 0}\left(R_n\cdot z^d\oplus \frac{I}{I^2}\cdot z^{d-1}\oplus \cdots\oplus \frac{I^{d-1}}{I^d}\cdot z \oplus \frac{I^d}{I^{d+1}} \right) \right] \\
&\cong \Proj \left[R_n\oplus (R_n\cdot z\oplus R_n\cdot u) \oplus (R_n\cdot z^2\oplus R_n\cdot zu\oplus R_n \cdot u^2) \oplus \cdots \right] \\
&= \Proj R_n[z,u] = \BP^1 \times X_n,
\end{align*}
which again is irreducible with generic point of length $n$.
\end{example}

The following is an instance of both \Cref{ex:critical-locus} and \Cref{ex:lci}.

\begin{example}
Set $X=\Spec R\subset \BA^N$, where $R=\BC[x_1,\ldots,x_N]/(x_1^{e_1},\ldots,x_{N}^{e_N})$. Then, $X$ is the critical locus of the function $\BA^N \to \BA^1$ sending
\[
(x_1,\ldots,x_N) \mapsto \sum_{1\leq i\leq N}\frac{1}{e_i+1}x_i^{e_{i}+1}.
\]
Thus by \Cref{ex:critical-locus} we have
\[
\nu_X=\length(X) = \prod_{1\leq i\leq N}e_i.
\]
Alternatively, this formula also follows from the multiplicativity of the Behrend function, proved in general in \cite[Prop.~1.5\,(ii)]{Beh}.
\end{example}

So far we have only seen instances where the Behrend number of a fat point agrees with its length. In general, the length is neither an upper bound nor a lower bound for the Behrend number, as we shall see in greater detail by means of the core calculations of this paper (see e.g.~\Cref{TEOREMA1}, \Cref{thm:general-tower-nu} and \Cref{rmk:no-upper-bound} for a few instances of this fact).
For now, we present an example of a fat point with embedding dimension $N=2$, that is neither a critical locus nor a local complete intersection. 

\begin{example}[Power of maximal ideal]\label{powermax}
Fix an integer $d>1$. Set $X=\Spec R$, where $R=\BC[x,y]/\mathfrak m^d$. Here $\mathfrak m = (x,y)$ denotes, as ever, the maximal ideal of the origin in $\BA^2$. We have a commutative diagram
\begin{equation}\label{diag:veronese_d}
\begin{tikzcd}
\Bl_{\mathfrak m}\BA^2 \arrow[hook]{r}\arrow{d}{g} & \BA^2 \times \BP^1 \arrow[hook]{d}{\id \times \mathsf v_{1,d}} \\
\Bl_{\mathfrak m^d}\BA^2 \arrow[hook]{r} & \BA^2 \times \BP^d
\end{tikzcd}
\end{equation}
where $\mathsf v_{1,d} \colon \BP^1 \into \BP^d$ is the Veronese embedding, sending $\BP^1$ onto the rational normal curve of degree $d$ inside $\BP^d$. The vertical map $g$ is an isomorphism, which by the commutativity of the diagram commutes with the projections down to $\BA^2$. It follows that, under this isomorphism, the exceptional divisor $E \subset \Bl_{\mathfrak m^d}\BA^2$ corresponds to the preimage of $X$ along $\varepsilon_{\mm}\colon \Bl_{\mathfrak m}\BA^2 \to \BA^2$. 

Now, as in \Cref{blowuplci}, we can write
\[
\Bl_{\mathfrak m}\BA^2 = \Set{((x,y),[u:v])|xv=yu} \subset \BA^2 \times \BP^1,
\]
and, after fixing coordinates $(u,y)$ in the chart $\set{v\neq 0} \subset \Bl_{\mm}\BA^2$, the blouwp map $\varepsilon_{\mm}$ becomes $(u,y) \mapsto (yu,y)$ in this chart. Therefore the pullback of $X = V(\mathfrak m^d)$ along $\varepsilon_{\mm}|_{v\neq 0}$ is the scheme cut out by the ideal
\[
J_v=(y^du^d,(y^{d-1}u^{d-1})y,\ldots,(yu)y^{d-1},y^d)=(y^d)\subset \BC[u,y].
\]
An identical calculation can be done in the chart $u\neq 0$, where one finds the ideal $J_u = (x^d)$ in $\BC[v,x]$. All in all, $\varepsilon_{\mm}^{-1}(X) \subset \Bl_{\mm}\BA^2$ (which is isomorphic to $E = E_{\mm^d}\BA^2$) is defined by the ideal sheaf $\mathscr J^d$, where $\mathscr J$ is the ideal defining the (reduced) exceptional divisor in $\Bl_{\mathfrak m}\BA^2$. It is thus a line with multiplicity $d$. Hence, by \Cref{lemma:multiplicities_blowup}, 
\[
\nu_X = \length_{ \OO_{E,E}}\left(\OO_{E,E}\right) = d.
\]
Note that, in this case, we have $d = \nu_X < \length(X) = (d+1)d/2$.
\end{example}

The previous example can be generalised as follows.

\begin{prop}\label{blowpowers}
Let $I \subset A=\BC[x_1,\ldots,x_N]$ be an ideal of finite colength. Then, for any integer $d>0$, one has a canonical $\BA^N$-isomorphism $\Bl_{I^d}\BA^N\cong\Bl_I\BA^N$, and an identity
\[
\nu_{A/I^d} = d\cdot \nu_{A/I}.
\]
In particular, if $I$ defines a local complete intersection subscheme of $\BA^N$, then
\[
\nu_{A/I^d} = d\cdot \ell_{A/I} = d\cdot \dim_{\BC}(A/I).
\]
\end{prop}

\begin{proof}
The second identity follows from the first combined with \Cref{ex:lci}. We have an isomorphism of $\BA^N$-schemes $g\colon \Bl_{I}\BA^N \simto \Bl_{I^d}\BA^N$, which is part of a larger diagram (constructed along the same lines as Diagram \ref{diag:veronese_d}): if we assume $I$ is minimally generated by polynomials $f_0,f_1,\ldots,f_r \in A$, then we have a commutative diagram
\[
\begin{tikzcd}
\Bl_I\BA^N \arrow[hook]{r} \arrow{d}{g} & \BA^N \times \BP^{r} \arrow[hook]{d}{\id \times \mathsf v_{r,d}} \\
\Bl_{I^d}\BA^N \arrow[hook]{r} & \BA^N \times \BP^{\binom{r+d}{d}-1}
\end{tikzcd}
\]
where $\mathsf v_{r,d} \colon \BP^{r} \to \BP^{\binom{r+d}{d}-1}$ is the Veronese embedding.

Now, if $\varepsilon_d \colon \Bl_{I^{d}}\BA^N \to \BA^N$ denotes the blowup morphism, $E_d = E_{I^d}\BA^N\subset \Bl_{I^{d}}\BA^N$ is the exceptional divisor embedded with ideal sheaf $\mathscr I_d \subset \OO_{\Bl_{I^{d}}\BA^N}$, and the inclusion $\Spec A/I^d \subset \BA^N$ has normal cone $C_d$, we compute
\begin{align*}
\nu_{A/I^d} 
&= \sum_{D \subset C_d} \mult_{P(D)}(E_d) \\
&= \sum_{D \subset C_d} \mult_{P(D)}\left(V(\varepsilon_d^{-1}(I^d) \cdot \OO_{\Bl_{I^d}\BA^N})\right) \\
&= \sum_{D \subset C_d} \mult_{g^{-1}P(D)}\left(V(\varepsilon_1^{-1}(I^d) \cdot \OO_{\Bl_{I}\BA^N})\right) \\
&=\sum_{D \subset C_d} \mult_{g^{-1}P(D)}\left(V(\mathscr I_1^d)\right) \\
&=d\cdot\sum_{D\subset C_d} \mult_{g^{-1}P(D)}(E_1) \\
&=d\cdot\sum_{D\subset C_1} \mult_{P(D)}(E_1) \\
&=d\cdot \nu_{A/I},
%&= \sum_{D \subset C_d} \length_{\OO_{E_d,P(D)}}\left(\OO_{E_d,P(D)}\right) \\
%&= \sum_{D \subset C_d} \length_{\OO_{f^{-1}(E_d),f^{-1}(P(D))}} \left(\OO_{f^{-1}(E_d),f^{-1}(P(D))}\right) \\
%&= \sum_{D \subset C_d} \length_{\OO_{\pi_1^{-1}(X_d),f^{-1}(P(D))}}\left(\OO_{\pi_1^{-1}(X_d),f^{-1}(P(D))}\right).
\end{align*}
as required.
\end{proof}

%%%%%%%%%%%%%%%%%%%%%%%%%%%%%%%%%%%%%%%%%%%%%%%%%%%%%%%%%%%
%%%%%%%%%%%%%%%%%%%%%%%%%%%%%%%%%%%%%%%%%%%%%%%%%%%%%%%%%%%
\section{Towers and their Behrend functions}
\label{sec:towers}

%%%%%%%%%%%%%%%%%%%%%%%%%%%%%%%%%%%%%%%%%%%%%%%%%%%%%%%%%%%
\subsection{Towers and their basic properties}
Before introducing \emph{towers}, special ideals in $\BC[x,y]$ particularly suited for our calculations, we quickly review some basics on curvilinear schemes. %In what follows, we let $\mm = (x,y)\subset\BC[x,y]$ be the maximal ideal of the origin $0\in\BA^2$.

We focus here on curvilinear schemes (cf.~\Cref{def:fat_points}) of length $n$ supported at the origin $0 \in \BA^2$. These are defined by ideals $I \subset \BC[x,y]$ of the form
\[
I=(f)+\mm^{n},
\]
where $f\in\mm\smallsetminus \mm^2$.
Given such a polynomial $f = ax+by+cx^2+dxy+ey^2+\cdots$, an explicit isomorphism $\BC[t]/t^n \to \BC[x,y]/I$ is given by sending $t + (t^n) \mapsto (ay-bx) + I$. Such association is an isomorphism because $(a,b) \neq (0,0)$ which follows from the condition $f\in\mm\smallsetminus \mm^2$.

	\begin{definition}
	Let $f\in\BC[x,y]$ be any nonzero polynomial, and, for $i\geq 0$, let $f_i$ be its homogeneous part of degree $i$. We will denote by $o(f)$ the \emph{order} of $f$, i.e. 
	\[
	o(f)=\min\Set{i\in\BN|f_i\not=0}.
	\]
	\end{definition}
	
\begin{lemma}[{\cite[Prop.~IV.1.1]{BRIANCON}}]
\label{normalform} Let $I=(f)+\mm^{n}$ be a curvilinear ideal. Then, the polynomial $f$ can be chosen in one of the following forms: either
	\[
	f(x,y)=x+g_x(y),
	\]
where $g_x\in\BC[y]$ is such that $o(g_x)\ge 1$ and $\deg (g_x)<n$, or
	\[
	f(x,y)=y+g_y(x),
	\]
where $g_y\in\BC[x]$ is such that $o(g_y) \ge 1$ and $\deg(g_y)< n$.
\end{lemma}

\begin{prop}\label{prop:A_n-sing}
Let $I=(f)+\mm^{n} \subset \BC[x,y]$ be a curvilinear ideal with $n\ge2$. Then, the blowup $\Bl_{I}{\BA^2}$ has a Kleinian singular point of type $A_{n-1}$.
\end{prop}
\begin{proof}
By \Cref{normalform}, we can suppose $f=x+g_x(y)$ and, as a consequence, $I=(x+g_x(y),y^{n})$. In particular, the sequence $x+g_x(y),y^{n}$ is a regular sequence. Thus, by \cite[Prop.~IV-25]{GEOFSCHEME}, we have 
\[
\Bl_{I}{\BA^2}=\Set{((x,y),[u:v])\in\BA^2\times\BP^1|uf-vy^{n}=0}.
\]
Now, an easy computation shows that the point $((0,0),[0:1])$ is a Kleinian singularity of type $A_{n-1}$.
\end{proof}

The next definition introduces our main objects of study for this section.

\begin{definition}\label{def:tower}
We will say that an ideal $K \subset \BC[x,y]$ is a \emph{tower of height} $i_s$ if there exists a polynomial $g_x\in \BC[y]$ or $g_y\in\BC[x]$, of degree strictly smaller than $i_s$, and a strictly increasing sequence of natural numbers $1\le i_1<i_2<\cdots< i_s$, such that 
\[
K=\prod_{k=1}^s (x+g_x(y))+\mm^{i_k}
\]
or
\[
K=\prod_{k=1}^s (y+g_y(x))+\mm^{i_k}.
\]
We will say that a tower is 
\begin{itemize}
\item [(i)] \emph{complete} if $i_k=k$ for $k=1,\ldots,s$, 
\item [(ii)] \emph{monomial} if $g_x=0$, in the first case, or if $g_y=0$, in the second case.
    \end{itemize} 
\end{definition}

\begin{example} \label{generatorsnoncomplete}
Let $1\le i_1<i_2<\cdots<i_{s}$ be a strictly increasing sequence of positive integers, and let
\[
K=\prod_{k=1}^s(x)+\mm^{i_k}=\prod_{k=1}^s(x,y^{i_k})
\]
be a monomial tower of height $i_s$, not necessarily complete. Then
\begin{equation}\label{eqn:nc-tower-explicit}
K=\left(x^s,x^{s-1}y^{i_1},x^{s-2}y^{i_1+i_2},\ldots,xy^{\underset{j=1}{\overset{s-1}{\sum}}i_j},y^{\underset{j=1}{\overset{s}{\sum}}i_j}\right).
\end{equation}
The associated Ferrers diagram is depicted in \Cref{fig:complete-monomial} in the complete case.
\begin{figure}[h!]\scalebox{1.2}{
\begin{tikzpicture}[scale=0.9]
    \draw (0,3)--(0.5,3)--(0.5,0);
    \draw (0,1.5)--(1,1.5)--(1,0);
    \draw (0,0.5)--(1.5,0.5)--(1.5,0)--(0,0);
    \draw (0,1)--(1,1);
    \draw (0,2)--(0.5,2);
    \draw (0,2.5)--(0.5,2.5);
    \draw (-1.5,3)--(-1.5,4)--(-2,4)--(-2,4.5);
    \draw (-2.5,4.5)--(-2.5,3);
    \draw (-2.5,4)--(-2,4)--(-2,3);
    \draw (-2.5,3.5)--(-1.5,3.5);
    \draw (-2.5,5.5)--(-2.5,6)--(-2,6)--(-2,5.5);
    \draw (-2.5,1)--(-2.5,0)--(-1.5,0) ;
    \draw (-2.5,0.5)--(-1.5,0.5);
    \draw (-2,0)--(-2,1);
    \draw[dashed] (-2.5,1)--(-1.5,1)--(-1.5,0);
    \draw[dashed] (0,0)--(0,3);
    \draw[dashed] (-2.5, 3)--(-1.5,3);
    \draw[dashed] (-2.5, 4.5)--(-2,4.5);
    \draw[dashed] (-2.5, 5.5)--(-2,5.5);
    \node at (-2,2) {$\vdots$};
    \node at (-2.25,5.15) {$\vdots$};
    \node at (-0.7,1.5) {$\cdots$};
    \node at (1.7,0.25) {\tiny $x^{s}$};
    \node at (1.4,0.75) {\tiny $x^{s-1}y$};
    \node at (0.95,1.75) {\tiny $x^{s-2}y^3$};
    \node at (0.25,3.25) {\tiny $x^{s-3}y^6$};
    \node at (-2.25,6.4) {\tiny $y^{\underset{i=1}{\overset{s}{\sum}}i}$};
    \node at (-1.55,4.35) {\tiny $xy^{\underset{i=1}{\overset{s-1}{\sum}}i}$};
\end{tikzpicture}}
    \caption{The Ferrers diagram of a complete monomial tower of height $s$.}
    \label{fig:complete-monomial}
\end{figure}
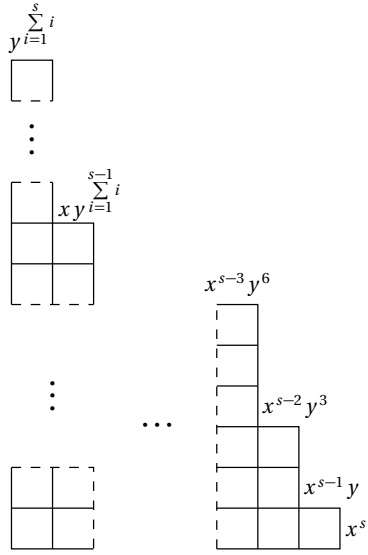
\end{example}

\begin{lemma}\label{lemma:normality-towers} 
The blowup of $\BA^2$ with center an arbitrary tower $K \subset \BC[x,y]$ is a normal surface. Equivalently, every tower is a normal ideal.
\end{lemma}

\begin{proof}
We first observe that the isomorphism class of a subscheme $X \subset \BA^2$ defined by a tower is completely determined by the sequence of positive integers $1\leq i_1<\cdots < i_s$. Indeed, if $K=\prod_{1\leq k\leq s} (x+g(y))+\mm^{i_k}$, then the automorphism
\begin{equation}\label{eqn:automorphism}
\begin{tikzcd}[row sep=tiny]
\BA^2 \arrow{r}{\sim} & \BA^2 \\
(x,y) \arrow[mapsto]{r} & (x-g(y),y)
\end{tikzcd}
\end{equation}
induces an isomorphism between $V(K)\into \BA^2$ and $V(K')\into \BA^2$, where $K'$ is the monomial tower $\prod_{1\leq k\leq s} (x)+\mm^{i_k}$.
On the other hand, the blowup of $\BA^2$ with center a monomial tower $I \subset \BC[x,y]$ is normal, because the subsets $A_I \subset \BN^2$ and $Q_I\subset \BQ^2$, defined as in \Cref{rmk:normality-convexity} and in \Cref{villareal} respectively, satisfy $A_I=Q_I\cap \BN^2$. The conclusion then follows from \Cref{villareal}.
\end{proof}

%%%%%%%%%%%%%%%%%%%%%%%%%%%%%%%%%%%%%%%%%%%%%%%%%%%%%%%%%%%%%%%%%%
\subsection{Blowing up along towers}
\label{sec:blowup-tower}
This subsection contains the key structural results that we will need for the calculation of the Behrend function of a fat point $X\subset \BA^2$ cut out by a tower.

\begin{prop}\label{toricver} Let $s\ge 1$ be a positive integer and let $K_s$ be the complete monomial tower 
    \[
    K_{s}=\prod_{k=1}^s (x,y^{k}).
    \]
Then the blowup $\Bl_{K_{s}}\BA^2$ factors as a sequence of blowups
\[
\begin{tikzcd}
X_s \arrow{r}{\varepsilon_s} &
X_{s-1} \arrow{r}{\varepsilon_{s-1}} & 
\cdots \arrow{r}{\varepsilon_3} & 
X_2 \arrow{r}{\varepsilon_2} & 
X_1 \arrow{r}{\varepsilon_1} & 
\BA^2
\end{tikzcd}
\]
where 
\begin{align*}
    X_1&=\Bl_0\BA^2,\\
    X_{k+1}&=\Bl_{t_{k}}X_k,\quad k= 1,\ldots,s-1. 
\end{align*}
Here, $t_1$ is the toric point of $\exc(\varepsilon_1) \subset X_1$ corresponding to the line $\{ x=0 \}$ and, for $k= 2,\ldots,s-1$, $t_{k}$ is the only toric point of $\exc(\varepsilon_k)\smallsetminus \varepsilon_{k
}^{-1}(\exc(\varepsilon_{k-1}))$.
    
In other words, $X_{s}$ and $\Bl_{K_{s}}{\BA^2}$ are canonically isomorphic as $\BA^2$-schemes.
\end{prop}

\begin{proof}
For the sake of readability, we set $I_k=(x,y^{k})$.
The proof goes by induction on the height $s$ of the tower. The first nontrivial case is $s=2$. We want to prove that there exists a canonical isomorphism of $\BA^2$-schemes $\varphi\colon X_2\rightarrow \Bl_{K_2}\BA^2$. 

Lemma \ref{LEMMATECH} implies that $$\Bl_{K_2} \BA^2 \cong \Bl_{\varepsilon_1^{-1}(I_2)\cdot \OO_{X_1}}X_1.$$
Recall (see {\cite[\S\,3.1]{COX}}) that $X_1$ is a toric surface covered by two toric charts $U_i\cong\BA^2$, for $i=0,1$, with the property that, if we call $a_i,b_i$ the toric coordinates on $U_i$, then the maps ${\varepsilon_1}|_{U_i}$, for $i=0,1$, take the the form
\begin{align*}
\varepsilon_1\big|_{U_0}(a_0,b_0)&=(a_0b_0,b_0)\\
\varepsilon_1\big|_{U_1}(a_1,b_1)&=(a_1,a_1b_1).
\end{align*}
As a consequence 
\begin{align*}
    \varepsilon_1\big|_{U_0}^{-1}(I_2)\cdot \BC[a_0,b_0]&=(a_0b_0,b_0^2)=(b_0)\cdot (a_0,b_0), \\
    \varepsilon_1\big|_{U_1}^{-1}(I_2)\cdot \BC[a_1,b_1]&=(a_1,a_1^2b_1^2)=(a_1).
\end{align*}
Therefore, we conclude that $\varepsilon_1^{-1}(I_2)\cdot \OO_{X_1}=\mathscr H_1\cdot \mathscr H_2$ where $\mathscr H_1 \subset \OO_{X_1}$ defines a Cartier divisor and $\mathscr H_2 \subset \OO_{X_1}$ defines a (reduced) toric point $t_1\in X_1$. Thus, we have
\[
\Bl_{\mathscr H_1\cdot\mathscr H_2}{X_1}\cong \Bl_{t_1}X_1=X_2,
\]
which concludes the proof of the base step.

Suppose now that we have a canonical isomorphism of $\BA^2$-schemes
$\varphi_s\colon X_s\simto \Bl_{K_s}\BA^2$. We need to construct a canonical isomorphism
\[
\begin{tikzcd}
\varphi_{s+1}\colon X_{s+1} \arrow{r}{\sim} & \Bl_{K_{s+1}}\BA^2.
\end{tikzcd}
\]
Setting $\psi_s=\varepsilon_1\circ\cdots\circ \varepsilon_s$, we have a commutative diagram
\[
\begin{tikzcd}[row sep=large,column sep=large]
X_{s+1}\arrow[dotted,bend left]{rr}[description]{\varphi_{s+1}}\arrow[swap]{d}{\varepsilon_{s+1}}\arrow{r}{\overline{\varphi}_s} & \Bl_{\varphi_{s}(t_{s})}\Bl_{K_s}\BA^2\arrow{d}{\varepsilon_{\varphi_{s}(t_{s})}} & \Bl_{K_{s+1}}\BA^2 \arrow{d}{\varepsilon_{K_{s+1}}} \\
X_s\arrow{r}{\varphi_s}\arrow[bend right]{rr}[description]{\psi_s} & \Bl_{K_s}\BA^2\arrow{r}{\varepsilon_{K_s}} & \BA^2 
\end{tikzcd}
\]
where the map $\overline{\varphi}_s$ is an isomorphism by the base change properties of blowups \cite[Prop.~IV-21]{GEOFSCHEME}.
 
Now, we exploit the toric variety structure on $X_i$, for all $i \in \BN$. If $N$ denotes the standard 2-dimensional lattice, then the variety $X_{s}$ can be constructed via the fan $\Sigma_{s}$ in $N\otimes_{\BZ} \BR\cong\BR^2$ depicted in \Cref{fan:X_s}.

\begin{figure}[h!]
 \begin{tikzpicture}
 \node at (-1,0.5) {$\Sigma_s=$};
 \draw[ -]  (0,0)--(7,1);
 \draw[ -] (0,0)--(1,1);
 \draw[ -] (0,1)--(0,0)--(1,0);
 \draw[ -] (2,1)--(0,0);
 \node[above] at (2.2,1) {\small $ 2e_1+e_2 $};
 \node[above] at (0.8,1) {\small $ e_1+e_2 $};
 \node[above] at (6.8,1) {\small $ s e_1+e_2 $};
 \node at (2,0.5) {\small $ \cdots $};
 \node[left] at (0,1) {\small $  e_2 $};
 \node[right] at (1,0) {\small $ e_1  $};
 \end{tikzpicture}
 \caption{A fan realising the toric variety $X_s$.}
 \label{fan:X_s}
\end{figure}
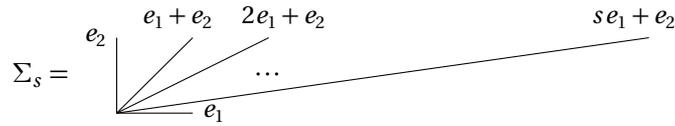

The variety $X_s$ is covered by $s+1$ smooth charts $U_k = \Spec S_k \cong \BA^2$ with toric coordinates $a_k$ and $b_k$, more precisely we set
\begin{align*}
    S_k &= \BC[xy^{-k+1},x^{-1}y^{k}] = \BC[a_k,b_k], \qquad 1\leq k\leq s \\
    S_{s+1} &= \BC[xy^{-s},y] = \BC[a_{s+1},b_{s+1}].
\end{align*}
As above, the maps ${\psi_s}|_{{U_i}}\colon U_i \to \BA^2$, for $i=1,\ldots,s+1$, have the explicit description
 \begin{align*}
 {\psi_s}\big|_{{U_k}}(a_k,b_k) &=(a_k^kb_k^{k-1},a_kb_k), \qquad 1\leq k\leq s\\
 {\psi_s}\big|_{{U_{s+1}}}(a_{s+1},b_{s+1}) &=(a_{s+1}b_{s+1}^{s},b_{s+1}).           
 \end{align*}
Therefore, the ideal sheaf $\psi_s^{-1}(I_{s+1})\cdot\OO_{X_s} \subset \OO_{X_s}$ is given, locally on each chart, by
\begin{align*}
    {\psi_s}\big|_{U_k}^{-1}(I_{s+1})\cdot \BC[a_k,b_k] &=(a_k^kb_k^{k-1},a_k^{s+1}b_k^{s+1})=(a_k^kb_k^{k-1}) \subset S_k, \qquad 1\leq k\leq s \\
    {\psi_s}\big|_{U_{s+1}}^{-1}(I_{s+1})\cdot \BC[a_{s+1},b_{s+1}]&=(a_{s+1}b_{s+1}^{s },b_{s+1}^{s+1})=(b_{s+1}^{s})\cdot (a_{s+1},b_{s+1}) \subset S_{s+1}.
\end{align*}
As a consequence, $\psi_s^{-1}(I_{s+1})\cdot\OO_{X_s}=\mathscr H_1\cdot\mathscr H_2$ where, as above, $\mathscr H_1\subset \OO_{X_s}$ defines a Cartier divisor on $X_{s}$ and $\mathscr H_2$ defines the reduced toric point $t_{s}\in  \exc(\varepsilon_s)\smallsetminus\varepsilon_s^{-1}(\exc(\varepsilon_{s-1}))$. Finally, the statement follows by applying again \Cref{LEMMATECH}.
\end{proof}

\begin{corollary}\label{blowuprod}
Let $K$ be the complete tower $K=\prod_{1\leq k\leq s}(x+g(y))+\mm^{k}$, where $g(y)\in \BC[y]$ is a polynomial of order $o(g)\ge 1$ and degree $\deg(g)<s$. Then, the blowup $\varepsilon_K\colon \Bl_{K}{\BA^2}\rightarrow \BA^2$ factors as a sequence of blowups 
\[
	\begin{tikzcd}
	\Bl_{p_{s-1}}{\Bl_{p_{s-2}}{\cdots\Bl_{p_1}{\Bl_{0}{\BA^2}}}} \arrow{r}{\varepsilon_s} &
	\cdots\arrow{r}{\varepsilon_3} &
	\Bl_{p_1}{\Bl_{0}{\BA^2}}\arrow{r}{\varepsilon_2} &
	\Bl_{0}{\BA^2}\arrow{r}{\varepsilon_1} & 
	\BA^2
	\end{tikzcd}
\]
where $p_1\in\exc(\varepsilon_1)$ and, for all $k=2,\ldots,s-1$, the point $p_k $ belongs to $\exc(\varepsilon_{k})\smallsetminus\varepsilon_{k}^{-1}\exc(\varepsilon_{k-1})$.
\end{corollary}

\begin{proof} 
It is enough to combine \Cref{toricver} with the automorphism \eqref{eqn:automorphism} introduced in the proof of \Cref{lemma:normality-towers}.
\end{proof}

\begin{remark}\label{blownoncompl}
The above corollary, combined with \Cref{lemma:normality-towers}, also allows one to handle the blowup of $\BA^2$ along \emph{any} tower
\[
K=\underset{k=1}{\overset{s}{\prod}} (x+g(y))+\mm^{i_k}.
\]
Indeed, given the complete tower $\overline{K}$ defined by
\[
\overline{K}=\underset{k=1}{\overset{i_s}{\prod}} (x+g(y))+\mm^{k},
\]
the blowup $B=\Bl_K\BA^2$ can be obtained by contracting some projective lines in $\overline{B}=\Bl_{\overline{K}}\BA^2$.

In a little more detail, if we call $\varepsilon\colon B\rightarrow \BA^2$ and $\bar{\varepsilon}\colon \overline{B}\rightarrow \BA^2$ the blowup maps, the same computations as in the proof of \Cref{toricver} show that $\bar{\varepsilon}^{-1}(K)\cdot \OO_{\overline{B}} $ defines a Cartier divisor on $\overline{B}$. Therefore, there is a canonical birational morphism of $\BA^2$-schemes
\[
\varphi\colon \overline{B}\rightarrow B
\]
which has connected fibres by Zariski's Main Theorem (\Cref{ZMT}) (which we may apply since $B$ is normal, by \Cref{lemma:normality-towers}). In particular, the map $\varphi$ is an isomorphism outside from the respective exceptional loci of $\overline{B}$ and $B$ and it may contract some of the irreducible components of $\exc(\overline{B})$. 

 Since any tower is isomorphic to a monomial tower (see the proof of \Cref{lemma:normality-towers}), in order to understand which rational projective curves of $\overline{B}$ are contracted by $\varphi$, we can first suppose that $K$ is a monomial tower. Then, the usual toric geometry methods apply. A fan $\Sigma$ for the toric variety $B$ consists of the following $s+1$ maximal cones
\begin{align*}
    \sigma_0&=\Braket{e_2,i_1e_1+e_2}, \\
    \sigma_1&=\Braket{i_1e_1+e_2,i_2e_1+e_2} ,\\
    &\vdots \\
    \sigma_{s-1}&=\Braket{i_{s-1}e_1+e_2,i_se_1+e_2}, \\
    \sigma_{s}&=\Braket{i_se_1+e_2,e_1}.
\end{align*}
In particular, if we put $i_{0}=0$, the cones $\sigma_{j}$, for $j=0,\ldots,s-1$, correspond either to a smooth point, if $i_{j+1}-i_{j}-1=0$, or to a Kleininian singularity of type $A_{i_{j+1}-i_{j}-1}$ otherwise, whereas the cone $\sigma_{s} $ corresponds to a smooth point of $B$. Now, a fan $\overline{\Sigma}$ for the toric variety $\overline{B}$ has the following maximal cones
\begin{align*}
    \tau_0&=\Braket{e_2,e_1+e_2} ,\\
    \tau_1&=\Braket{e_1+e_2,2e_1+e_2}, \\
    &\vdots \\
    \tau_{i_s-1}&=\Braket{(i_{s}-1)e_1+e_2,i_se_1+e_2} ,\\
    \tau_{i_s}&=\Braket{i_se_1+e_2,e_1}.
\end{align*}
Moreover, the fact that each cone $\tau_i$ of $\overline{\Sigma}$ is contained in some cone $\sigma_j$ of $\Sigma$ implies that there is a morphism of $\BA^2$-schemes from $\overline{B} $ to $B$ which, by universality, must coincide with $\varphi$. Therefore, the lines contracted by $\varphi$ are the lines in $\overline{B} $ corresponding to the rays of $\overline{\Sigma}$ not belonging to $\Sigma$.

Notice also that, if one knows how to compute the Behrend number of a complete tower (which we do, as we shall see in \Cref{TEOREMA1}), then, thanks to this remark, one also knows how to compute the Behrend number of an arbitrary tower. To see this, consider a curve $C\subset \overline{B}$ which is not contracted by $\varphi$. Then, since, when restricted to the complement $U\subset\overline{B}$ of the contracted lines, $\varphi$ is an isomorphism, we have the identity
\[
\mult_{\varphi(C)}(V({\varepsilon}^{-1}(K)\cdot\OO_{ {B}}))=\mult_{C}(V(\bar{\varepsilon}^{-1}(K)\cdot \OO_{\overline{B}})).
\]
\end{remark}
The following example provides a generalisation of \cite[Prop.~IV-40]{GEOFSCHEME}.

\begin{example}\label{exblowcurv}
The easiest non-complete tower one can think of is given by a curvilinear ideal $I=(x)+\mm^{n}=(x,y^{n})$ with $n\ge2$. We can deduce, from the above remark, an alternative way to \Cref{prop:A_n-sing}, to prove that $\Bl_I\BA^2$ has an (isolated) singularity of type $A_{n-1}$. 

Let $K$ be the monomial complete tower
\[
K=\prod_{k=1}^n (x)+\mm^{k} = \prod_{k=1}^n (x,y^k).
\]
Then, \Cref{toricver} implies that the exceptional locus of the map  $\varepsilon_K\colon \Bl_{K}{\BA^2}\rightarrow \BA^2 $ is a chain of $n$ rational smooth projective curves
\[
{E}_1 \cup {E}_2\cup \cdots \cup {E}_{n} \subset \Bl_K\BA^2
\]
and, the classical blowup formula (see {\cite[I-\S\,9, Thm.~(9.1)]{BARTHECC}}), implies:
\begin{equation*}
    E_k^2 =
    \begin{cases}
    -2 & \textrm{ if }k=1,\ldots,n-1 \\
    -1 & \textrm{ if }k=n.
    \end{cases}
\end{equation*}
Now, as a consequence of \Cref{blownoncompl}, the canonical projective birational morphism
\[
\varphi\colon \Bl_K\BA^2\rightarrow \Bl_I\BA^2
\]
contracts the curves
\[
{E}_1,\ldots,{{E}}_{n-1},
\]
and, the characterisation of Kleinian singularities (see {\cite[III-\S\,3, Prop.~(3.4)]{BARTHECC}}), implies that $\Bl_I\BA^2$ has an Kleinian singularity of type $A_{n-1}$.
\end{example}

\subsection{The Behrend function of a tower}
We are ready to tackle the calculation of the Behrend number of a tower. We start with the complete case.

\begin{theorem}\label{TEOREMA1}
	Let $K_s\subset\BC[x,y]$ be a complete tower of height $s$. Then
	\begin{align}
	\ell_{\BC[x,y]/K_s} &=
	\binom{s+2}{3},\label{id1}\\
	\nu_{\BC[x,y]/K_s}&=
	\frac{s(s+1)(2s+1)}{6}.\label{id2}
	\end{align}
	In particular $\ell_{\BC[x,y]/K_s}<\nu_{\BC[x,y]/K_s}$ for all $s>1$.
\end{theorem}

\begin{proof} \Cref{id1} follows directly from \Cref{generatorsnoncomplete} together with the equality\footnote{Such number is known as the $s$-th 
\emph{tetrahedral number}.}
\[
\binom{s+2}{3}=\frac{s(s+1)(s+2)}{6}=\sum_{k=1}^s\sum_{i=1}^k i=\sum_{j=0}^{s-1}\sum_{i=1}^{j+1}i .
\]
We now prove \Cref{id2}. Let $D = \psi_s^{-1}(V(K_s))$ be the subscheme of  $X_s=\Bl_{t_{s-1}}{\cdots}\Bl_{t_1}\Bl_{0}\BA^2$, where the points $t_i$ are as in \Cref{toricver}, defined as the scheme-theoretic preimage of $V(K_s)\subset \BA^2$ via the iterated blowup map $\psi_s\colon X_s\rightarrow \BA^2$. Then, \Cref{toricver} allows us to identify the $\BA^2$-schemes $X_s$ and $\Bl_{K_s}\BA^2$ and, as a consequence, to compute the Behrend number of the ideal $K_s$ as
\[
\nu_{\BC[x,y]/K_s}=\underset{C\subset D}{\sum}\mult_CD,
\]
where the sum ranges over all irreducible components $C$ of $D$. Notice that, if $\varepsilon \colon \Bl_{K_s}\BA^2\rightarrow \BA^2$ denotes the blowup morphism then, under the canonical isomorphism $X_s\cong\Bl_{K_s}\BA^2$, the exceptional locus $\exc(\varepsilon)$ corresponds to $D_{\red}$.

Recall that $D_{\red}$ is a chain of smooth rational projective curves $C_1,\ldots,C_s$, where $C_1$ corresponds to the blowup of the origin of $\BA^2$ under the isomorphism of \Cref{toricver} and
\[
C_i\cap C_j=\begin{cases}
\mbox{one point}&\mbox{if }|i-j|=1\\
\emptyset &\mbox{if }|i-j|>1.
\end{cases}
\]
We thus have to compute the sum
\[
\sum_{i=1}^s \mult_{C_i}D.
\]
Recall also that $X_s$ is covered by $s+1$ charts isomorphic to $\BA^2$ defined by
\begin{align*}
    U_k &=\Spec \BC[xy^{-k+1},x^{-1}y^{k}], \qquad 1\leq k\leq s \\
    U_{s+1} &=\Spec \BC[xy^{-s},y].
\end{align*}
If $a_k,b_k$ are the toric coordinates on $U_k$ for $k=1,\ldots,s+1$, the map $\psi_s$ is
 \begin{align*}
 {\psi_s}\big|_{{U_k}}(a_k,b_k) &=(a_k^kb_k^{k-1},a_kb_k), \qquad 1\leq k\leq s\\
 {\psi_s}\big|_{{U_{s+1}}}(a_{s+1},b_{s+1}) &=(a_{s+1}b_{s+1}^{s},b_{s+1}).           
 \end{align*}
Therefore, the ideal sheaf $\psi_s^{-1}(K_s)\cdot \OO_{X_s} \subset \OO_{X_s}$ is given, locally on each chart, by
\begin{align*}
{\psi_s}\big|_{U_k}^{-1}(K_s)\BC[a_k,b_k]
&=(a_k^kb_k^{k-1},a_kb_k)(a_k^kb_k^{k-1},a_k^2b_k^2)\cdots(a_k^kb_k^{k-1},a_k^kb_k^k)\cdots(a_k^kb_k^{k-1},a_k^{s}b_k^{s}) \\
&=(a_kb_k\cdot a_k^2b_k^2\cdots a_k^{k-1}b_k^{k-1}\cdot a_k^kb_k^{k-1} \cdots a_k^kb_k^{k-1}) \qquad \textrm{ for }k\leq s \\
{\psi_s}\big|_{U_{s+1}}^{-1}(K_s)\BC[a_{s+1},b_{s+1}]
&=(a_{s+1}b^{s}_{s+1},b_{s+1})(a_{s+1}b_{s+1}^{s},b_{s+1}^2)\cdots (a_{s+1}b_{s+1}^{s},b_{s+1}^{s}) \\
 &=(b_{s+1}b_{s+1}^2\cdots  b_{s+1}^{s}).
\end{align*}
We can read from the above formulas the contribution $a_{ij}$ of the curvilinear ideal $(x)+\mm^{i}$ to the multiplicity of the component $C_j$ of the exceptional divisor of $\varepsilon$. This information is encoded in the matrix
 $$A=(a_{i,j})_{i,j\in\{1,\ldots,s \}}=\begin{pmatrix}
 1&1&\cdots &1 &1 \\
 1&2&\cdots &2 &2 \\
 \vdots&\vdots&\ddots&\vdots&\vdots\\
 1&2&\cdots& s-1&s-1\\
 1&2&\cdots&s-1&s
 \end{pmatrix}.$$
For instance, the last column is given by the vector of the exponents of $b_{s+1}$ in the last displayed equation. Notice that, $a_{ij}=\min\{i,j\}$.

The Behrend number of $K_s$ is
\[
\nu_{\BC[x,y]/K_s}=\underset{i,j\in\{1,\ldots,s \}}{\sum }a_{ij}.
\]
 In order to complete the proof we observe that
  $$A=\begin{pmatrix}
 1&1&\cdots &1 &1 \\
 1&1&\cdots &1 &1 \\
 \vdots&\vdots&\ddots&\vdots&\vdots\\
 1&1&\cdots& 1&1\\
 1&1&\cdots&1&1
 \end{pmatrix}+\begin{pmatrix}
 0&0&\cdots &0 &0 \\
 0&1&\cdots &1 &1 \\
 \vdots&\vdots&\ddots&\vdots&\vdots\\
 0&1&\cdots& 1&1\\
 0&1&\cdots&1&1
 \end{pmatrix}+\cdots+\begin{pmatrix}
 0&0&\cdots &0 &0 \\
 0&0&\cdots &0 &0 \\
 \vdots&\vdots&\ddots&\vdots&\vdots\\
 0&0&\cdots& 1&1\\
 0&0&\cdots&1&1
 \end{pmatrix}+\begin{pmatrix}
 0&0&\cdots &0 &0 \\
 0&0&\cdots &0 &0 \\
 \vdots&\vdots&\ddots&\vdots&\vdots\\
 0&0&\cdots&0&0\\
 0&0&\cdots&0&1
 \end{pmatrix}.$$
 Hence, we have
	$$\nu_{\BC[x,y]/K_s}=\sum_{k=1}^s k^2=\frac{s(s+1)(2s+1)}{6}$$
 which complete the proof.
\end{proof}

Comparing Behrend functions, we obtain the following easy corollary.

\begin{corollary}
A complete tower of height at least $2$ is not a curvilinear ideal, i.e.~it has embedding dimension $2$.
\end{corollary} 

Following the prescriptions in \Cref{blownoncompl}, with similar techniques, one can prove the following generalisation of \Cref{TEOREMA1}.

\begin{theorem}\label{thm:general-tower-nu}
Let $1\le i_1<\cdots<i_s$ be a strictly increasing sequence of positive integers and let $K=\prod_{1\leq k\leq s} (x+f(y))+\mm^{i_k} $ be a tower of height $i_s$. Then there are identities
\begin{align*}
\ell_{\BC[x,y]/K}
&=\underset{k=1}{\overset{s}{\sum}}\underset{j=1}{\overset{k}{\sum}}i_j, \\
\nu_{\BC[x,y]/K}
&=\ell_{\BC[x,y]/K}+\underset{j=1}{\overset{s-1}{\sum}}i_j(s-j).
\end{align*}
\end{theorem}

We have thus computed the length and the Behrend number of an arbitrary tower, and the latter happens to be greater than the former.

%%%%%%%%%%%%%%%%%%%%%%%%%%%%%%%%%%%%%%%%%%%%%%%%%%%%%%%%%%%%%%%%%%%%%%%%
\subsection{Products of towers, Dynkin diagrams and Behrend functions}

\begin{definition}\label{defdinkyn}
Let $X$ be a smooth quasiprojective surface and let $C_1,\ldots,C_s\subset X$ be $s$ distinct rational smooth projective curves with the property that $C_i\cap C_j$ is either empty or a singleton for $i\not=j$. We will call \emph{Dynkin diagram} of the set of curves $\set{C_i | i=1,\ldots,s}$ a diagram made of:
\begin{itemize}
    \item [(i)] $s$ circles that we will call nodes, each labeled by one of the curves, and decorated with its self-intersection, 
    \item [(ii)] for any $i\not= j$ such that $C_i\cap C_j\not=\emptyset$, a segment joining the nodes labeled by $C_i$ and $C_j$.
\end{itemize}
\end{definition}
\begin{example}\label{ex:Dynkin} Let $K=\prod_{1\leq i\leq s}I_i$ be a complete tower. Then, as explained in \Cref{blowuprod} the variety $X=\Bl_K\BA^2$ can be obtained after a sequence of blowups each with center a reduced point and hence, $X$ is smooth. Moreover, the exceptional locus $\exc(\varepsilon)$ of the blowup map
\[
\varepsilon\colon  X\rightarrow\BA^2
\]
consists of a chain of rational smooth projective curves
$\Cald=\{C_1,\ldots,C_s\}$. In particular, they satisfy the same property as the curves in \Cref{defdinkyn}.

Notice that, for all $j=1,\ldots,s$ the ideal sheaf $\varepsilon^{-1}(I_j)\cdot \OO_X \subset \OO_X$ defines a Cartier divisor. As a consequence, we can associate, to each ideal $I_j$ one of the curves $C_i$. We say that the curve $C_i$ corresponds to $I_j$ if the canonical morphism 
\[
\varphi\colon X\rightarrow \Bl_{I_j}\BA^2
\]
contracts all the curves $C_k\subset X$ for $k\not=i$ (see \Cref{exblowcurv}). Notice that this association is well defined. Indeed, $\Bl_{I_j}\BA^2$ is normal by  \Cref{lemma:normality-towers} and hence $\varphi$ has connected fibres by \Cref{ZMT}. As a consequence, only one of the curves $C_i$ can map bijectively onto the irreducible rational curve $\exc(\Bl_{I_j}\BA^2)$.

\begin{figure}[h!]
    \begin{tikzpicture}
    
    \node at (2,1.5) {\footnotesize $\cdots$};
    
    \draw (0.1,1.5)--(0.9,1.5);
    \draw (1.1,1.5)--(1.7,1.5);
    \draw (2.3,1.5)--(2.9,1.5);
    \draw (3.1,1.5)--(3.9,1.5);

\draw (0,1.5) circle (0.1);
		\node[above] at (0,1.6) {\footnotesize $I_1 $ };
		\node[below] at (0,1.4) {\footnotesize $-2 $ };
    
		\draw (1,1.5) circle (0.1);
		\node[above] at (1,1.6) {\footnotesize $I_{2} $ };
		\node[below] at (1,1.4) {\footnotesize $-2 $ };
    
		\draw (3,1.5) circle (0.1);
		\node[above] at (3,1.6) {\footnotesize $I_{s-1} $ };
		\node[below] at (3,1.4) {\footnotesize $-2 $ };
    
		\draw (4,1.5) circle (0.1);
		\node[above] at (4,1.6) {\footnotesize $I_{s} $ };
		\node[below] at (4,1.4) {\footnotesize $-1 $ };
    \end{tikzpicture}
    \caption{The Dynkin diagram of the tower $K$, with each node labeled by an ideal.}
    \label{fig:Dynkin-with-ideals}
\end{figure}
Sometimes, in the literature (see e.g.~\cite{BRIANCON}), the underlying unlabeled diagram is called \emph{bamboo}.
\end{example}

\begin{lemma}\label{ex:length_product_towers}
Let us consider the product two complete monomial towers of height $h$, of the form
\[
I_h= \left( \prod_{k=1}^h (x) + \mm^k\right)\cdot \left(\prod_{k=1}^h (y) + \mm^k\right).
\]
Then, the numbers $\set{\ell_{\BC[x,y]/I_h}|h\geq 1}$ satisfy the recursive relation
\[
\ell_{\BC[x,y]/I_h}=\ell_{\BC[x,y]/I_{h-1}} + h^2 + 3h - 1.
\]
Equivalently, we have 
\[
\ell_{\BC[x,y]/I_h} = \frac{h(h+1)(h+2)}{3}+h^2,
\]
which in turn equals $2\cdot\ell + h^2$, where $\ell$ is the colength of the tower $\prod_{k=1}^h (x) + \mm^k$.
\end{lemma}

\begin{proof}
The equivalence between the two formulas is straightforward to check and we leave it to the reader. We now prove the former.

For each $h\geq 1$, we have 
\begin{align*}
I_h 
&= \mm^2(x,y^2)(x^2,y)\cdots (x,y^h)(x^h,y) \\
&= (x^2,xy,y^2) (x^3,xy,y^3)\cdots (x^{h+1},xy,y^{h+1}).
\end{align*}
Then $I_h$, a product of $h$ monomial ideals, can be generated by $2h+1$ monomials, namely we have
\begin{equation}\label{eqn:2h+1_monomials}
I_h = \left(x^{h-i}y^{h+\binom{i+1}{2}},x^{h+\binom{i+1}{2}}y^{h-i}\,\,\big|\,\,0\leq i\leq h\right).
\end{equation}
A few examples are given in \Cref{fig:partitions}.
The integers 
\begin{equation}\label{eqn:maximal_power_tower}
    a_h = h + \binom{h+1}{2}, \quad h \geq 1,
\end{equation}
represent the maximal power of $x$ (equivalently, of $y$) appearing among the $2h+1$ generators of $I_h$.
The colength of $I_h$ is computed, thanks to \Cref{eqn:2h+1_monomials}, in a recursive way from the base case $\ell_{\BC[x,y]/I_1} = 3$.
We obtain
\[
\ell_{\BC[x,y]/I_h}=\ell_{\BC[x,y]/I_{h-1}} + 2a_h - 1 = \ell_{\BC[x,y]/I_{h-1}} + h^2 + 3h - 1.
\]
as required.
\end{proof}

The induction described in the proof works as depicted in \Cref{fig:partitions} below.
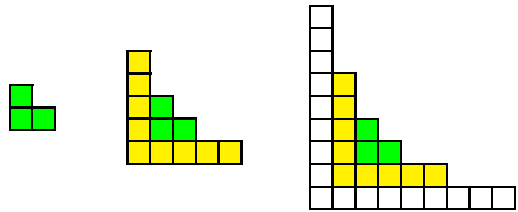
\begin{figure}[h]
\ytableausetup{smalltableaux}
\begin{tikzpicture}
\node at (1,0) {
\begin{ytableau}
*(green) \\
*(green) & *(green)
\end{ytableau}};

\node at (3,0) {\begin{ytableau}
*(yellow) \\
*(yellow) \\
*(yellow) & *(green) \\
*(yellow) & *(green) & *(green) \\
*(yellow) & *(yellow) & *(yellow) & *(yellow) & *(yellow) 
\end{ytableau}};

\node at (6,0) {\begin{ytableau}
*(white) \\
*(white) \\
*(white) \\
*(white) & *(yellow) \\
*(white) & *(yellow) \\
*(white) & *(yellow) & *(green) \\
*(white) & *(yellow) & *(green) & *(green) \\
*(white) & *(yellow) & *(yellow) & *(yellow) & *(yellow) & *(yellow) \\
*(white) & *(white) & *(white) & *(white) & *(white) & *(white) & *(white) & *(white) & *(white) 
\end{ytableau}};
\end{tikzpicture}
    \caption{The ideals $I_h$ for $h=1,2,3$. The lengths are $3$, $12$, $29$, and the heights of the respective Ferrers diagrams are $a_1=2$, $a_2=5$, $a_3=9$.}
    \label{fig:partitions}
\end{figure}

\begin{theorem}\label{TORIC PROD TOWERS} The following properties hold for complete towers.
\begin{enumerate}
    \item 
Let $K_x$ and $K_y $ be two complete towers, of heights  $h_x$ and $h_y$ respectively, of the form
\[
K_x =\underset{k=1}{\overset{h_x}{\prod}} (x+g_x(y))+\mm^{k} ,\qquad  K_y =\underset{k=1}{\overset{h_y}{\prod}} (y+g_y(x))+\mm^{k},
\]
for some $g_x\in\BC[y]$ and $g_y\in \BC[x]$ such that 
\[
[(x+g_x)+\mm^2] \neq  [(y+g_y)+\mm^2]\in\BP(\mm/\mm^2).
\]
Then
\begin{align}
    \ell_{\BC[x,y]/K_x\cdot K_y}&=\ell_{\BC[x,y]/K_x}+\ell_{\BC[x,y]/K_y}+h_x h_y \label{eqn:length_product_of_towers} \\
    \nu_{\BC[x,y]/K_x\cdot K_y}&=\nu_{\BC[x,y]/K_x}+\nu_{\BC[x,y]/K_y}+2h_xh_y -h_x-h_y.\label{nu_product_of_towers}
\end{align}
\label{item1}
\item Let $K_1$ and $K_2$ be two complete towers, of height respectively $h_1$ and $h_2$, of the form 
\[
K_1=\underset{k=1}{\overset{h_1}{\prod}} (x+g_1(y))+\mm^{k},\qquad K_2=\underset{k=1}{\overset{h_2}{\prod}} (x+g_2(y))+\mm^{k}.
\]
for some $g_1\not=g_2\in\BC[y]$ of respective degrees
\[
\deg(g_1)<h_1\mbox{ and }\deg(g_2)<h_2.
\]
Let $d=o(g_1-g_2)$ be the order of $g_1-g_2\in\BC[y]$. Then
\begin{equation}\label{eqn:product-towers-general}
 \nu_{\BC[x,y]/K_1\cdot K_2}=\nu_{\BC[x,y]/K_1}+\nu_{\BC[x,y]/K_2}+2h_1 h_2-d(h_1+h_2).   
\end{equation}
\label{item2}
\end{enumerate}
\end{theorem}

\begin{proof} 
First of all, a trivial computation shows that, since $[(x+g_x)+\mm^2] \neq  [(y+g_y)+\mm^2]$ are different points in $\BP(\mm/\mm^2)$, one also has $[(x-g_x)+\mm^2]\neq  [(y-g_y)+\mm^2]\in\BP(\mm/\mm^2)$. In particular, the Jacobian of the map
\[
\begin{tikzcd}[row sep=tiny]
\BA^2 \arrow{r}{\psi} & \BA^2 \\
(x,y) \arrow[mapsto]{r} & (x-g_x (y),y-g_y(x))
\end{tikzcd}
\]
has maximal rank at the origin $0 \in \BA^2$, i.e.~it is a biholomorphism nearby the origin. 
Such observation ensures that, in order to prove \eqref{item1}, it is enough to prove the statement for $g_x =g_y = 0$. Therefore, we have reduced to the case
\begin{align*}
K_x  &= \prod_{k=1}^{h_x} (x) + \mm^k = \prod_{k=1}^{h_x} (x,y^k) \\
K_y  &= \prod_{k=1}^{h_y} (y) + \mm^k = \prod_{k=1}^{h_y} (x^k,y).
\end{align*}
The statement about the length is already proved in the case $h_x=h_y$  (\Cref{ex:length_product_towers}). We prove the general case via an inductive argument. Let us assume, without loss of generality, that $e = h_x-h_y > 0$. We argue by induction on $e$. We set $h=h_y$ and we denote by $K_\bullet^{(h)}$ the tower
\[
K_\bullet ^{(h)} = \prod_{k=1}^{h} (\bullet) + \mm^k ,
\]
for $\bullet\in\{x,y\}$.

\smallbreak
\textbf{Step 1}. Assume $e=1$ (so that $h_x = h+1$). To prove \Cref{eqn:length_product_of_towers} in this case, it is enough to observe that the Ferrers diagram of the ideal
\[
K_x^{(h+1)}\cdot K_y^{(h)} = (x,y^{h+1})\cdot K_x^{(h)}\cdot K_y^{(h)} = (x)\cdot K_x^{(h)}\cdot K_y^{(h)} + (y^{a_h + h + 1}).
\]
is obtained from the Ferrers diagram of $K_x^{(h)}\cdot K_y^{(h)}$ by shifting it to the right by one position and adding a column of height $a_h+h+1$ to the left, where $a_h$ is defined as in \Cref{eqn:maximal_power_tower}. Thus the colength of $K_x^{(h_x)}\cdot K_y^{(h_y)} \subset \BC[x,y]$ is 
\begin{align*}
\ell_{\BC[x,y]/K_x^{(h+1)}\cdot K_y^{(h)}}
&= \ell_{\BC[x,y]/K_x^{(h)}\cdot K_y^{(h)}} + a_h + h + 1 \\
&=h^2 + 2\cdot \binom{h+2}{3} + \binom{h+1}{2} + 2h + 1,
\end{align*}
where we have exploited \Cref{ex:length_product_towers} in the last equality.
It is straightforward to check that this number agrees with
\[
\binom{h+3}{3} + \binom{h+2}{3} + (h+1)h = \ell_{\BC[x,y]/K_x^{(h_x)}} + \ell_{\BC[x,y]/K_y^{(h_y)}} + h_xh_y.
\]
So the base of the induction is proved.

\smallbreak
\textbf{Step 2}. Now we assume \Cref{eqn:length_product_of_towers} up to $e$ and we prove the formula for $e+1$ (so now $h=h_y$ and $h_x=h+e+1$). The ideal we have to compute the length of is
\[
I_{h,e+1} = K_x^{(h+e+1)}\cdot K_y^{(h)} = \left(x,y^{h+1}\right)\left(x,y^{h+2}\right)\cdots \left(x,y^{h+e}\right)\left(x,y^{h+e+1}\right)\cdot K_x^{(h)} \cdot K_y^{(h)}.
\]
By direct calculation, or by an application of \Cref{eqn:nc-tower-explicit} taken with $s=e$ and $i_k=h+k$, one finds that, for every $e>0$, there is an identity
\[
\left(x,y^{h+1}\right)\left(x,y^{h+2}\right)\cdots \left(x,y^{h+e}\right) = \left(x^iy^{(e-i)h+\binom{e+1-i}{2}} \,\,\big|\,\,0\leq i\leq e\right).
\]
We know, by the inductive hypothesis, that
\begin{align*}
\ell_{\BC[x,y]/I_{h,e}} 
&=\ell_{\BC[x,y]/K_x^{(h+e)}} + \ell_{\BC[x,y]/K_y^{(h)}} + (h+e)h \\
&= \binom{h+e+2}{3}+\binom{h+2}{3}+(h+e)h.
\end{align*}
As in \textbf{Step 1}, the Ferrers diagram of the ideal
\[
I_{h,e+1} = \left(x,y^{h+e+1}\right)\cdot I_{h,e} = (x)\cdot I_{h,e} + \left(y^{h+e+1}\right)
\]
is obtained from the Ferrers diagram of $I_{h,e}$ by shifting it to the right by one position, and adding a column of height 
\[
\binom{e+1}{2}+eh+\binom{h+1}{2}+h + (h+e+1)
\]
to the left. The number $\binom{e+1}{2}+eh+\binom{h+1}{2}+h$ is the height of the Ferrers diagram of $I_{h,e}$.
We obtain
\[
\ell_{\BC[x,y]/I_{h,e+1}} = \ell_{\BC[x,y]/I_{h,e}} + \binom{e+1}{2}+eh+\binom{h+1}{2}+h + (h+e+1) .
\]
It is now straightforward to check that this number agrees with
\begin{align*}
\binom{h+e+3}{3} + \binom{h+2}{3} + (h+e+1)h 
&= \ell_{\BC[x,y]/K_x^{(h+e+1)}} + \ell_{\BC[x,y]/K_y^{(h)}} + (h+e+1)h \\
&= \ell_{\BC[x,y]/K_x^{(h_x)}} + \ell_{\BC[x,y]/K_y^{(h_y)}} + h_x\cdot h_y. 
\end{align*}
So we have proved \Cref{eqn:length_product_of_towers}.

\smallbreak
We now move to proving \Cref{nu_product_of_towers}. This equation is implied by the more general \Cref{eqn:product-towers-general}, whose proof is essentially equivalent to that of \Cref{nu_product_of_towers}. Therefore we will give full details on the former and precise indications on how to prove the latter.

We shall use the shorthand notation $I = K_x \cdot K_y $. \Cref{LEMMATECH} implies, together with the usual toric construction, that there is a canonical isomorphism of $\BA^2$-schemes $ \varphi\colon Y\rightarrow \Bl_{I}\BA^2 $ where $Y$ is the toric variety with the following fan
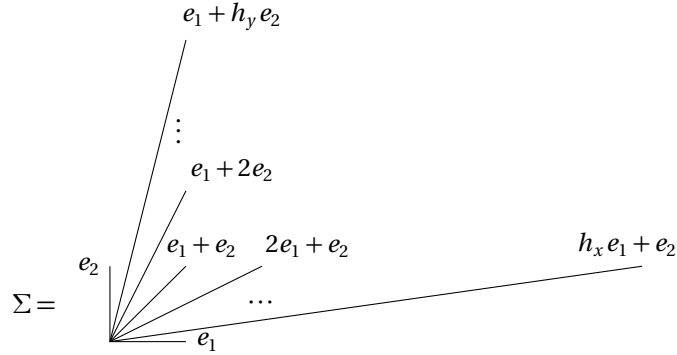
\begin{center}
    \begin{tikzpicture}
 \node at (-1,0.5) {$\Sigma=$};
 \draw[ -]  (0,0)--(7,1);
 \draw[ -] (0,1)--(0,0)--(1,0);
 \draw[ -] (1,1)--(0,0);
 \draw[ -] (2,1)--(0,0) ;
 \draw[ -] (1,2)--(0,0) ;
 \draw[ -] (0,0)--(1,4);
 \node[above] at (1.6,4) {\small $ e_1+h_ye_2 $};
 \node[above] at (0.9,2.5) {\small $ \vdots $};
 \node[above] at (1.6,2) {\small $ e_1+2e_2 $};
 \node[above] at (2.6,1) {\small $ 2e_1+e_2 $};
 \node[above] at (1.2,1) {\small $ e_1+e_2 $};
 \node[above] at (6.8,1) {\small $ h_xe_1+e_2 $};
 \node at (2,0.5) {\small $ \cdots $};
 \node[left] at (0,1) {\small $  e_2 $};
 \node[right] at (1,0) {\small $ e_1  $};
    \end{tikzpicture}
\end{center}
and the structure of $\BA^2$-scheme of $Y$ is induced by the identity map of the standard lattice $\BZ^2\subset\BR^2$.

Now, as in the proof of \Cref{TEOREMA1}, we can create a table encoding the contribution of each ideal $I_{x,i}=(x)+\mm^{i}$ and $I_{y,j}=(y)+\mm^{j}$, for $i=1,\ldots,h_x$ and $j=1,\ldots,h_y$, to the multiplicity of each irreducible component of the exceptional divisor of the blowup $\Bl_{I}\BA^2$. Such table has the following form. 
\begin{center}
    \begin{tikzpicture}[scale=0.9]
    \node at (-5.1,-9) {$I_{y,h_y}$};
    \node at (-5.3,-8) {$I_{y,h_y-1}$};
    \node at (-5.1,-7) {$\vdots$};
    \node at (-5.1,-6) {$I_{y,2}$};
    \node at (-5.45,-5) {$I_{y,1}=\mm$};
    \node at (-5.1,-4) {$I_{x,h_x}$};
    \node at (-5.3,-3) {$I_{x,h_x-1}$};
    \node at (-5.1,-2) {$\vdots$};
    \node at (-5.1,-1) {$I_{x,2}$};
    \node at (-5.45,0) {$I_{x,1}=\mm$};
    
    \draw (-4.6,-9.5)--(-4.6,2);
    \draw (-5.5,0.8)--(4.5,0.8);
    \node at (-4,0) {$1$};
    \node at (-3,0) {$1$};
    \node at (-2,0) {$\cdots$};
    \node at (-1,0) {$1$};
    \node at (0,0) {$1$};
    \node at (1,0) {$1$};
    \node at (2,0) {$\cdots$};
    \node at (3,0) {$1$};
    \node at (4,0) {$1$};

    \node at (-4,-1) {$1$};
    \node at (-3,-1) {$1$};
    \node at (-2,-1) {$\cdots$};
    \node at (-1,-1) {$1$};
    \node at (0,-1) {$1$};
    \node at (1,-1) {$2$};
    \node at (2,-1) {$\cdots$};
    \node at (3,-1) {$2$};
    \node at (4,-1) {$2$};
    
    \node at (-4,-2) {$\vdots$};
    \node at (-3,-2) {$\vdots$};
    \node at (-2,-2) {\reflectbox{$\ddots$}};
    \node at (-1,-2) {$\vdots$};
    \node at (0,-2) {$\vdots$};
    \node at (1,-2) {$\vdots$};
    \node at (2,-2) {$\ddots$};
    \node at (3,-2) {$\vdots$};
    \node at (4,-2) {$\vdots$};
    
    \node at (-4,-3) {$1$};
    \node at (-3,-3) {$1$};
    \node at (-2,-3) {$\cdots$};
    \node at (-1,-3) {$1$};
    \node at (0,-3) {$1$};
    \node at (1,-3) {$2$};
    \node at (2,-3) {$\cdots$};
    \node at (3,-3) {\footnotesize $h_x-1$};
    \node at (4,-3) {\footnotesize $h_x-1$};
    
    \node at (-4,-4) {$1$};
    \node at (-3,-4) {$1$};
    \node at (-2,-4) {$\cdots$};
    \node at (-1,-4) {$1$};
    \node at (0,-4) {$1$};
    \node at (1,-4) {$2$};
    \node at (2,-4) {$\cdots$};
    \node at (3,-4) {\footnotesize $h_x-1$};
    \node at (4,-4) {\footnotesize $h_x$};
    
    \node at (-4,-5) {$1$};
    \node at (-3,-5) {$1$};
    \node at (-2,-5) {$\cdots$};
    \node at (-1,-5) {$1$};
    \node at (0,-5) {$1$};
    \node at (1,-5) {$1$};
    \node at (2,-5) {$\cdots$};
    \node at (3,-5) {$1$};
    \node at (4,-5) {$1$};
 
    \node at (4,-6) {$1$};
    \node at (3,-6) {$1$};
    \node at (-2,-6) {$\cdots$};
    \node at (1,-6) {$1$};
    \node at (0,-6) {$1$};
    \node at (-1,-6) {$2$};
    \node at (2,-6) {$\cdots$};
    \node at (-3,-6) {$2$};
    \node at (-4,-6) {$2$};
    
    \node at (4,-7) {$\vdots$};
    \node at (3,-7) {$\vdots$};
    \node at (-2,-7) {\reflectbox{$\ddots$}};
    \node at (1,-7) {$\vdots$};
    \node at (0,-7) {$\vdots$};
    \node at (-1,-7) {$\vdots$};
    \node at (2,-7) {$\ddots$};
    \node at (-3,-7) {$\vdots$};
    \node at (-4,-7) {$\vdots$};
    
    \node at (4,-8) {$1$};
    \node at (3,-8) {$1$};
    \node at (-2,-8) {$\cdots$};
    \node at (1,-8) {$1$};
    \node at (0,-8) {$1$};
    \node at (-1,-8) {$2$};
    \node at (2,-8) {$\cdots$};
    \node at (-3,-8) {\footnotesize $h_y-1$};
    \node at (-4,-8) {\footnotesize $h_y-1$};
    
    \node at (4,-9) {$1$};
    \node at (3,-9) {$1$};
    \node at (-2,-9) {$\cdots$};
    \node at (1,-9) {$1$};
    \node at (0,-9) {$1$};
    \node at (-1,-9) {$2$};
    \node at (2,-9) {$\cdots$};
    \node at (-3,-9) {\footnotesize $h_y-1$};
    \node at (-4,-9) {\footnotesize $h_y$};
    
    \node at (2,1.5) {\footnotesize $\cdots$};
    \node at (-2,1.5) {\footnotesize $\cdots$};
    
    \draw (0.1,1.5)--(0.9,1.5);
    \draw (-0.1,1.5)--(-0.9,1.5);
    \draw (1.1,1.5)--(1.7,1.5);
    \draw (2.3,1.5)--(2.9,1.5);
    \draw (3.1,1.5)--(3.9,1.5);
    \draw (-1.1,1.5)--(-1.7,1.5);
    \draw (-2.3,1.5)--(-2.9,1.5);
    \draw (-3.1,1.5)--(-3.9,1.5);
    
		\draw (0,1.5) circle (0.1);
		\node[above] at (0,1.6) {\footnotesize $\mm $ };
		\node[below] at (0,1.4) {\footnotesize $-3 $ };
    
		\draw (1,1.5) circle (0.1);
		\node[above] at (1,1.6) {\footnotesize $I_{x,2} $ };
		\node[below] at (1,1.4) {\footnotesize $-2 $ };
    
		\draw (3,1.5) circle (0.1);
		\node[above] at (3,1.6) {\footnotesize $I_{x,h_x-1} $ };
		\node[below] at (3,1.4) {\footnotesize $-2 $ };
    
		\draw (4,1.5) circle (0.1);
		\node[above] at (4,1.6) {\footnotesize $I_{x,h_x} $ };
		\node[below] at (4,1.4) {\footnotesize $-1 $ };
    
		\draw (-1,1.5) circle (0.1);
		\node[above] at (-1,1.6) {\footnotesize $I_{y,2} $ };
		\node[below] at (-1,1.4) {\footnotesize $-2 $ };
    
		\draw (-3,1.5) circle (0.1);
		\node[above] at (-3,1.6) {\footnotesize $I_{y,h_y-1} $ };
		\node[below] at (-3,1.4) {\footnotesize $-2 $ };
    
		\draw (-4,1.5) circle (0.1);
		\node[above] at (-4,1.6) {\footnotesize $I_{y,h_y} $ };
		\node[below] at (-4,1.4) {\footnotesize $-1 $ };
    \end{tikzpicture}
\end{center}
Now, the Behrend number is the sum of all entries of the above table and, the equality
\[
\nu_{\BC[x,y]/I} = \nu_{\BC[x,y]/(K_x \cdot K_y}=\nu_{\BC[x,y]/K_x} +\nu_{\BC[x,y]/K_y}+2h_xh_y -h_x-h_y
\]
is obtained similarly as in the proof of \Cref{TEOREMA1}. Finally, \eqref{item1} is proved.

\smallbreak
In order to prove \eqref{item2}, we reduce to the simpler case $g=g_1=-g_2$ by applying the biholomorphism
\[
\begin{tikzcd}[row sep=tiny]
\BA^2 \arrow{r}{\sim} & \BA^2 \\
(x,y) \arrow[mapsto]{r} & \left( x-\frac{g_1(y)+g_2(y)}{2},y\right).
\end{tikzcd}
\]
In particular, in this case, we have $d=o(g)=o(g_1)=o(g_2)$. Consider the ideals
\begin{align*}
I_i&=(x)+\mm^{i} &\mathrm{for}\,\,i=1,\ldots,d-1,\\
J_j&=\left(x-\frac{j}{\abs{j}}g(y)\right)+\mm^{\abs{j}+d} &\mathrm{for}\,\,-h_1+d \le j  \le h_2-d\,\, \mathrm{and}\,\, j\neq 0,\\
J_0&=(x)+\mm^{d}.
\end{align*}
Then, we can write the ideal $K=K_1 \cdot K_2$ as $K=I\cdot J$, where
\[
I=\left(\underset{i=1}{\overset{d-1}{\prod}}I_i\right)^2, \qquad J=  \left(\underset{j=0}{\overset{h_2-d}{\prod}}J_j\right)\cdot\left( \underset{j=0}{\overset{h_1-d}{\prod}}J_{-j}\right).
\]
Let $\varepsilon_I\colon B_I=\Bl_I\BA^2\rightarrow \BA^2$ be the blowup map. Then,
\[
\Bl_{K}\BA^2=\Bl_{IJ}\BA^2\cong\Bl_{\varepsilon_I^{-1}(J)\cdot \OO_{B_I}}B_I
\]
where the isomorphism is over $\BA^2$. Notice that $B_I$ is a toric variety. A direct computation in toric geometry shows that:
\[
\varepsilon_I^{-1}(J)\cdot \OO_{B_I}=\widetilde{\OI}\cdot\widetilde{\OK}_1\cdot\widetilde{\OK}_2
\]
where $\widetilde{\OI}$ defines a Cartier divisor, while $\widetilde{\OK}_1$ and $\widetilde{\OK}_2$ are ideal sheaves of two $0$-dimensional schemes with the same support $\{p \}\subset B_I$ with the property that $p$ is a toric point. Consider a toric chart $U\subset B_I$ such that $U\cong\BA^2$ and $p\in U$ is the origin. If $a,b$ are toric coordinates on $U$, the two $\BC[a,b]$-modules $\widetilde{\OK}_i(U)$, for $i=1,2$ are complete towers of the form
\[
\widetilde{K}_1=\underset{j=1}{\overset{h_1-d+1}{\prod}}\left( a +\widetilde{g}(b)\right)+\mm_p^{j},\qquad \widetilde{ K}_2=\underset{j=1}{\overset{h_2-d+1}{\prod}}\left(a-\widetilde{g}(b)\right)+\mm_p^{j}, 
\]
where $o(\widetilde{g})=1$ and $\mm_p=(a,b)$ is the ideal of the origin of $U$. Now, the property $o(\widetilde{g})=1$ implies that there is a biholomorphism around $p$ which transforms the towers $\widetilde{K}_1$ and $\widetilde{K}_2$ in the monomial towers $K_x$ and $K_y$ in the first part of the statement. As a consequence, the Dynkin diagram of $\Bl_K\BA^2$ is the following.
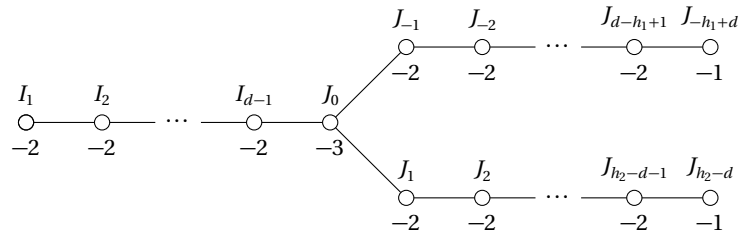
\begin{center}
    \begin{tikzpicture}
    
    \node at (2,1.5) {\footnotesize $\cdots$};
    \node at (7,2.5) {\footnotesize $\cdots$};
    \node at (7,0.5) {\footnotesize $\cdots$};
    
    \draw (0.1,1.5)--(0.9,1.5);
    \draw (1.1,1.5)--(1.7,1.5);
    \draw (2.3,1.5)--(2.9,1.5);
    \draw (3.1,1.5)--(3.9,1.5);
    
    \draw (4.071,1.571)--(4.929,2.429);
    \draw (4.071,1.429)--(4.929,0.571);
    
    \draw (5.1,2.5)--(5.9,2.5);
    \draw (6.1,2.5)--(6.7,2.5);
    \draw (7.3,2.5)--(7.9,2.5);
    \draw (8.1,2.5)--(8.9,2.5);
    
    \draw (5.1,0.5)--(5.9,0.5);
    \draw (6.1,0.5)--(6.7,0.5);
    \draw (7.3,0.5)--(7.9,0.5);
    \draw (8.1,0.5)--(8.9,0.5);

		\draw (0,1.5) circle (0.1);
		\node[above] at (0,1.6) {\footnotesize $I_1 $ };
		\node[below] at (0,1.4) {\footnotesize $-2 $ };
    
		\draw (1,1.5) circle (0.1);
		\node[above] at (1,1.6) {\footnotesize $I_{2} $ };
		\node[below] at (1,1.4) {\footnotesize $-2 $ };
    
		\draw (3,1.5) circle (0.1);
		\node[above] at (3,1.6) {\footnotesize $I_{d-1} $ };
		\node[below] at (3,1.4) {\footnotesize $-2 $ };
    
		\draw (4,1.5) circle (0.1);
		\node[above] at (4,1.6) {\footnotesize $J_0 $ };
		\node[below] at (4,1.4) {\footnotesize $-3 $ };
		\draw (0,1.5) circle (0.1);
	
		\node[above] at (5,2.6) {\footnotesize $J_{-1} $ };
		\node[below] at (5,2.4) {\footnotesize $-2 $ };
    
		\draw (5,2.5) circle (0.1);
		\node[above] at (6,2.6) {\footnotesize $J_{-2} $ };
		\node[below] at (6,2.4) {\footnotesize $-2 $ };
    
		\draw (6,2.5) circle (0.1);
		\node[above] at (8,2.6) {\footnotesize $J_{d-h_1+1} $ };
		\node[below] at (8,2.4) {\footnotesize $-2 $ };
    
		\draw (8,2.5) circle (0.1);
		\node[above] at (9,2.6) {\footnotesize $J_{-h_1+d} $ };
		\node[below] at (9,2.4) {\footnotesize $-1 $ };
		\draw (9,2.5) circle (0.1);
		
		\node[above] at (5,0.6) {\footnotesize $J_1 $ };
		\node[below] at (5,0.4) {\footnotesize $-2 $ };
    
		\draw (5,0.5) circle (0.1);
		\node[above] at (6,0.6) {\footnotesize $J_{2} $ };
		\node[below] at (6,0.4) {\footnotesize $-2 $ };
    
		\draw (6,0.5) circle (0.1);
		\node[above] at (8,0.6) {\footnotesize $J_{h_2-d-1} $ };
		\node[below] at (8,0.4) {\footnotesize $-2 $ };
    
		\draw (8,0.5) circle (0.1);
		\node[above] at (9,0.6) {\footnotesize $J_{h_2-d} $ };
		\node[below] at (9,0.4) {\footnotesize $-1 $ };
		\draw (9,0.5) circle (0.1);
    \end{tikzpicture}
\end{center}
At this point, finding the Behrend number is a simple calculation analogous to those made in the proof of \eqref{item1} and we leave it to the reader.
\end{proof}

\begin{remark}\label{productnoncomplete}
Similarly as we have done in \Cref{blownoncompl}, the above proposition can be easily generalised to non-complete towers. For example, in the easiest case when
\[
K_x=\underset{k=1}{\overset{s_x}{\prod}} (x )+\mm^{i_k},\qquad  K_y=\underset{k=1}{\overset{s_y}{\prod}} (y )+\mm^{j_k},
\]
for $1\le i_1<\cdots <i_{s_x} $ and $1\le j_1<\cdots <j_{s_y} $, are two monomial non-complete towers. Then, a fan $\Sigma$ in $ \BR^2$ of the toric variety $X=\Bl_{K_x K_y}\BA^2$ is the following.
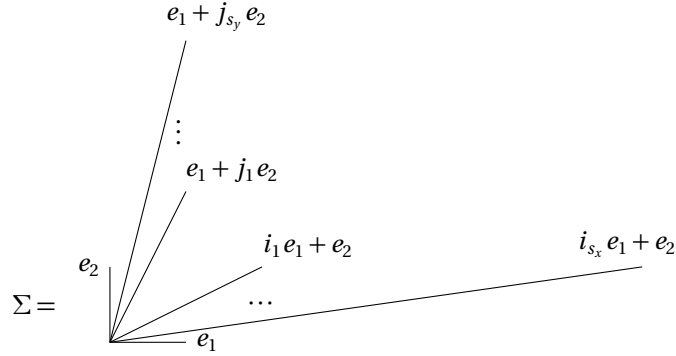
\begin{center}
    \begin{tikzpicture}
 \node at (-1,0.5) {$\Sigma=$};
 \draw[ -]  (0,0)--(7,1);
 \draw[ -] (0,1)--(0,0)--(1,0);
 \draw[ -] (2,1)--(0,0);
 \draw[ -] ( 0,0)--(1,4);
 \draw[ -] (1,2)--(0,0) ;
 \node[above] at (1.4,4) {\small $ e_1+ j_{s_y}e_2 $};
 \node[above] at (0.9,2.5) {\small $ \vdots $};
 \node[above] at (1.6,2) {\small $ e_1+j_1 e_2 $};
 \node[above] at (2.6,1) {\small $ i_1 e_1+e_2 $};
 \node[above] at (6.8,1) {\small $ i_{s_x}e_1+e_2 $};
 \node at (2,0.5) {\small $ \cdots $};
 \node[left] at (0,1) {\small $  e_2 $};
 \node[right] at (1,0) {\small $ e_1  $};
    \end{tikzpicture}
\end{center}
Moreover, the Behrend number $\nu_{\BC[x,y]/K_x\cdot K_y}$ can be computed similarly as described in \Cref{blownoncompl}.

Notice that, if the ray $\rho=e_1+e_2$ belongs to the fan $\Sigma$ then, $X$ has only singularity of type $A_n$. While, if $\rho=e_1+e_2\notin \Sigma$ then, $X$ has an isolated singularity of different kind associated to the cone $<e_1+j_1e_2,i_1e_1+e_2>$. In particular, such singularity is never Gorenstein (see {\cite[Prop.~10.1.6.]{COX}}), while the $A_n$ singularities are always Gorenstein. 
\end{remark}

%%%%%%%%%%%%%%%%%%%%%%%%%%%%%%%%%%%%%%%%%%%%%%%%%%%%%%%%%%%%%%%%%%%
\subsection{Behrend function and Hilbert--Samuel strata}
By work of Briançon \cite{BRIANCON} and Iarrobino \cite{IARRO1}, the punctual Hilbert scheme $\Hilb^n(\BA^2)_0 \subset \Hilb^n(\BA^2)$, parametrising subschemes entirely supported at the origin, contains the locus of the curvilinear schemes $\mathscr C_n$ as a Zariski open (and hence dense) subset. Moreover, the complement $\Hilb^n(\BA^2)_0 \setminus \mathscr C_n$ can be stratified according to the Hilbert--Samuel function, also called the \emph{type}, of fat points (see \cite{IARRO1} for a definition). Let us consider the set-theoretic map
\[
\beta_n\colon \Hilb^n(\BA^2)_0 \to \BZ, \qquad [I] \mapsto \nu_{\BC[x,y]/I}.
\]
We know by \Cref{ex:curvilinear} this function is constantly equal to $n$ on $\mathscr C_n \subset \Hilb^n(\BA^2)_0$. Continuity of $\beta_n$ is of course out of question. In fact, the following example shows that $\beta_n$ is in general \emph{not} even constant on the Hilbert--Samuel strata.

\begin{example}
Consider the two ideals
\[
I=(xy,x^3-y^3),\qquad J=(xy,x^4,y^3).
\]
Then, $\length(\BC[x,y]/I) = 6 = \length(\BC[x,y]/J)$, and since $I$ is a complete intersection we have also 
\[
\nu_{\BC[x,y]/I} = 6
\]
by \Cref{ex:lci}. However, this is not the case for the Behrend number of the ideal $J$. Indeed, $J =(x,y^2)\cdot (x^3,y)$ is a product of curvilinear ideals and hence, in particular, a product of two towers. In order to compute the Behrend number of the ideal $J$ we can proceed as suggested in \Cref{productnoncomplete}. Alternatively, we will show in the next section (see \Cref{prdnoncomplenne}) an algorithm to compute the Behrend number of such kind of ideals. By applying it, one finds
\[
\nu_{\BC[x,y]/J} = 7 > 6.
\]
Finally, we observe that they have the same type 
\[
T(I)=T(J)=(1,2,2,1,0,0,0),
\]
and hence, they belong to the same Hilbert--Samuel stratum of $\Hilb^6(\BA^2)_0$.
\end{example}

%%%%%%%%%%%%%%%%%%%%%%%%%%%%%%%%%%%%%%%%%%%%%%%%%%%%%%%%%%%%%%%%%%%
%%%%%%%%%%%%%%%%%%%%%%%%%%%%%%%%%%%%%%%%%%%%%%%%%%%%%%%%%%%%%%%%%%%
\section{An algorithm for the Behrend number of a product of towers}\label{sec:Algorithm}

In the previous section we computed the Behrend number of an arbitrary tower $K \subset \BC[x,y]$ and of the product of two towers. In this section we explain an algorithmic procedure to perform the calculation for an arbitrary (finite) product of towers. This produces a huge number of examples of Behrend numbers of non-monomial schemes, analogously to \Cref{TORIC PROD TOWERS}\,(2).

%%%%%%%%%%%%%%%%%%%%%%%%%%%%%%%%%%%%%%%%%%%%%%%%%%%%%%%%%%%%%%%%%%%
\subsection{Products of towers: the complete case}
We already observed (cf.~\Cref{ex:Dynkin}) that the nodes of the Dynkin diagram attached to the blowup $\Bl_K\BA^2$ along a complete tower $K\subset \BC[x,y]$ can be labeled by ideals: see \Cref{fig:Dynkin-with-ideals}.

In this section we shall construct more general Dynkin diagrams, each associated with a product of towers; in this subsection we focus on the complete case. In this context,
all the Dynkin diagrams under consideration will have a node $c_\mm$ associated to the maximal ideal $\mm = (x,y) \subset \BC[x,y]$, and all the other nodes will be connected to $c_\mm$ by a sequence of edges. In this section we will use the following terminology: we will say that $c_\mm$ has \emph{level} $1$, while the level of each other node $c$ is defined as 
\[
\textrm{level}(c) = 1 + \big| \textrm{number of edges separating } c \textrm{ from } c_\mm \big|.
\]
For example, one can show that, if $I$ is the ideal 
\begin{equation}\label{eqn:ideal31937}
I=\mm \cdot ((x)+\mm^2)\cdot ((y)+\mm^2)\cdot ((x+y)+\mm^2)\cdot ((x+y)+\mm^3)
\end{equation}
then, the blowup $\Bl_I\BA^2$ is smooth and the Dynkin diagram of $\exc(\Bl_I\BA^2)$ is as in \Cref{fig:dynkin-example}.

\begin{figure}[h!]
\begin{tikzpicture}
\draw (0,0) circle (0.1);
		\node[right] at (0,0) {\footnotesize $\mm$ };
		\node[left] at (0,0) {\footnotesize $-4 $ };

\draw (0,1.5) circle (0.1);
		\node[above] at (0,1.5) {\footnotesize $(x)+\mm^2$ };
		\node[right] at (0,1.5) {\footnotesize $-1 $ };

\draw (1.5,1.5) circle (0.1);
		\node[above] at (1.5,1.5) {\footnotesize $(y)+\mm^2$ };
		\node[right] at (1.5,1.5) {\footnotesize $-1 $ };
		
\draw (-1.5,3) circle (0.1);
		\node[above] at (-1.5,3) {\footnotesize $(x+y)+\mm^3$ };
		\node[left] at (-1.5,3) {\footnotesize $-1 $ };

\draw (-1.5,1.5) circle (0.1);
		\node[left] at (-1.5,1.5) {\footnotesize $(x+y)+\mm^2$ };
		\node[right] at (-1.5,1.5) {\footnotesize $-2 $ };
		
\draw (-0.071,0.071)--(-1.429,1.429);
    \draw (-1.5,2.9)--(-1.5,1.6);
    \draw (0.071,0.071)--(1.429,1.429);
    \draw (0,0.1)--(0,1.4);
    \node at (4,0) {\textbf{LEVEL 1}};
    \node at (4,1.5) {\textbf{LEVEL 2}};
    \node at (4,3) {\textbf{LEVEL 3}};
    \draw[dashed] (-3,0.75)--(5,0.75);
    \draw[dashed] (-3,-0.75)--(5,-0.75);
    \draw[dashed] (-3,2.25)--(5,2.25);
    \draw[dashed] (-3,3.75)--(5,3.75);
\end{tikzpicture}
\caption{The Dynkin diagram of the ideal \eqref{eqn:ideal31937}.}
\label{fig:dynkin-example}
\end{figure}
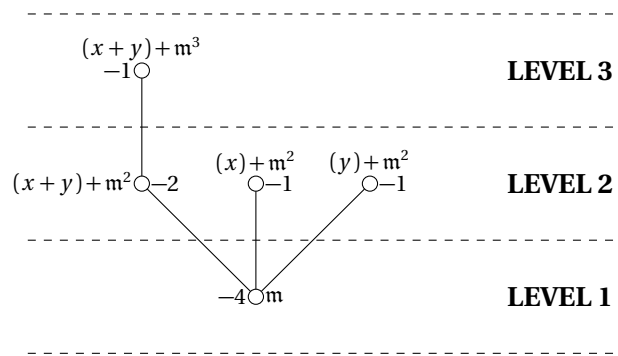

Note that, by construction of the Dynkin diagram associated to the exceptional locus of the blowup with center the product of complete towers, a necessary condition for two nodes to be connected by an edge is that their levels differ by one unit. We will say that a node $c_1$ is a \emph{descendant} of another node $c_2$  (or that $c_2$ is an \emph{ancestor} of $c_1$) if $c_1$ and $c_2$ are connected by a sequence of edges and the level of $c_1$ is greater than the level of $c_2$.

\smallbreak
Here is the general setup. Consider a set of complete towers
\[
\mathcal{T}=\Set{ K_{x,1},\ldots,K_{x,s_x},K_{y,1},\ldots,K_{y,s_y}}
\]
of the form
\[
K_{x,i}=\underset{k=1}{\overset{a_{x,i}}{\prod}}(x+g_{x,i}(y))+\mm^{k},\qquad K_{y,j}=\underset{k=1}{\overset{a_{y,j}}{\prod}}(y+g_{y,j}(x))+\mm^{k},
\]
 for $i=1,\ldots,s_x$ and $j=1,\ldots,s_y$, and their product
\[
T=\underset{K\in\mathcal{T}}{\prod}K.
\]
Then, the blowup $\Bl_T\BA^2$ is smooth and there is an algorithm to construct the Dynkin diagram of the exceptional locus of the map
\[
\varepsilon\colon \Bl_T\BA^2\rightarrow \BA^2.
\]
Moreover, we will show how to compute, starting from such diagram, the Behrend number of the fat point $\BC[x,y]/T$. We briefly describe the algorithm.

Let $h>0$ be the maximum of the heights of the towers in $\mathcal{T}$, i.e. 
\[
h=\max\Set{a_{x,i},a_{y,j}|i=1,\ldots,s_x\mbox{ and }j=1,\ldots,s_y}.
\]
Consider the $h$ equivalence relations on $\mathcal{T}$ defined, for $r=1,\ldots,h$, by
\[
K_{*,i}\sim_r K_{\bullet,j} \quad \Leftrightarrow \quad 
\begin{cases}
r=1, \textrm{or}\\
1<r\le\min\{a_{*,i},a_{\bullet,j} \},\ *=\bullet\mbox{ and }g_{*,i}\equiv g_{\bullet,j} \mod \mm^{r}, \textrm{or} \\
r>\max\{a_{*,i},a_{\bullet,j} \},
\end{cases}
\]
and call classes in excess the classes of the form $[K_{*.i}]^{\sim_r}$ for $r>a_{*,i}$. In particular, there is at most one class in excess for any $r=1,\ldots,h$.

We are now ready to construct the underlying graph of the Dynkin diagram for $\exc(\Bl_T\BA^2)$. We put one node at the first level, namely the node $c_\mm$ corresponding to the unique class in $\Calt/\sim_1$, and, at the $i$-th level, we put a node for each element in $\mathcal{T}/\sim_i$ excluding the possible class in excess. Finally, we add an edge joining the node associated to some class $[K_1]^{\sim_{r_1}}$ to the node associated to some class $[K_2]^{\sim_{r_2}}$ if and only if
\[
\abs{r_1-r_2}= 1\mbox{ and }[K_1]^{\sim_{r_1}}\cap [K_2]^{\sim_{r_2}}\not=\emptyset. 
\]
The self-intersection at each node of level strictly greater then one is given by
\[
-\big|\{\mbox{edges issuing from the node}\}\big|
\]
while, the node $c_{\mm}$ is labeled by the self intersection
\[
-\big|\{ \mbox{edges issuing from }c_\mm\}\big|-1.
\]
This allows to compute the multiplicities of the irreducible components of the exceptional divisor of $\Bl_T\BA^2$. We now briefly explain how to compute Behrend numbers. 

Let $I$ be an ideal appearing as a factor of some tower in $\Calt$, an let $c_I$ be the node of the Dynkin diagram associated to $I$ (see \Cref{ex:Dynkin}). Let also $c$ be any node in the Dynkin diagram and let $D_c$ be the corresponding irreducible component of $\exc(\Bl_T\BA^2)$. Then, the contribution of the ideal $I$ to the multiplicity of the exceptional divisor along the component $D_c$ is given by
\[
\begin{cases}
\mbox{level of } c_I &\mbox{if $c=c_I$ or if $c$ is a descendant of $c_I$,}\\
\mbox{level of } c &\mbox{if $c$ is a ancestor of $c_I$,}\\
1&\mbox{otherwise.}
\end{cases}
\]
Now, summing up all these contributions over all pairs $(I,c)$, one obtains the Behrend number of $\BC[x,y]/T$.

\subsection{Products of towers: the non-complete case}\label{prdnoncomplenne}
Again, the complete case helps us in understanding the non-complete case. 

Let
\[
\mathcal{T}=\Set{ K_{x,1},\ldots,K_{x,s_x},K_{y,1},\ldots,K_{y,s_y}}
\]
be a set of towers of the form
\[
K_{x,i}=\underset{k=1}{\overset{a_{x,i}}{\prod}}(x+g_{x,i}(y))+\mm^{i_{k,i}},\qquad  K_{y,j}=\underset{k=1}{\overset{a_{y,j}}{\prod}}(y+g_{y,j}(x))+\mm^{j_{k,j}},
\]
for $i=1,\ldots,s_x$, $j=1,\ldots,s_y$, $1\le i_{1,i}<\cdots< i_{a_{x,i},i}$ and $1\le j_{1,j}<\cdots< j_{a_{y,j},j}$, and let
\[
T=\underset{K\in\mathcal{T}}{\prod}K
\]
be their product. Consider also the set of complete towers
\[
\widetilde{\mathcal{T}}=\Set{ \widetilde{K}_{x,1},\ldots,\widetilde{K}_{x,s_x},\widetilde{K}_{y,1},\ldots,\widetilde{K}_{y,s_y}}
\]
defined by
\[
\widetilde{K}_{x,i}=\underset{k=1}{\overset{i_{a_{x,i},i}}{\prod}}(x+g_{x,i}(y))+\mm^{k}, \qquad  \widetilde{K}_{y,j}=\underset{k=1}{\overset{j_{a_{y,j},j}}{\prod}}(y+g_{y,j}(x))+\mm^{k}
\]
for $i=1,\ldots,s_x$, $j=1,\ldots,s_y$, and let us set
\[
\widetilde{T}=\underset{K\in\widetilde{\mathcal{T}}}{\prod}K.
\]
Let us also call $\varepsilon\colon B=\Bl_T\BA^2\rightarrow \BA^2$ and $\widetilde{\varepsilon}\colon \widetilde{B}=\Bl_{\widetilde{T}}\BA^2\rightarrow \BA^2$ the blowup maps. Then, $\widetilde{\varepsilon}^{-1}(T)\cdot \OO_{\widetilde{B}}$ defines a Cartier divisor. Hence, there is a canonical $\BA^2$-morphism $\varphi\colon \widetilde{B}\rightarrow B.$
Let 
\[
\exc(\widetilde{B})=\widetilde{C}_1\cup\cdots \cup \widetilde{C}_\alpha \subset \widetilde B, \qquad \exc({B})={C}_1\cup\cdots \cup  {C}_\beta \subset B
\]
be the decompositions of $\exc(\widetilde{B})$ and $\exc({B})$ into irreducible components. Clearly, we have $\alpha\ge \beta$. Then, as per \Cref{ZMT}, the morphism $\varphi$ must contract some of the curves in $\exc(\widetilde{B})$ and it is an isomorphism when restricted to the complement of the contracted curves. In particular, up to reordering the components of $\exc(\widetilde{B})$, the map
\[
\varphi\big|_{\widetilde{B}\smallsetminus(\widetilde{C}_{\beta+1}\cup\widetilde{C}_{\beta+2}\cup \cdots \cup\widetilde{C}_{\alpha})}\colon \widetilde{B}\smallsetminus(\widetilde{C}_{\beta+1}\cup\widetilde{C}_{\beta+2}\cup \cdots \cup\widetilde{C}_{\alpha}) \rightarrow B\smallsetminus\varphi(\widetilde{C}_{\beta+1}\cup\widetilde{C}_{\beta+2}\cup \cdots \cup\widetilde{C}_{\alpha})
\]
is an isomorphism which restricts, for $j=1,\ldots,\beta$, to an isomorphism
\[
\widetilde{C}_j\smallsetminus(\widetilde{C}_{\beta+1}\cup\widetilde{C}_{\beta+2}\cup \cdots \cup\widetilde{C}_{\alpha}) \rightarrow C_j\smallsetminus\varphi(\widetilde{C}_{\beta+1}\cup\widetilde{C}_{\beta+2}\cup \cdots \cup\widetilde{C}_{\alpha}).
\]
Therefore, for $j=1,\ldots,\beta$, the multiplicity of the exceptional divisor $E_T\BA^2=V(\varepsilon^{-1}(T)\cdot \OO_B)$ along $C_j$ equals the multiplicity of the Cartier divisor defined by $\widetilde{\varepsilon}^{-1}(T)\cdot \OO_{\widetilde{B}}$ along $\widetilde{C}_j$.

This observation allows one to compute the Behrend number of any (finite) product of towers. 

%%%%%%%%%%%%%%%%%%%%%%%%%%%%%%%%%%%%%%%%%%%%%%%%%%%%%%%%%%%%%%%%%%%
\subsection{Examples}
In general, one does not need to pass trough the blowup of a complete tower to compute the Behrend number of some tower $K$. The convenience in introducing the complete towers even in the non-complete case is in the computations. The following example should explain the situation.

\begin{example}
Consider the ideal $I=(x,y^2)\subset \BC[x,y]$ and the blowup $\varepsilon\colon B_I:=\Bl_I\BA^2\rightarrow \BA^2$. Then, $B_I$ is described in \Cref{prop:A_n-sing} and it is a toric surface covered by two charts: the first, $U_1$, is isomorphic to the affine quadric cone, whereas the second, $U_2$, is smooth and isomorphic to $\BA^2$. In particular, $U_1$ is the affine toric variety described by the cone generated by the rays of primitive vectors $,\rho_1=e_2, \rho_2=2e_1+e_2$ and, by standard toric geometry (see {\cite[\S\,3.1]{COX}}), we have an isomorphism of $\BA^2$-schemes 
\[
U_1\cong \Spec\BC[x,x^{-1}y^2,y].
\]
If we introduce the variables $a=x,b=x^{-1}y^2,c=y$, then the restriction to $U_1$ of the blowup map is associated to the $\BC$-algebra homomorphism
\[
\begin{tikzcd}[row sep=tiny]
\BC[x,y] \arrow{r} & S=\BC[a,b,c]/(ab-c^2) \\
x \arrow[mapsto]{r} & a\\
y \arrow[mapsto]{r} & c.
\end{tikzcd}
\]
Computing $\varepsilon^{-1}(I)\cdot \OO_{U_1}$, one finds 
\[
\varepsilon^{-1}(I)\cdot \OO_{U_1}=(a,c^2)=(a,ab)=(a) \subset S,
\]
so that one would be tempted to conclude that $\nu_{\BC[x,y]/I}=1$. This is, in fact, incorrect, because, in the local ring $\OO_{U_1,\exc(\varepsilon)}\cong S_{(a,c)}$ the function $b$ is invertible and we have
$\nu_{\BC[x,y]/I}=\ord_{\exc(\varepsilon)}(a)=\ord_{\exc(\varepsilon)}(c^2b^{-1})=\ord_{\exc{(\varepsilon)}}(c^2)=2\cdot\ord_{\exc(\varepsilon)}(c)=2$, as expected.

This complication never occurs in the case of smooth surfaces.
\end{example}

Even though the procedure described in \Cref{prdnoncomplenne} above is quite straightforward, one may need to use a computer to actually compute the Behrend number of an arbitrary (finite) product of towers. Computing the length of a product of towers, on the other hand, can often be quite complicated. Below we show an example that can be computed by hands. %Fortunately, some tricks can be of some help. Below we show an example.

\begin{example}
Let $K=\prod_{1\leq k\leq s}(x)+\mm^{i_k}=\prod_{1\leq k\leq s}(x,y^{i_k})$ be a monomial tower with $1<i_1<\cdots< i_{s}$, and set $J=K\cdot \mm^n$ for some integer $n>0$. Then, the following formula holds
\[
J=(x,y)^{n+s}\cap\left(x^s,\left.x^{s-i}y^{n+\underset{k=1}{\overset{i}{\sum}}i_{k}}\ \right|\ i=1,\ldots,s\right).
\]
As the ideal has now taken on a more pleasant form, the following formulas can easily be obtained:
\begin{align*}
    \ell_{\BC[x,y]/J} &=\ell_{\BC[x,y]/K}+\frac{n(n+1)+2ns}{2}, \\
    \nu_{\BC[x,y]/J} &= \nu_{\BC[x,y]/K}+sn+n+s.
\end{align*}
\end{example}

%%%%%%%%%%%%%%%%%%%%%%%%%%%%%%%%%%%%%%%%%%%%%%%%%%%%%%%%%%%%%%%%%%%
%%%%%%%%%%%%%%%%%%%%%%%%%%%%%%%%%%%%%%%%%%%%%%%%%%%%%%%%%%%%%%%%%%%
\section{The general normal case}\label{sec:monomial-normal}
In this section, we will completely solve the problem of computing the Behrend number $\nu_{\BC[x,y]/I}$ for a normal monomial ideal $I \subset \BC[x,y]$. Moreover, in \Cref{toricnormal} we will give a toric description of the blowup $\Bl_I\BA^2$ and we will prove a factorisation theorem (\Cref{normal-factorisation}) that allows one to write the ideal $I$ uniquely as a product of powers of much easier ideals, namely the normalisations of the monomial complete intersection ideals $(x^h,y^k)$.

%%%%%%%%%%%%%%%%%%%%%%%%%%%%%%%%%%%%%%%%%%%%%%%%%%%%%%%%%%%%%%%%%%%
\subsection{The key example}\label{example missed}
Consider the ideals 
\begin{equation}\label{ideals_IJ}
  I=x^2+\mm^3=(x^2,xy^2,y^3),\qquad J=\mm\cdot((x)+\mm^2)\cdot I. 
\end{equation}
In particular $\ell_{\BC[x,y]/I}=5$ and $\ell_{\BC[x,y]/J} = 14$.

We want to perform, for the ideals $I$ and $J$, the same analysis that we did for the ideals in the previous sections. We have (for instance via \Cref{blowaffine})
\[
X_I = \Bl_I\BA^2 = 
\Set{
((x,y),[w_0:w_1:w_2])\in\BA^2\times \BP^2 | 
\begin{array}{c}
  xw_1=yw_2 \\
  y^2w_0=xw_2 \\
  yw_0w_1=w_2^2
\end{array}
}
\]
and the exceptional locus is
\[
D_I=\exc(X_I)=\Set{((x,y),[w_0:w_1:w_2])\in X_I|x=y=w_2=0}\cong\BP^1.
\]
In order to study the variety $X_I$ we cover it with the three affine charts
\[
X_{I,i}=X_I\cap( \BA^2\times \{ w_i\not=0 \}),\quad i=0,1,2.
\]
We will also denote by $D_{I,i}$, for $i=0,1$, the chart on $D_I$ given by $D_{I,i}=D_I\cap X_{I,i}$.

Now, $X_{I,2}$ is smooth, while $X_{I,0}$ and $X_{I,1}$ have each an isolated singular point $p_i\in D_{I,i}$ for $i=0,1$. The singular charts have the form
\[
X_{I,0}\cong \BA^2/G_0,\quad X_{I,1}\cong \BA^2/G_1,
\]
where, given a primitive third root of unity $\xi_3\in\BC^\times$, the groups $G_0,G_1$ are
\[
G_0=\left<\begin{pmatrix} \xi_3 &0 \\ 0 & \xi_3
\end{pmatrix}
\right> \subset \GL(2,\BC),\qquad G_1=\left<\begin{pmatrix}-1 &0 \\ 0 & -1
\end{pmatrix}
\right>\subset \SL(2,\BC).
\]
As a consequence, the surface $X_I$ is normal. This is a general fact about quotient surface singularities and it can also be deduced from \Cref{villareal}.

Let $\varepsilon_I\colon X_I\rightarrow \BA^2$ be the blowup map and let
\[
\varphi\colon \widetilde{X}_I\rightarrow X_I
\]
be the minimal resolution of $X_I$, i.e.~$\widetilde{X}_I$ is a smooth surface and $\varphi$ is a projective birational morphism which does not contract any rational  $(-1)$-curve. It is well known that the variety $\widetilde{X}_I$ is obtained by blowing up the singular points $p_0$ and $p_1$. Let us also denote by $\widetilde{D}_I\subset\exc(\varepsilon_I\circ \varphi)\subset \widetilde{X}_I$ the strict transform of $D_I$, i.e.~the Zariski closure of $\varphi^{-1}(D_I\smallsetminus\{p_0,p_1\})$. Given the description of the singularities, we have that $\exc(\varphi)$ is a disjoint union of two smooth projective rational curves $L_0,L_1$, corresponding respectively to $p_0$ and $p_1$, each of which intersects the line $\widetilde{D}_I$ at a point. Furthermore, the self-intersections of $L_0$ and $L_1$ are respectively 
\[
L_0^2=-3,\quad L_1^2=-2.
\]
Notice that the map $\varepsilon_I\circ \varphi$ is a projective birational morphism of smooth surfaces and hence, by classical theory of surfaces (see \cite[Ch.~III]{BARTHECC}), it follows that $\widetilde{X}_I$ contains a smooth rational projective $(-1)$-curve, and the only possible such curve is $\widetilde{D}_I$. The Dynkin diagram attached to $\set{\widetilde{D}_I,L_0,L_1}$ is depicted in \Cref{fig:dynkin89}.

\begin{figure}[h!]
\begin{tikzpicture}
    
    \draw (0.1,1.5)--(0.9,1.5);
    \draw (1.1,1.5)--(1.9,1.5);

    \draw (0,1.5) circle (0.1);
	\node[above] at (0,1.6) {\footnotesize $L_0 $ };
	\node[below] at (0,1.4) {\footnotesize $-3 $ };
    
	\draw (1,1.5) circle (0.1);
	\node[above] at (1,1.6) {\footnotesize $\widetilde{D}_I $ };
	\node[below] at (1,1.4) {\footnotesize $-1 $ };
    
	\draw (2,1.5) circle (0.1);
	\node[above] at (2,1.6) {\footnotesize $L_1 $ };
	\node[below] at (2,1.4) {\footnotesize $-2 $ };
\end{tikzpicture}
\caption{The Dynkin diagram attached to $\set{ \widetilde{D}_I,L_0,L_1}$.}
\label{fig:dynkin89}
\end{figure}

We claim that there is a canonical isomorphism of $\BA^2$-schemes between $X_J=\Bl_J\BA^2$ and $\widetilde{X}_I$, where $J$ is as in \eqref{ideals_IJ}. Thanks to \Cref{LEMMATECH}, we know that there is a canonical morphism of $\BA^2$-schemes 
\[
\psi\colon X_J\rightarrow X_I.
\]
This follows from the existence of the isomorphism of $\BA^2$-schemes
\[
\begin{tikzcd}
X_J \arrow{r}{\sim} & \Bl_{\varepsilon^{-1}(I)\cdot \OO_B}B
\end{tikzcd}
\]
where $\varepsilon\colon B\rightarrow \BA^2$ is the blowup with center the ideal $\mm\cdot(x,y^2)$, together with the universal property of blowups. In fact, we observe that there are canonical isomorphisms of $\BA^2$-schemes
\[
\begin{tikzcd}
& Z\arrow[swap]{dl}{\widetilde{\vartheta}}\arrow{dr}{\vartheta_J} & \\
\widetilde{X}_I & & X_J
\end{tikzcd}
\]
where $Z$ is the result of an iterated blowup, namely
\[
\begin{tikzcd}
Z \arrow{r}{\mu} & B \arrow{r}{\varepsilon} & \BA^2 
\end{tikzcd}
\]
where $\mu$ is the blowup of $B$ with center the intersection point of the two irreducible components of $\exc(\varepsilon) $.

In order to construct the isomorphisms $\widetilde{\vartheta}$ and $\vartheta_J$, we start by noticing that the surface $Z$ just described is the toric variety associated to the fan $\Sigma$ in $ \BR^2$ shown below.
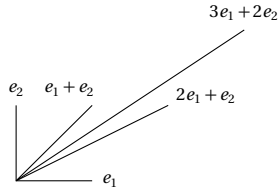
\begin{center}
\begin{tikzpicture}
\draw (2,1)--(0,0);
\draw  (0,0)--(3,2);
\draw (0,1)--(0,0)--(1,0);
\draw (1,1)--(0,0);
\node[right] at (1,0) {\tiny $e_1$};
\node[above] at (0,1) {\tiny $e_2$};
\node[above] at (0.7,1) {\tiny $e_1+e_2$};
\node[above] at (3,2) {\tiny $3e_1+2e_2$};
\node[above] at (2.5,0.9) {\tiny $2e_1+e_2$};
\end{tikzpicture}
\end{center}
Now, the same computations as those we did in the previous section show that $(\varepsilon\circ\mu )^{-1}(I)\cdot \OO_{Z}$ and $(\varepsilon\circ\mu )^{-1}(J)\cdot \OO_{Z}$ define Cartier divisors on $Z$ and, as a consequence, there exist canonical $\BA^2$-morphisms ${\vartheta}_I\colon Z\rightarrow X_I$ and $\vartheta_J\colon Z\rightarrow X_J$. Moreover, $\vartheta_I$ does not contract any $(-1)$-curve and, as a consequence, it lifts to an isomorphism $\widetilde{\vartheta}\colon Z\rightarrow \widetilde{X}_I$ because of the universal property of the minimal resolution (see {\cite[Thm.~(6.2)]{BARTHECC}}). The map $\vartheta_J$ is an isomorphism because $\varepsilon^{-1}(I)\cdot \OO_{B}$ is the product of a principal (Cartier) ideal sheaf times the ideal sheaf of the reduced intersection point of the two irreducible components of $\exc(\varepsilon)$.

As a consequence, $X_J$ and $\widetilde{X}_I$ are canonically isomorphic and, if we label their Dynkin diagram (see \Cref{fig:dynkin89}) as explained in \Cref{ex:Dynkin}, then we obtain the following diagram. 
\begin{center}
    \begin{tikzpicture}
\draw (0.1,1.5)--(0.9,1.5);
\draw (1.1,1.5)--(1.9,1.5);
\draw (0,1.5) circle (0.1);
	\node[above] at (0,1.6) {\footnotesize $\mm $ };
	\node[below] at (0,1.4) {\footnotesize $-3 $ };
    
\draw (1,1.5) circle (0.1);
	\node[above] at (1,1.6) {\footnotesize $I $ };
	\node[below] at (1,1.4) {\footnotesize $-1 $ };
    
\draw (2,1.5) circle (0.1);
	\node[above] at (2,1.6) {\footnotesize $(x)+\mm^2 $ };
	\node[below] at (2,1.4) {\footnotesize $-2 $ };
    \end{tikzpicture}
\end{center}
Notice that, the Dynkin diagram above is different from those we encountered in \Cref{TORIC PROD TOWERS} or appearing in \Cref{sec:Algorithm}.

\smallbreak
Now we move to the computation of the Behrend numbers $\nu_{\BC[x,y]/I}$ and $\nu_{\BC[x,y]/J}$ exploiting the canonical isomorphisms of $\BA^2$-schemes just described. The computation of $\nu_{\BC[x,y]/J}$ is achieved, just as in the proof of \Cref{TEOREMA1}, via toric geometry, yielding the answer
\[
\nu_{\BC[x,y]/J}=21.
\]
In order to compute $\nu_{\BC[x,y]/I}$, we start by noticing that the morphism $\varphi\circ \widetilde{\vartheta}\colon Z\rightarrow X_I$ contracts two disjoint smooth rational projective curves over two distinct points of $X_I$ and it is an isomorphism outside such curves. Therefore, if $C\subset Z$ is the curve that dominates $\exc(X_I)$, then $\varphi\circ \widetilde{\vartheta}|_{C}$ is a birational morphism and we have
\[
\nu_{\BC[x,y]/I}=\mult_C(V((\varepsilon\circ\mu )^{-1}(I)\cdot \OO_{Z})).
\]
Again, toric geometry applied as in the proof of \Cref{TEOREMA1} gives the answer, namely
\[
\nu_{\BC[x,y]/I}=6.
\]
In \Cref{Behrenormalisation}, we shall describe a general procedure which, in particular, allows one to compute the number $\nu_{\BC[x,y]/I}$.

%%%%%%%%%%%%%%%%%%%%%%%%%%%%%%%%%%%%%%%%%%%%%%%%%%%%%%%%%%%%%%%%%%%
\subsection{Behrend number and factorisations of normal ideals}

\begin{notation}
Set $I_{h,k}=(x^h,y^k) \subset \BC[x,y]$. Then we let
$\nn_{h,k}=\overline{I}_{h,k}$ be the normalisation of $I_{h,k}$, defined as in \Cref{villareal}.
\end{notation}

\begin{example}
For istance, $\nn_{h,h}=\mm^h$ for all $h\ge 0$. One also has $\nn_{2,3} = (x^2,xy^2,y^3)$.
\end{example}
\begin{lemma}\label{powernorm}
For any $\delta\geq 0$ and $h,k > 0$, there is an identity of ideals
\[
\nn_{h,k}^\delta =\nn_{\delta h,\delta k} \subset \BC[x,y].
\]
\end{lemma}

\begin{proof}
This is trivial for $\delta=0,1$, and it follows, for higher $\delta $, combining \Cref{villareal} with the general formula
\[
\conv_{\BQ}(\delta v_1,\ldots,\delta v_s)=\conv_{\BQ} \left( \left.\begin{matrix}
\underset{i=1}{\overset{s}{\sum}}n_iv_i
\end{matrix}\ \right|\ \begin{matrix}  
\underset{i=1}{\overset{s}{\sum}}n_i=\delta,\ n_i\ge 0
\end{matrix}\right) \subset V
\]
for any choice of vectors  $v_1,\ldots,v_s\in V$ in some $\BQ$-vector space $V$.
\end{proof}

\begin{lemma}\label{missedlemma}
Let $X$ be the toric surface with fan $\Sigma$ in $\BR^2$ generated by the primitive vectors
$$\rho_0=e_1,\quad \rho_1 =\beta e_1 +\alpha e_2,\quad \rho_2=e_2.$$
Then $X$ and  $\Bl_{\nn_{\alpha,\beta}}\BA^2$ are canonically isomorphic as $\BA^2$-schemes.
\end{lemma}

\begin{proof}
The variety $X$ is, by construction covered by two charts $U_1$ and $U_2$ respectively associated to the cones $\sigma_1=\braket{\rho_0,\rho_1}$ and $\sigma_2=\braket{\rho_1,\rho_2}$. In particular, by standard toric geometry (see {\cite[\S\,3.1]{COX}}), there exist two integers $s_1,s_2\ge0$ and Laurent monomials $m_{1,1},\ldots,m_{1,s_1},m_{2,1},\ldots,m_{2,s_2}\in\BC(x,y)$ (the cone $\sigma_i$ is smooth if and only if $s_i=0$  and no Laurent monomial is needed) such that
\begin{align*}
    U_1&=\Spec \BC\left[x^{\alpha}y^{-\beta},y,m_{1,1},\ldots,m_{1,s_1}\right]  \\
    U_2&=\Spec \BC\left[x, x^{-\alpha}y^\beta  ,m_{2,1},\ldots,m_{2,s_2}\right].
\end{align*}
Let us denote by $S_1=\BC[x^{\alpha}y^{-\beta},y,m_{1,1},\ldots,m_{1,s_1}]$ and $S_2=\BC[x, x^{-\alpha}y^\beta  ,m_{2,1},\ldots,m_{2,s_2}]$ the affine rings of $U_1$ and $U_2$ and by $\varepsilon\colon X\rightarrow \BA^2$ the structure morphism. Then, if $I=(x^\alpha,y^\beta)$, we have
\[
\varepsilon\big|_{U_1}^{-1}(I)\cdot \OO_{U_1}=(y^\beta)\subset S_1,\qquad \varepsilon\big|_{U_2}^{-1}(I)\cdot \OO_{U_2}=(x^\alpha) \subset  S_2,
\]
which implies that the sheaf $ \varepsilon^{-1}(I)\cdot \OO_X$ defines a Cartier divisor on $X$.

As a consequence we have a cononical birational morphism of $\BA^2$-schemes
$\psi\colon X\rightarrow \Bl_I\BA^2$. If such a morphism is finite then, by \Cref{prop:normalisation}, it must coincide with the normalisation morphism and this would provide an isomorphism of $\Bl_I\BA^2$-schemes between $X$ and $\Bl_{\nn_{\alpha\beta}}\BA^2$ (which in particular is an isomorphism of $\BA^2$-schemes).
Finally, the morphism $\psi$ is finite because it is proper and has finite fibres. Indeed, it is an isomorphism away from the exceptional loci $\exc(X)$ and $\exc(\Bl_I\BA^2)$, and it is a dominant morphism between irreducible projective curves when restricted to the exceptional loci.
\end{proof}

\begin{theorem}\label{toricnormal}
Let $X$ be a toric surface which admits a fan $\Sigma_X$ in $ \BR^2 $ that covers the first quadrant $\BR^2_{\ge0}$ i.e.~$\Sigma_X$ is generated by the rays with primitive vectors $\rho_0=e_1,\rho_{r+1}=e_2,\rho_k=m_ke_1+n_ke_2$, for $k=1,\ldots,r$, where $m_k,n_k>0$ and $\gcd(m_k,n_k)=1$. Then, there is a canonical isomorphism 
\[
    \begin{tikzcd}
    X \arrow{r}{\sim} & \Bl_I\BA^2
    \end{tikzcd}
\]
where $I=\prod_{1\leq k\leq r} \nn_{n_k,m_k}$.
\end{theorem}

\begin{proof}
The statement follows by applying \Cref{missedlemma} and \Cref{LEMMATECH}.
\end{proof}

\begin{theorem}\label{Behrenormalisation} Let $\alpha,\beta>0$ be two positive integers. Then,
\[
\nu_{\BC[x,y]/\nn_{\alpha,\beta}}=\frac{\alpha\cdot\beta}{\gcd(\alpha,\beta)}.
\]
\end{theorem}

\begin{proof}
Suppose first that $\gcd(\alpha,\beta)=1$. Then, Euclid's algorithm provides two positive integers $h,k\in \BN$ such that
$$ k\beta-h\alpha=1.$$
Let $J$ be the ideal $$J=\nn_{\alpha,\beta}\cdot\nn_{k,h}.$$
\Cref{toricnormal} implies that there exists an open affine subset $U\subset \Bl_J\BA^2$, isomorphic to $\BA^2$, such that 
\[
U \cong
\begin{cases}
\Spec(\BC[x^{-\alpha}y^\beta,x^ky^{-h}]) & 
\textrm{if }\frac{\alpha}{\beta}<\frac{k}{h} \\
\Spec(\BC[x^{\alpha}y^{-\beta},x^{-k}y^{h}]) & 
\textrm{if }\frac{\alpha}{\beta}>\frac{k}{h}.
\end{cases}
\]
We put $\alpha /\beta>k/h$, the other case being identical. If we denote by $\varepsilon\colon \Bl_J\BA^2\rightarrow \BA^2$ the blowup map, and by $s=x^{\alpha}y^{-\beta},t=x^{-k}y^{h}$ the affine coordinates on $U$, then the restriction of the blowup map to $U$ is given by
\[
\begin{tikzcd}[row sep=tiny]
U \arrow{r}{\varepsilon|_U} & \BA^2 \\
(s,t)\arrow[mapsto]{r} & (s^ht^\beta, s^kt^\alpha).
\end{tikzcd}
\]
By construction, the intersection $\exc(U)=\exc(\varepsilon)\cap U$ consists of the two coordinate axes of $U$. In particular, given the natural map $\lambda\colon \Bl_J\BA^2\rightarrow \Bl_{\nn_{\alpha,\beta}}\BA^2$, the strict transform of $\exc(\Bl_{\nn_{\alpha,\beta}}\BA^2)$ via $\lambda|_{U}$ is the irreducible component of  $\exc(U)$ given by $C=V(t)$. This implies that 
\[
\nu_{\BC[x,y]/\mathfrak  n_{\alpha,\beta}}=\mult_C\left(\varepsilon|_{U}^{-1}(\nn_{\alpha,\beta}\right)\cdot \OO_U)=\alpha\cdot\beta,
\]
where, the first equality follows from the fact that $\lambda$ is an isomorphism away from its exceptional locus $\exc(\lambda)$ and the second follows from an easy computation.

Suppose now that $\gcd(\alpha,\beta)=\delta>1$. Then, by \Cref{powernorm,blowpowers}, we have
\[
\nu_{\BC[x,y]/\nn_{\alpha,\beta}}=\nu_{\BC[x,y]/\nn_{\alpha',\beta'}^\delta}=\delta\cdot \nu_{\BC[x,y]/\nn_{\alpha',\beta'}} = \delta \cdot \alpha'\cdot \beta',
\]
where $\alpha'=\alpha/\delta$ and $\beta'=\beta/\delta$, i.e.
\[
\nu_{\BC[x,y]/\nn_{\alpha,\beta}}=\frac{\alpha\cdot\beta}{\delta},
\]
as required.
\end{proof}

\begin{remark}
Exploiting toric geometry techniques and the isomorphism of \Cref{toricnormal}, one can generalise the computation of $\nu_{\BC[x,y]/\mathfrak n_{\alpha,\beta}}$ to an arbitrary normal monomial ideal, along the lines of the example fully worked out in \Cref{example missed}.
\end{remark}

\begin{prop}\label{conjecture}
Let $I$ be the ideal generated by the monomials
\[
x^{a_0},x^{a_1}y^{b_{n-1}},\ldots,x^{a_i}y^{b_{n-i}},\ldots,x^{a_{n-1}}y^{b_{1}},y^{b_0},
\]
where  $a_i> a_{i+1},b_i>b_{i+1}$ and we also put $a_n=b_n=0$. Suppose that $I$ is normal.

Let $0 = i_0 < \cdots < i_t= n$ be the strictly increasing sequence of positive integers such that 
\[
v_k=(a_{i_k},b_{n-i_k}), \mbox{ for }k=0,\ldots,t,
\]
are the vertices of $\partial Q_I$ (see \cref{conseguenzeVillareal}); then
\begin{equation}\label{ideal:product_normalisations}
I=\underset{k=1}{\overset{t}{\prod}}\nn_{ a_{i_{k-1}}-a_{i_{k}},b_{n-i_k}-b_{n-i_{k-1}}}.
\end{equation}
\end{prop}

\begin{proof}
Let us set 
\[
J=\underset{k=1}{\overset{t}{\prod}}\nn_{ a_{i_{k-1}}-a_{i_{k}},b_{n-i_k}-b_{n-i_{k-1}} },
\]
and let $Q_I,Q_J\subset \BQ^2$ be defined as in \Cref{villareal}. Then, the blowup $\Bl_J\BA^2$ is a normal surface, as per \Cref{rmk:normality-convexity}, and therefore the claim is equivalent to the equality
\[
Q_I=Q_J.
\]
Since, in general, we have $Q_{\nn_{\alpha,\beta}}=\conv_\BQ((\alpha,0),(0,\beta)) +\BQ^2_{\ge 0}$, we also have $Q_J=\conv_{\BQ}(A) +\BQ^2_{\ge0}$ where
\[
A=
\Set{
(a_0,0)+\underset{j\in\Delta}{\sum}[(a_{i_{j}},b_{n-i_j})-(a_{i_{j-1}},b_{n-i_{j-1}}) ] | \Delta\subset \{1,\ldots,t\}
    }.
\]
Notice that $Q_I\subset Q_J$ because $v_0\in A$ and
\[
v_k =  (a_0,0)+\sum_{j=1}^k\,\left[(a_{i_{j}},b_{n-i_j})-(a_{i_{j-1}},b_{n-i_{j-1}})\right]\in A
\]
for all $k=1,\ldots,t$.
On the other hand, the inclusion $A\subset Q_I$ is an easy consequence of the convexity of $Q_I$ and it implies $Q_J\subset Q_I$.
\end{proof}
\begin{example}
Let $I=(x^6,x^4y,x^2y^2,xy^3,y^5)$ be the same ideal as in \Cref{ex:villa}. Then, $I$ is normal and it factors as
\[
I=\nn_{1,2}\cdot\nn_{1,1}\cdot\nn_{2,1}^2.
\]
\end{example}

\begin{remark}
Thanks to the celebrated Pick's theorem on lattice polygons, we can compute the Behrend number of the ideals of the form $\nn_{\alpha,\beta}$ as well as the length of normal ideals $I$ given as in \Cref{ideal:product_normalisations}. In particular, we have
\[
\ell_{\BC[x,y]/\nn_{\alpha,\beta}}
=\frac{\alpha\beta+\alpha+\beta-\gcd(\alpha,\beta)}{2}
\]
and, for $I$  as in \Cref{ideal:product_normalisations},
\[
\ell_{\BC[x,y]/I}
=\frac{a_0+b_0+\displaystyle\sum_{k=1}^t \left[  \det\begin{pmatrix}a_{i_{j-1}}&b_{n-i_{j-1}}\\a_{i_{j}}&b_{n-i_{j}}
\end{pmatrix}-\gcd(a_{i_{j-i}}-a_{i_j},b_{n-i_j}-b_{n-i_{j-1}})\right]}{2}.
\]
\end{remark}

\begin{corollary}\label{cor:madame_bovary}
Let $I\subset\BC[x,y]$ be the ideal \eqref{ideal:product_normalisations} appearing in \Cref{conjecture}. Then $\Bl_I\BA^2$ is canonically isomorphic, as an $\BA^2$-scheme, to the toric surface whose fan is generated by the primitive vectors
$$\rho_0,\ldots,\rho_{t+1}$$
defined by $\rho_0=e_1$, $\rho_{t+1}=e_2$ and 
\[
\rho_k=\beta_k\cdot e_1+ \alpha_k \cdot e_2  \mbox{ for }k=1,\ldots,t
\]
where, if $\delta_k=\gcd(a_{i_{k-1}}-a_{i_{k}},b_{n-i_k}-b_{n-i_{k-1}})$, then
\[
\alpha_k=\frac{a_{i_{k-1}}-a_{i_{k}}}{\delta_k}  \mbox{ and } \beta_k=\frac{b_{n-i_k}-b_{n-i_{k-1}}}{\delta_k}.
\]

In particular, there is a bijective correspondence 
\[
\Set{
\Sigma \mbox{ fan in }N\otimes_{\BZ}\BR\cong\BR^2 | \begin{array}{c}
 N\cong\BZ^2,\\
  \Supp(\Sigma )=\BR^2_{\ge0} 
\end{array}
    }
\xleftrightarrow{1:1}
\Set{
I=\prod_{k=1}^t\nn_{\alpha_k,\beta_k} | 
\begin{array}{c}
 (\alpha_i,\beta_i)\not=(\alpha_j,\beta_j) \mbox{ for }i\not=j,\\
  \gcd(\alpha_i,\beta_i)=1 
\end{array}
    }.
\]
\end{corollary}

We note that \Cref{conjecture} can be interpreted as a factorisation statement, as follows.

\begin{corollary}\label{normal-factorisation}
Let $\mathfrak{N}$ be the set of normal monomial ideals in $\BC[x,y]$. Then, every $I \in \mathfrak{N}$ factors as a product of ideals in $\mathfrak{N}$,
\begin{equation}\label{eqn:factorisation3333}
I=\prod_{k=1}^t \nn_{\alpha_k,\beta_k}^{\delta_k},
\end{equation}
where $\delta_k\ge1$ and $\gcd(\alpha_k,\beta_k)=1$ for $k=1,\ldots,t$. Such factorisation is unique up to reordering the factors.
\end{corollary} 
 
A similar property cannot be expected to hold on a larger class of ideals than $\FN$. For instance, as mentioned in \Cref{rmk:normality-convexity}, if we drop the normality assumption we have, for the same ideal, two factorisations $\mm^3=\mm\cdot(x^2,y^2)$, where $(x^2,y^2)$ is not normal.

Combining \Cref{powernorm,blowpowers,toricnormal} with one another, we also obtain the following correspondence, announced as \Cref{thm:intro1241} in the introduction.

\begin{theorem}\label{bijection-irr-cpt-factorisation}
Let $I \subset \BC[x,y]$ be a normal monomial ideal of finite colength. There is a bijective correspondence 
\[
\Set{
\begin{array}{c}
 \mbox{ideals }\nn_{\alpha,\beta}^\delta \mbox{ appearing in the}\\
  \mbox{factorisation \eqref{eqn:factorisation3333} of $I$}
\end{array}
}\xleftrightarrow{1:1}\left\{
\begin{array}{c}
\mbox{irreducible}\\
  \mbox{components of }E_I\BA^2
\end{array}\right\}.
\]
In particular, if $J \subset \BC[x,y]$ is an arbitrary monomial ideal and $I = \overline J$ is its normalisation, then $E_J\BA^2$ has at most $t$ irreducible components, where $t$ is as in \Cref{eqn:factorisation}.
\end{theorem}

%%%%%%%%%%%%%%%%%%%%%%%%%%%%%%%%%%%%%%%%%%%%%%%%%%%%%%%%%%%%%%%%%%%
%%%%%%%%%%%%%%%%%%%%%%%%%%%%%%%%%%%%%%%%%%%%%%%%%%%%%%%%%%%%%%%%%%%
\section{The Behrend function of a fat point via normalisation}\label{sec:monomial}
In this section we will prove \Cref{thm:intro3} (\Cref{thm:formula:monomial} below).

Let $I \subset \BC[x,y]$ be a monomial ideal of finite colength. When $\Bl_I\BA^2$ is not normal, the computation of the Behrend number of $I$ poses some difficulties, but the main result in this section resolves them explicitly. More precisely, we shall prove a general formula  for the Behrend number
\begin{equation}\label{eqn:b_monomial_general}
    \nu_{\BC[x_1,\ldots,x_N]/I} = \sum_{D \subset E_I\BA^N} \mult_D \left(E_I\BA^2\right)
\end{equation}
of an arbitrary fat point $I \subset A= \BC[x_1,\ldots,x_N]$ supported at $0 \in \BA^N$. Such formula involves algebraic data defined through the normalisation morphism
\[
\mu_I \colon Z_I \to \Bl_I\BA^N.
\]
We note here that, when $I$ is monomial, the normalisation $Z_I$ is explicit, being equal to $\Proj \overline{A[It]}$ (cf.~\Cref{sec:normal-blowups}) and one has $Z_I=\Bl_{\overline I}\BA^2$ when $N=2$ (and $I$ is monomial), where $\overline I$ is defined in \Cref{eqn:I_bar}.

%%%%%%%%%%%%%%%%%%%%%%%%%%%%%%%%%%%%%%%%%%%%%%%%%%%%%%%%%%%%%%%%%%%
\subsection{The key example}
We present in this subsection the key example (with $N=2$) of the more general formula that will be proven just afterwards (\Cref{thm:formula:monomial}).

\begin{example}\label{rectangle}
Let $k\ge 2$ be an integer. Then, the ideal $I=(x^k,y^k) \subset \BC[x,y]$ satisfies
\[
\nu_{\BC[x,y]/I}=\ell_{\BC[x,y]/I}=k^2,
\]
 by \Cref{ex:lci}. As explained in \Cref{blowuplci}, the blowup $\Bl_{I}{\BA^2}$ identifies with the $\BA^2$-surface $V(vx^k-uy^k)\subset\BA^2\times\BP^1$. As a consequence, $\Bl_{I}{\BA^2}$ is singular in codimension 1 and hence it is not normal. Then, as observed in \cref{esnonnorm}, \Cref{villareal} implies that there is a canonical isomorphism of $\BA^2$-schemes between $\Bl_{\mm}\BA^2$ and the normalisation of $\Bl_I\BA^2$. Under this identification, the normalisation map $\mu_I \colon \Bl_{\mm}\BA^2 \to \Bl_I\BA^2$ can be realised as the restriction of the morphism
\[
\begin{tikzcd}[row sep=tiny]
\BA^2\times\BP^1 \arrow{r} & \BA^2\times\BP^1 \\
((x,y),[u:v])\arrow[mapsto]{r} & ((x,y),[u^k:v^k])
\end{tikzcd}
\]
to the subscheme $\Bl_{\mm}\BA^2 \subset \BA^2\times\BP^1$.
The exceptional locus $D=\exc(\Bl_{I}\BA^2)$ is irreducible and satisfies 
\[
\deg\left(\exc(\Bl_{\mm}\BA^2) \xrightarrow{\mu_I} D \right) = k,
\]
for such map agrees with the map $\BP^1 \to \BP^1$ sending $[u:v] \mapsto [u^k:v^k]$.
Notice that, if $\varepsilon\colon \Bl_\mm\BA^2\rightarrow \BA^2$ is the blowup map and $Y_I \subset \Bl_{\mm}\BA^2$ is the subscheme defined by the ideal sheaf $\varepsilon^{-1}(I)\cdot \OO_{\Bl_{\mm}\BA^2} \subset \OO_{\Bl_{\mm}\BA^2}$, then
\[
\mult_{\exc(\Bl_{\mm}\BA^2)} (Y_I)=k,
\]
hence $\nu_{\BC[x,y]/I} = k^2$ is also obtained as
\[
k^2 = \deg\left(\exc(\Bl_{\mm}\BA^2) \xrightarrow{\mu_I} D \right)\cdot \mult_{\exc(\Bl_{\mm}\BA^2)} (Y_I).
\]
\end{example}

%%%%%%%%%%%%%%%%%%%%%%%%%%%%%%%%%%%%%%%%%%%%%%%%%%%%%%%%%%%%%%%%%%%
\subsection{The general formula}
Let $I \subset A=\BC[x_1,\ldots,x_N]$ be the ideal defining a fat point in $\BA^N$ supported at $0 \in \BA^N$.
The normalisation morphism 
\[
\mu_I \colon  Z_I \to \Bl_I\BA^N
\]
is a finite morphism by \Cref{prop:normalisation} (and is induced by the inclusion of $A$-algebras $A[It] \into \overline{A[It]}$ in the special case where $I$ is monomial, cf.~\Cref{sec:normal-blowups}). Note that $Z_I$ is also a blowup of a fat point in $\BA^N$ supported at $0 \in \BA^N$, and $\mu_I$ is an $\BA^N$-morphism. In other words, there is a commutative diagram
\[
\begin{tikzcd}
Z_I \arrow{rr}{\mu_I} \arrow[swap]{dr}{\bar{\varepsilon}} & & \Bl_I\BA^N  \arrow{dl}{\varepsilon}\\
& \BA^N &
\end{tikzcd}
\]
where $\mu_I$ restricts to a morphism $\exc(\overline{\varepsilon}) \to \exc(\varepsilon)$ between exceptional loci. Let 
\[
D_1,\ldots,D_s \,\subset \,E_I\BA^N
\]
be the irreducible components of the exceptional locus $\exc(\varepsilon)$, each of which is taken with the reduced structure. Note that, since $E_I\BA^N$ is purely of codimension $1$, each $D_i$ has dimension $N-1$. For instance, if $N=2$, each $D_i$ is a (possibly singular) rational curve. Consider the Cartier divisor
\[
Y_I = \mu_I^{-1}(E_I\BA^N) = \bar{\varepsilon}^{-1}(V(I)) =  V(\bar{\varepsilon}^{-1}(I)\cdot \OO_{Z_I}) \subset Z_I,
\]
and notice that $Y_{I,\red} = \exc(\bar{\varepsilon})$. Hence $Y_I$ and the exceptional divisor of $Z_I$ share the same irreducible components. For each $i=1,\ldots,s$, let  
\[
V_1^{(i)},\ldots,V_{k_i}^{(i)} \,\subset \,Y_I
\]
be the irreducible components covering $D_i$, each taken with the reduced structure. The restrictions
\[
\mu_{ij} = \mu_I \big|_{V_j^{(i)}}\colon V_{j}^{(i)} \to D_i
\]
are finite dominant morphisms of varieties, and we set
\[
d_{ij} = \deg \mu_{ij}.
\]
The subscheme $Y_I \subset Z_I$ is an effective Cartier divisor, hence it is determined by an invertible ideal sheaf $\mathscr I \subset \OO_{Z_I}$. Consider the canonical section
\[
s_I \,\in\, \HH^0\left(Z_I,\mathscr I^\ast\right) \,\subset \,\BC(Z_I) = \BC(\Bl_{I}\BA^N)
\]
attached to the Cartier divisor $Y_I$.
For every pair $(i,j)$, where $i=1,\ldots,s$ and $j=1,\ldots,k_i$, we can define
\[
e_{ij} = \ord_{V_j^{(i)}}(s_I) = \mult_{V_j^{(i)}}(Y_I),
\]
namely we set $e_{ij}$ to be the order of vanishing of the rational function $s_I \in \BC(Z_I)$ along the prime $(N-1)$-cycle $V_j^{(i)}$.

We can now state and prove an explicit formula for the Behrend number of $I$.

\begin{theorem}\label{thm:formula:monomial}
Let $I \subset \BC[x_1,\ldots,x_N]$ be the ideal of a fat point $X \into \BA^N$. Then
\[
\nu_{\BC[x_1,\ldots,x_N]/I} = \sum_{i=1}^s \sum_{j=1}^{k_i} d_{ij} e_{ij}.
\]
\end{theorem}

\begin{proof}
We know by \Cref{fulton_example} applied to $Y=Z_I$ and $X=\Bl_I\BA^N$ that
\[
\ord_{D_i}(s_I) = \sum_{j=1}^{k_i} \ord_{V_j^{(i)}}(s_I)\cdot \left[\BC(V_j^{(i)}):\BC(D_i)\right] =  \sum_{j=1}^{k_i} e_{ij} d_{ij}.
\]
for every $i=1,\ldots,s$. On the other hand, we have
\[
\ord_{D_i}(s_I) = \mult_{D_i}(E_I\BA^N).
\]
The sought after relation then follows from \Cref{eqn:b_monomial_general} by summing over $i$.
\end{proof}

\begin{example}
We generalise here \Cref{rectangle}. Let $h,k\ge1$ be integers and let $\delta=\gcd(h,k)$ be their greatest common divisor. Consider the complete intersection ideal $I_{h,k}=(x^h,y^k)$, and the normalisation map
\[
\mu\colon \Bl_{\nn_{h,k}}\BA^2\rightarrow  \Bl_{I_{h,k}}\BA^2.
\]
Then, the exceptional loci $(E_{\nn_{h,k}}\BA^2)_{\red}$ and  $(E_{I_{h,k}}\BA^2)_{\red}$ are both irreducible, and $\deg \mu|_{(E_{\nn_{h,k}}\BA^2)_{\red}}=\delta$. This follows from the formulas in \Cref{powernorm,thm:formula:monomial}, namely
\[
\nn_{h,k}=\nn_{h',k'}^\delta
\]
where $k'=k/\delta$ and $h'=h/\delta$.

For instance, given $I_{4,6}=(x^4,y^6)$, one has $\nn_{4,6}=\nn_{2,3}^2$ and, as a consequence of \Cref{blowpowers}, up to isomorphism, the normalisation map has the form
\[
\mu\colon \Bl_{\nn_{2,3}}\BA^2\rightarrow \Bl_{I_{4,6}}\BA^2.
\]
Then, with the same notation as in \Cref{thm:formula:monomial}, $\nu_{\BC[x,y]/I_{4,6}}=24$ because $I_{4,6}$ is a complete intersection, $d= \deg(\mu|_{E_{\nn_{2,3}}})=2=\gcd(4,6)$ because of what we just said, and a direct computation in toric geometry (see \Cref{example missed}) shows $e=12$ where, if $\varepsilon\colon  \Bl_{\nn_{2,3}}\BA^2\rightarrow \BA^2$ denotes the blowup map, then $e=\mult_{E_{\nn_{2,3}}}( \varepsilon^{-1}(I_{4,6} )\cdot \OO_{\Bl_{\nn_{2,3}}\BA^2} )$. 
\end{example}

\begin{example}\label{samedegreemonomials}
Let $I\subset \mm\subset\BC[x,y]$ be an ideal of finite colength generated by $s+1$ monomials 
\[
m_0,\ldots,m_s \in \BC[x,y],
\]
all of degree $\delta$. Then, by \Cref{villareal}, the normalisation of $\Bl_I\BA^2$ is given by $\Bl_\mm\BA^2$. In particular, the exceptional locus $\exc(\Bl_I\BA^2)$ is irreducible. Consider the rational map
\[
\begin{tikzcd}[row sep=tiny]
\BA^2 \arrow[dashed]{r}{\varphi_I} & \BP^s \\
(a,b) \arrow[mapsto]{r} & {[}m_0(a,b):\cdots:m_s(a,b){]}
\end{tikzcd}
\]
whose indeterminacy locus is the origin $\{0\}=V(\sqrt{I})\subset \BA^2$. The fact that all the monomials have the same degree $\delta$ implies that $\varphi_I$ induces a morphism 
\[
\begin{tikzcd}[row sep=tiny]
\BP^1 \arrow{r}{\overline{\varphi}_I} & \BP^s \\
{[}a:b{]} \arrow[mapsto]{r} & {[}m_0(a,b):\cdots:m_s(a,b){]}.
\end{tikzcd}
\]
Let $X\subset \BA^2\times \BP^s$ be the Zariski closure of the graph of the map $\varphi_I$, i.e. 
\[
X=\overline{\Set{((a,b),q)\in(\BA^2\smallsetminus \{0\})\times\BP^s | q={[}m_0(a,b):\cdots:m_s(a,b){]}}}\subset \BA^2\times \BP^s.
\]
Then, by \Cref{blowaffine}, there is a canonical isomorphism of $\BA^2$-schemes $\Bl_I\BA^2\cong X$.
 We claim that, if we identify $\BP^s\cong \{ 0\}\times \BP^s\subset \BA^2\times \BP^s$, then $\Im(\overline{\varphi}_I)=\exc(\varepsilon_I)$ where $\varepsilon_I\colon X\rightarrow \BA^2$ is the structure morphism, i.e.~the restricion of the canonical projection onto $\BA^2$. Since $\Im(\overline{\varphi}_I)$ has dimension one and $\exc{(\varepsilon_I)}$ is irreducible, in order to prove $\Im(\overline{\varphi}_I)=\exc(\varepsilon_I)$, it is enough to prove $\Im(\overline{\varphi}_I)\subset\exc(\varepsilon_I)$.
 
Let $p=(0,[m_0(a,b):\cdots:m_s(a,b)])$ be a point in $\Im(\overline{\varphi}_I)$ and let $L_{a,b}=V(bx-ay)\subset \BA^2$ be a line trough the origin of $\BA^2$. Let us denote by $\varphi_{a,b}$ the restriction  $\varphi_{a,b}={\varphi_I}|_{L_{a,b}}$ and by $X_{a,b}\subset X$ the Zariski closure of its graph in $\BA^2\times \BP^s$. Then, the map $\varphi_{a,b}$ is constant and we have
\[
\varphi_{a,b}\equiv [m_0(a,b):\cdots:m_s(a,b)].
\]
As a consequence, $p\in X_{a,b }\subset X$ which proves $\Im(\overline{\varphi}_I)\subset\exc(\varepsilon_I)$.
 
Notice that, if $k>0$ is an integer and $J\subset\BC[x,y]$ is the ideal $J = (m_0^k,\ldots,m_s^k)$, then $\Im(\overline{\varphi}_I)=\Im(\overline{\varphi}_J)$
and 
\[
\deg (\overline{\varphi}_J \colon \BP^1 \to \Im(\overline{\varphi}_J)) = k\cdot \deg (\overline{\varphi}_I \colon \BP^1 \to \Im(\overline{\varphi}_I)).
\]
\end{example}

\begin{prop}\label{corHOPE} Let $h,k\in\BN$ be two positive integers and let $\delta = \gcd(h,k)$ be their greatest common divisor. Consider the ideals $I=(x^h,y^h)$, $J=(x^k,y^k)$ in $\BC[x,y]$, and their product
\[
IJ=(x^{h+k},x^ky^h,x^hy^k,y^{h+k}).
\]
Then, the Behrend number of the subscheme defined by $IJ \subset \BC[x,y]$ is
\[
\nu_{\BC[x,y]/IJ}=\delta \cdot(h+k).
\]
\end{prop}
\begin{proof}
First of all we observe that, by \Cref{villareal}, there is a canonical isomorphism of $\BA^2$-schemes between $\Bl_{\mm}\BA^2$ and the normalisation of $\Bl_{IJ}{\BA^2}$.
We have (see \Cref{section:normalization}) the following commutative diagram
\[
\begin{tikzcd}[row sep=large]
& & \Bl_I\BA^2 \arrow[bend left=20]{drr}[description]{\varepsilon_I} & & \\
\Bl_{\mm}\BA^2\arrow{rr}{\mu_{IJ}}\arrow[bend left=20]{urr}[description]{\mu_I}\arrow[bend right=20,swap]{drr}[description]{\mu_J} & & \Bl_{IJ}\BA^2\arrow{rr}{\varepsilon_{IJ}}\arrow[swap]{u}{\theta_I}\arrow{d}{\theta_J} & & \BA^2 \\
& & \Bl_J\BA^2\arrow[bend right=20,swap]{urr}[description]{\varepsilon_J} & &
\end{tikzcd}
\]
where all the maps are birational morphisms. The maps $\varepsilon_I,\varepsilon_J,\varepsilon_{IJ}$ and $\theta_I,\theta_J$ are the blowup morphisms and $\mu_I,\mu_J,\mu_{IJ}$ are the normalisation morphisms. Moreover, any composition $\Bl_\mm\BA^2\rightarrow\BA^2$ which connects $\Bl_{\mm}\BA^2$ and $\BA^2$ coincides with the blowup map $\varepsilon_{\mm}\colon \Bl_\mm\BA^2\rightarrow \BA^2$. Notice that, since $\exc(\varepsilon_\mm)$ is irreducible, also $\exc(\varepsilon_I),\exc(\varepsilon_J)$ and $\exc(\varepsilon_{IJ})$ are irreducible because they are dominated by $\exc(\varepsilon_\mm)$.

As a consequence, in order to compute (through \Cref{thm:formula:monomial}) the Behrend number of the ideal $IJ$, we have to compute only two numbers, namely
\[
e=\mult_{\exc(\varepsilon_{\mm})}\left(\varepsilon_\mm^{-1}(IJ)\cdot \OO_{\Bl_\mm\BA^2}\right)
\]
and
\[
d=\deg\left(  {\mu_{IJ}}\big|_{\exc(\varepsilon_{\mm})}\right).
\]
We start from the computation of $e$. As usual, $\Bl_\mm\BA^2$ is covered by two charts $U_0$ and $U_1$ isomorphic to $\BA^2$. In order to compute $e$, it is enough to focus on $U_0$. We introduce toric coordinates $a,b$ and the map $\varepsilon_\mm$ restricts to the map 
\[
\begin{tikzcd}[row sep=tiny]
U_0 \arrow{r}{{\varepsilon_\mm}|_{U_0}} & \BA^2 \\
(a,b) \arrow[mapsto]{r} & (ab,b).
\end{tikzcd}
\]
Hence, we have
\[
{\varepsilon_\mm}\big|_{U_0}^{-1}(IJ)\cdot \BC[a,b]=(b^{h+k})\subset \BC[a,b],
\]
and, as a consequence, $e=h+k$.

Now we move to the computation of $d$. We split this computation in two steps.\\
\textbf{Step 1:} Suppose $\delta =1$.

Since all the exceptional loci of the varieties in the above diagram are irreducible rational curves, we get a commutative diagram of fields extensions
\[
\begin{tikzcd}[row sep=large]
& & \BC(t)\arrow[hook',swap]{dll}{\varphi_I}\arrow[hook]{d}{\psi_I} \\
\BC(t) & & \BC(t)\arrow[hook',swap]{ll}{\varphi_{IJ}} \\
& & \BC(t)\arrow[hook',swap]{u}{\psi_J}\arrow[hook]{ull}{\varphi_J}
\end{tikzcd}
\]
where, up to canonical identifications, we have
\[
\begin{tikzcd}
\varphi_\bullet={\mu_\bullet}\big|_{\exc(\Bl_\mm\BA^2)}^{\ast}\colon \BC(\exc(\Bl_\bullet\BA^2))\arrow{r}{} &  \BC(\exc(\Bl_{\mm}\BA^2))
\end{tikzcd}
\]
and
\[
\begin{tikzcd}
\psi_\bullet={\theta_\bullet}\big|_{\exc(\Bl_{IJ}\BA^2)}^{\ast}\colon \BC(\exc(\Bl_{\bullet}\BA^2))\arrow{r}{} &  \BC(\exc(\Bl_{IJ}\BA^2)).
\end{tikzcd}
\]
Now, as a consequence of general field theory and of \Cref{rectangle}, we have
\begin{align*}
    [\BC(t):\psi_I(\BC(t))]\cdot [\BC(t):\varphi_{IJ}(\BC(t))]&=[\BC(t):\varphi_I(\BC(t))] \\
    [\BC(t):\psi_J(\BC(t))]\cdot [\BC(t):\varphi_{IJ}(\BC(t))]&=[\BC(t):\varphi_J(\BC(t))] \\
    [\BC(t):\varphi_I(\BC(t))]&=h \\
    [\BC(t):\varphi_J(\BC(t))]&=k,
\end{align*}

which, together with the hypothesis $\delta=\gcd(h,k)=1$, imply
\begin{align*}
[\BC(t):\psi_I(\BC(t))]&=h,\\
[\BC(t):\psi_J(\BC(t))]&=k,\\
[\BC(t):\varphi_{IJ}(\BC(t))]&=1.
\end{align*}
Thus, $d=\delta =1$.

\smallbreak
\textbf{Step 2:} Suppose $\delta >1$. Consider the positive integers $h'=h/\delta$ and $k'=k/\delta$ and the ideals $I'=(x^{h'},y^{h'})$ and $J'=(x^{k'},y^{k'})$. Let $f$, $f'$ and $g$ be the rational maps defined as follows:
\[
\begin{tikzcd}[row sep=tiny]
\BA^2 \arrow[dashed]{r}{f} & \BP^3 \\
(x,y) \arrow[mapsto]{r} & {[}x^{h+k}:x^h y^k:x^ky^h:y^{h+k}{]}
\end{tikzcd}
\]
\[
\begin{tikzcd}[row sep=tiny]
\BA^2 \arrow[dashed]{r}{f'} & \BP^3 \\
(x,y) \arrow[mapsto]{r} & {[}x^{h'+k'}:x^{h'} y^{k'}:x^{k'}y^{h'}:y^{h'+k'}{]}
\end{tikzcd}
\]
\[
\begin{tikzcd}[row sep=tiny]
\BP^3 \arrow{r}{g} & \BP^3 \\
{[}w_0:w_1:w_2:w_3{]} \arrow[mapsto]{r} & {[}w_0^\delta:w_1^\delta:w_2^\delta:w_3^\delta{]}.
\end{tikzcd}
\]
Then, a trivial computation shows that the diagram
\[
\begin{tikzcd}
\BA^2\arrow[dashed,swap]{d}{f} \arrow[dashed]{r}{f'} & \BP^3\arrow{dl}{g} \\
\BP^3
\end{tikzcd}
\]
commutes. By \Cref{blowaffine}, we have canonical isomorphisms of $\BA^2$-schemes
\begin{equation}\label{identifications}
    X\cong \Bl_{IJ}\BA^2,\qquad X'\cong \Bl_{I'J'}\BA^2
\end{equation}
where, if we denote by $\Gamma(f)$ and $\Gamma(f')$ the graphs of the rational maps $f$ and $f'$, then $X ,X'\subset \BA^2\times\BP^3$ are respectively defined as the Zariski closures of $\Gamma(f)$ and $\Gamma(f')$, i.e.~$X=\overline{\Gamma(f)}$ and $X'=\overline{\Gamma(f')}$,
and the $\BA^2$-structure morphism is given, in both cases, by the restriction of the first projection. Define now the morphism $\lambda_{IJ}$ as the restriction of the map 
\[
\id_{\BA^2}\times g \colon  \BA^2\times\BP^3\rightarrow  \BA^2\times\BP^3
\]
to $X'$. Up to the identifications \eqref{identifications} we have a commutative diagram
\[
\begin{tikzcd}[row sep=large,column sep=large]
& 
\Bl_{I'}\BA^2 \arrow{r}{\lambda_I} & 
\Bl_{I}\BA^2 \arrow[bend left=20]{dr}[description]{\varepsilon_I} & \\
\Bl_{\mm}\BA^2\arrow[bend left=20]{ur}[description]{\mu_{I'}}\arrow[bend right=20]{dr}[description]{\mu_{J'}} \arrow{r}{\mu_{I'J'}} & 
\Bl_{I'J'}\BA^2 \arrow[swap]{u}{\theta_{I'}}\arrow{r}{\lambda_{IJ}}\arrow{d}{\theta_{J'}} & \Bl_{IJ}\BA^2 \arrow{r}{\varepsilon_{IJ}} \arrow{d}{\theta_{J}} \arrow[swap]{u}{\theta_{I}} & 
\BA^2 \\
&
\Bl_{J'}\BA^2\arrow{r}{\lambda_J} & 
\Bl_{J}\BA^2\arrow[bend right=20]{ur}[description]{\varepsilon_J} &
\end{tikzcd}
\]
where
\begin{itemize}
    \item [$\circ$] the maps $\mu_{I'},\mu_{J'}$ and $\mu_{I'J'}$ are the normalisation morphisms,
    \item [$\circ$] the maps $\lambda_{I},\lambda_{J}$ are defined similarly to $\lambda_{IJ}$,
    \item [$\circ$] the compositions $\mu_{I}=\lambda_{I}\circ \mu_{I'},\mu_{J}=\lambda_{J}\circ \mu_{J'}$ and $\mu_{IJ}=\lambda_{IJ}\circ\mu_{I'J'} $ are the normalisation morphisms mentioned above,
    \item [$\circ$] any composition $\Bl_\bullet\BA^2\rightarrow \BA^2$ agrees with the blowup map $\varepsilon_\bullet \colon \Bl_\bullet\BA^2\rightarrow \BA^2$.
\end{itemize}
We know, by general theory, that
\[
\deg\left(  {\mu_{IJ}}\big|_{{\exc(\varepsilon_{\mm})}}  \right)=\deg\left(  {\mu_{I'J'}}\big|_{{\exc(\varepsilon_{\mm})}}  \right)\cdot \deg\left(  {\lambda_{IJ}}\big|_{{\exc(\varepsilon_{I'J'})}}  \right)
\]
and we also know, by \textbf{Step 1}, that
\[
\deg\left(  {\mu_{I'J'}}\big|_{{\exc(\varepsilon_{\mm})}}  \right)=1.
\]
Therefore, we have
\[
d=\deg\left(  {\lambda_{IJ}}\big|_{\exc(\varepsilon_{I'J'})}  \right).
\]
Finally, as a consequence of \Cref{samedegreemonomials}, if we call $E=(X'\cap \{0\}\times \BP^3)_{\red}$ the exceptional locus of $X'$, then we have 
\[
\deg \lambda_{IJ}\big|_E=\delta.
\]
This complete the proof.
\end{proof}

\begin{example}
For $h=k$ we find the formula
$$\nu_{\BC[x,y]/(x^h,y^h)^2}=2h^2=2\cdot \nu_{\BC[x,y]/(x^h,y^h)},$$
which may also be deduced from \Cref{blowpowers}.
For $I=(x^2,y^2)(x^3,y^3)$ and $J=(x^2,y^2)(x^6,y^6)$ we find
\[
\nu_{\BC[x,y]/I}= 5, \qquad \nu_{\BC[x,y]/J}=16.
\]
\end{example}

\begin{remark}
Let $\mu_\bullet$, for $\bullet \in \{ I,J,IJ,I',J',I'J' \}$, be defined as in the proof of \Cref{corHOPE}, and let $\vartheta_\bullet$ be the restrictions ${\vartheta_\bullet}={\mu_\bullet}|_{\exc(\Bl_\mm\BA^2)}
$. Then, up to isomorphism, the maps ${\vartheta_\bullet}$ are of the form 
\[
{\vartheta_\bullet}=\pi\circ \mathsf v_{1,d}\colon \BP^1\rightarrow \BP^s
\]
where 
$$\mathsf v_{1,d}\colon \BP^1\rightarrow\BP^d$$
is the $d$-th Veronese embedding of $\BP^1$ for some positive integer $d\ge1$ and 
$$\pi\colon \BP^d\dashrightarrow \BP^s$$
is the projection onto some coordinate projective subspace of dimension $s\le d$.
\end{remark}

The blowup $\Bl_{IJ}\BA^2$ of a product of ideals as in \Cref{corHOPE} has a peculiarity that the authors find here for the first time: its exceptional locus is (in general) not normal. For example, if $I=(x^3,y^3)$ and $J=(x^2,y^2)$, the exceptional locus $\exc(\Bl_{IJ}\BA^2)$ has two cusps as singularities. The general situation is described by the following result.

\begin{prop}
Fix positive integers $h,k>1$ and let $\delta=\gcd(h,k)$ be their greatest common divisor. Set
\[
I=(x^h,y^h)\cdot (x^k,y^k)=(x^{h+k},x^hy^k,x^ky^h,y^{h+k}) \subset \BC[x,y].
\]
Then, the exceptional locus of $\Bl_{I}\BA^2$ is an irreducible projective rational curve with two singularities of local equations of the form $\alpha^{h/\delta}-\beta^{k/\delta}=0$.
\end{prop}

\begin{proof}
As explained in \Cref{samedegreemonomials}, the image of the map
\[
\begin{tikzcd}[row sep=tiny]
\BP^1 \arrow{r}{\varphi_I} & \BP^3 \\
{[}a:b{]}\arrow[mapsto]{r} & {[}a^{h+k}:a^hb^k:a^kb^h:b^{h+k}{]}
\end{tikzcd}
\]
is isomorphic to the exceptional locus of $\Bl_I\BA^2$. The statement follows now by an easy computation.
\end{proof}

\begin{remark}\label{rmk:no-upper-bound}
Let $h,k>1$ be two natural numbers such that $\gcd(h,k)=1$ and $h^2<h+k<k^2$. Consider the ideals $I=(x^h,y^h)$ and $J=(x^k,y^k)$. Then, we have the following inequalities:
\[
\nu_{\BC[x,y]/I}< \nu_{\BC[x,y]/IJ}<\nu_{\BC[x,y]/J}.
\]
For instance, this happens for $h=2$ and $k=3$.
\end{remark}

\begin{corollary}[of \Cref{corHOPE}] Let $d_1,\ldots,d_s$ be positive integers. Given the ideals $I_k=(x^{d_k},y^{d_k})$, for $k=1,\ldots,s$, we have
\[
\nu_{\BC[x,y]/I_1\cdots I_s}=\gcd(d_1,\ldots,d_s)\cdot \sum_{k=1}^s d_k.
\]
\end{corollary}

\begin{example}
Consider the ideal $I=(x,y)\cdot(x^2,y^2)\cdot(x^3,y^3)\cdots(x^s,y^s)$ then
\[
\nu_{\BC[x,y]/I}=\underset{k=1}{\overset{s}{\sum}} k=\binom{s+1}{2}.
\]
\end{example}

%%%%%%%%%%%%%%%%%%%%%%%%%%%%%%%%%%%%%%%%%%%%%%%%%%%%%%%%%%%%%%%%%%%%
%%%%%%%%%%%%%%%%%%%%%%%%%%%%%%%%%%%%%%%%%%%%%%%%%%%%%%%%%%%%%%%%%%%%
\section{Some difficulties in dimension 3}\label{sec:3-fold-difficulties}
In \Cref{sec:monomial-normal} we proved that any normal monomial ideal $I\subset\BC [x, y]$ factors in a unique way as a product of powers of ideals of the form $\nn_{\alpha,\beta}$. Furthermore, we noticed in \Cref{bijection-irr-cpt-factorisation} that there is a bijective correspondence between the ideals that appear in such factorisation and the irreducible components of the exceptional divisor of the blowup $\Bl_I\BA^2$. This correspondence has allowed us, in numerous cases, to calculate the Behrend number of $I$. Unfortunately, as we show in the discusion below, the situation is more complicated in higher dimension.

\smallbreak
Let $I,J\subset \BC[x,y,z]$ be the curvilinear ideals defined by
\[
I=(x^2,y,z), \qquad J=(x,y^2,z),
\]
and let $\mm_{\BA^3}=(x,y,z)\subset\BC[x,y,z]$, be the maximal ideal of the origin $0\in\BA^3$.
We want to study the blowup of the ideal $IJ$. We will show that the exceptional divisor $E_{IJ}\BA^3$ decomposes into three irreducible components instead of the expected two.
 
First, we deal with the blowup $\varepsilon_I\colon B_I=\Bl_I\BA^3\rightarrow\BA^3$ and then we will move to the analysis of $\Bl_{IJ}\BA^3$. Since $I$ is a complete intersection, we have
\[
\nu_{\BC[x,y,z]/I}=\ell_{\BC[x,y,z]/I}=2,
\]
by \Cref{ex:lci}. Moreover, as a consequence of \cite[Ex.~IV-26]{GEOFSCHEME}, we have
 \[
 B_I=\Set{((x,y,z),[u_0:u_1:u_2])|\rank\begin{pmatrix}
 x&y&z^2\\u_0&u_1&u_2
 \end{pmatrix}\le 1}\subset \BA^3\times\BP^2.
 \]
 Notice that $\exc(\varepsilon_I)\cong\BP^2$ and that the threefold $B_I$ is singular along the projective line
 \[
 L=\Set{((0,0,0),[\lambda:\mu:0])|[\lambda:\mu]\in\BP^1}\subset B_I.
 \]
Let us now focus on the open neighborhood of $L$ defined by
 \[
 U=B_I\cap((\BA^3\times\{u_0\not=0\})\cup (\BA^3\times\{u_1\not=0\})).
 \]
The projection $U \to \BP^1$ sending $(p,[u_0:u_1:u_2])\mapsto [u_0:u_1]$ is an isotrivial family of singular surfaces of type $A_1$, which shows that $B_I$ is normal (this can also be deduced from a general version of \Cref{villareal}, see \cite[Prop.~1.1]{MR2029820}). Notice that the base of the family corresponds to the pencil of planes containing $V(I)$ as a closed subscheme.
 
Alternatively, similarly as we have done in \Cref{exblowcurv}, we could have built $B_I$ in the following way. Let $\varepsilon_{\mm_{\BA^3}}\colon B_{\mm_{\BA^3}}=\Bl_{\mm_{\BA^3}}\BA^3\rightarrow\BA^3$ be the blowup of $\BA^3$ at the origin. Then, a direct computation shows that 
\[
\varepsilon_{\mm_{\BA^3}}^{-1}(I)\cdot\OO_{B_{\mm_{\BA^3}}}=\mathscr H_1\cdot\mathscr H_2,
\]
where $\mathscr H_1$ is the ideal sheaf of a Cartier divisor and $\mathscr H_2$ is the ideal sheaf of a reduced point $p\in\exc(\varepsilon_{\mm_{\BA^3}})\subset B_{\mm_{\BA^3}}$. Therefore, the decomposition into irreducible components of the exceptional locus of $B_{{\mm_{\BA^3}}\cdot I}=\Bl_{{\mm_{\BA^3}}\cdot I}\BA^3$ is given by
 \[
 \exc(B_{{\mm_{\BA^3}}\cdot I})=S_1\cup S_2,
 \]
 where 
 \begin{itemize}
     \item [$\circ$] $S_1\cong \BP^2$,
     \item [$\circ$] $S_2\cong \Bl_q\BP^2$ for some $q\in\BP^2$, and agrees with the strict transform of the exceptional locus $\exc(B_{\mm_{\BA^3}})$ via the blowup map $\lambda_{\mm_{\BA^3}}\colon B_{{\mm_{\BA^3}}\cdot I}\rightarrow B_{\mm_{\BA^3}}$ induced by \Cref{LEMMATECH}, 
     \item [$\circ$] $S_1\cap S_2=\exc(S_2)\cong\BP^1$.
 \end{itemize}
 Now, consider the following canonical morphisms
\[
\begin{tikzcd}
B_{\mm_{\BA^3}} &
B_{{\mm_{\BA^3}}\cdot I}\arrow[swap]{l}{\lambda_{\mm_{\BA^3}}}\arrow{r}{\lambda_I} & 
B_I
\end{tikzcd}
\]
where the existence of $\lambda_I$ follows by the universal property of $B_I$. Since $B_{{\mm_{\BA^3}}\cdot I}$ and $B_{\mm_{\BA^3}}$ are smooth and $B_I$ is normal, all the morphisms above have connected fibres by \Cref{ZMT}. In particular, they are isomorphisms outside their exceptional loci, ${\lambda_I}{|_{S_1}}\colon S_1 \rightarrow \exc{(B_I)}$ is an isomorphism and ${\lambda_I}{|_{S_2}}\colon S_2 \rightarrow L\subset \exc{(B_I)}$ coincides with the tautological projection of the blowup of the projective plane at a point.
 
Both constructions confirm the correspondence in \Cref{bijection-irr-cpt-factorisation}. Unfortunately, such relation cannot, in general, be expected in dimension 3. To see this, consider this time the ideal $K=IJ$. Then, as above, we have canonical morphisms
\[
 \begin{tikzcd}
B_{\mm_{\BA^3}} &
B_{{\mm_{\BA^3}}\cdot K} \arrow[swap]{l}{\theta_{\mm_{\BA^3}}}\arrow{r}{\theta_K} & 
B_K
 \end{tikzcd}
\]
and we can apply again \Cref{ZMT} because $B_{{\mm_{\BA^3}}\cdot K}$ and $B_{\mm_{\BA^3}}$ are smooth and $B_K$ is normal as per \Cref{villareal}. The analogue of the description above is
\begin{itemize}
 \item [$\circ$] $\exc(B_{{\mm_{\BA^3}}\cdot K})=S_1\cup S_2\cup S_3$,
 \item [$\circ$] $S_2\cong\BP^2\cong S_3$,
 \item [$\circ$] $S_1\cong \Bl_{q_1,q_2}\BP^2$, and agrees with the strict transform of the exceptional locus $\exc(B_{\mm_{\BA^3}})$ via the blowup map $\theta_{\mm_{\BA^3}}$, 
 \item [$\circ$] $S_2\cap S_3=\emptyset$,
 \item [$\circ$] $S_i\cap S_1=L_i$ for $i=2,3$, where $L_2$ and $L_3$ are the irreducible (disjoint) components of $\exc(\theta_{\mm_{\BA^3}}|_{S_1})$.  
\end{itemize}
 Finally, one can prove that the map $\theta_K$ contracts one line to a singular point, namely the strict transform (via $\theta_{\mm_{\BA^3}}$) of the line trough the two points that correspond to $q_1,q_2$ via the isomorphism $S_1\cong\Bl_{q_1,q_2}\BP^2$ mentioned above. As a consequence, the irreducible components of $\exc(B_{K})$ are:
 \[
 \theta_K(S_1)\cong\BP^1\times\BP^1,\quad \theta_K(S_2)\cong \BP^2\cong\theta_K(S_3).
 \]
One can also find, via a direct computation, the Behrend number of the ideal $K$, which is
\[
\nu_{\BC[x,y,z]/K}=8.
\]

\medskip
The above discussion shows that, even for towers, generalising to dimension $3$ the constructions and algorithms carried out in \Cref{sec:towers,sec:Algorithm} is a nontrivial task, that we leave for future research.

\bibliographystyle{amsplain-nodash}
\bibliography{bib}

\bigskip

\medskip
\noindent
\emph{Michele Graffeo}, \texttt{mgraffeo@sissa.it} \\
\textsc{Scuola Internazionale Superiore di Studi Avanzati (SISSA), Via Bonomea 265, 34136 Trieste, Italy}

\medskip
\noindent
\emph{Andrea T. Ricolfi}, \texttt{andreatobia.ricolfi@unibo.it} \\
\textsc{Università di Bologna, Piazza di Porta S.~Donato 5, 40126 Bologna, Italy}

\end{document}